\documentclass[11pt]{article}
\usepackage{mathrsfs}
\usepackage[T1]{fontenc}
\usepackage{latexsym,amssymb,amsmath,amsfonts,amsthm}
\usepackage{graphicx,subfigure}
\usepackage[utf8]{inputenc}
\usepackage{authblk}
\usepackage{mathrsfs}

\newcommand{\R}{{\mathbb R}}

\newcommand{\C}{{\mathbb C}}

\newcommand{\be}{\begin{eqnarray}}
\newcommand{\ben}{\begin{eqnarray*}}
\newcommand{\en}{\end{eqnarray}}
\newcommand{\enn}{\end{eqnarray*}}
\newcommand{\ba}{\backslash}
\newcommand{\pa}{\partial}

\newcommand{\ov}{\overline}

\newcommand{\g}{\gamma}

\newcommand{\eps}{\epsilon}

\newcommand{\Om}{\Omega}
\newcommand{\om}{\omega}

\newcommand{\al}{\alpha}
\newcommand{\la}{\lambda}

\newcommand{\hth}{\hat{\theta}}
\newcommand{\hx}{\hat{x}}

\newtheorem{theorem}{Theorem}[section]
\newtheorem{lemma}[theorem]{Lemma}

\newtheorem{assumption}[theorem]{Assumption}

\begin{document}
\begin{titlepage}
\title{A novel sampling method for multiple multiscale targets from scattering amplitudes at a fixed frequency}
\author{Xiaodong Liu\\
Institute of Applied Mathematics, Academy of Mathematics and Systems Science, Chinese Academy of Sciences, 100190 Beijing, China.\\
xdliu@amt.ac.cn (XL)}
\date{}
\end{titlepage}
\maketitle

\begin{abstract}
A sampling method by using scattering amplitude is proposed for shape and location reconstruction in inverse acoustic scattering problems.
Only matrix multiplication is involved in the computation, thus the novel sampling method is very easy and simple to implement.
With the help of the factorization of the far field operator, we establish an inf-criterion for characterization of underlying scatterers.
This result is then used to give a lower bound of the proposed indicator functional for sampling points inside the scatterers.
While for the sampling points outside the scatterers, we show that the indicator functional decays like the bessel functions as the sampling
point goes away from the boundary of the scatterers.
We also show that the proposed indicator functional continuously dependents on the scattering amplitude, this further implies that the novel sampling
method is extremely stable with respect to errors in the data.
Different to the classical sampling method such as the linear sampling method or the factorization method, from the numerical point of view, the novel
indicator takes its maximum near the boundary of the underlying target and decays like the bessel functions as the sampling points go away from the boundary.
The numerical simulations also show that the proposed sampling method can deal with multiple multiscale case, even the different components are close to
each other.

\vspace{.2in} {\bf Keywords:}
Acoustic scattering, scattering amplitude, bessel functions, multiple, multiscale.

\vspace{.2in} {\bf AMS subject classifications:}
35P25, 35Q30, 45Q05, 78A46

\end{abstract}

\section{Introduction}
\setcounter{equation}{0}

In the last twenty years, sampling methods for shape reconstruction in inverse scattering problems have attracted a lot of interest.
Typical examples include the Linear Sampling Method by Colton and Kirsch \cite{ColtonKirsch}, the Singular Sources Method by Potthast \cite{Potthast2000}
and the Factorization Method by Kirsch \cite{Kirsch98}.
The basic idea is to design an indicator which is big inside the underlying scatterer and relatively small outside.
We refer to the monographs of Cakoni and Colton \cite{CC2014}, Colton and Kress \cite{CK} and Kirsch and Grinberg \cite{KirschGrinberg} for a comprehensive understanding.
We also refer to Liu and Zhang \cite{LiuZhang2015} for a recent progress on the Factorization Method.
Recently, a type of direct sampling methods are proposed for inverse scattering problems,
e.g., Orthogonality Sampling by Potthast \cite{Potthast2010}, Direct Sampling Method by Ito et.al. \cite{ItoJinZou}, Single-shot Method by Li et.al. \cite{LiLiuZou},
Reverse Time Migration by Chen et.al. \cite{CCHuang}.
These direct sampling methods inherit many advantages of the classical ones, e.g., they are independent of any a priori information on the geometry and physical properties of
the unknown objects. The main feature of these direct sampling methods is that only inner product of the measurements with some suitably chosen functions is involved in the
computation of the indicator, thus is robust to noises and computationally faster than the classical sampling methods.
However, the theoretical foundation of the direct sampling methods is far less well developed than for the classical sampling methods.
In particular, there are no theoretical analysis of the indicators for the sampling points inside the scatterers.
In this paper, we propose a new direct sampling method for inverse acoustic scattering problems by using the scattering amplitudes.
We will study the behavior of our new indicator for the sampling points both outside and inside the scatterer.

We begin with the formulations of the acoustic scattering problems.  Let $k=\om/c>0$ be the wave number of
a time harmonic wave where $\om>0$ and $c>0$ denote the frequency and sound
speed, respectively. Let $\Om\subset\R^n (n=2,\, 3)$ be an open and bounded domain with
Lipschitz-boundary $\pa\Om$ such that the exterior $\R^n\ba\ov{\Om}$ is connected.
Furthermore, let the incident field $u^i$ be a plane wave of the form
\be\label{incidenwave}
u^i(x)\ =\ u^i(x,\hth) = e^{ikx\cdot \hth},\quad x\in\R^n\,,
\en
where $\hth\in S^{n-1}$ denotes the direction of the incident wave and $S^{n-1}:=\{x\in\R^n:|x|=1\}$ is the unit sphere in $\R^n$.
Then the scattering problem for the inhomogeneous medium is to find the total field $u=u^i+u^s$ such that
\be
\label{HemEqumedium}\Delta u + k^2 (1+q)u = 0\quad \mbox{in }\R^n,\\
\label{Srcmedium}\lim_{r:=|x|\rightarrow\infty}r^{\frac{n-1}{2}}\left(\frac{\pa u^{s}}{\pa r}-iku^{s}\right) =\,0
\en
where $q\in L^{\infty}(\R^n)$ such that $\Im (q)\geq 0$ and $q=0$ in $\R^n\ba\ov{\Om}$,
the Sommerfeld radiating condition \eqref{Srcmedium} holds uniformly with respect to all directions $\hx:=x/|x|\in S^{n-1}$.
If the scatterer $\Om$ is impenetrable, the direct scattering is to find the total field $u=u^i+u^s$ such that
\be
\label{HemEquobstacle}\Delta u + k^2 u = 0\quad \mbox{in }\R^n\ba\ov{\Om},\\
\label{Bc}\mathcal{B}(u) = 0\quad\mbox{on }\pa\Om,\\
\label{Srcobstacle}\lim_{r:=|x|\rightarrow\infty}r^{\frac{n-1}{2}}\left(\frac{\pa u^{s}}{\pa r}-iku^{s}\right) =\,0,
\en
where $\mathcal{B}$ denotes one of the following three boundary conditions:
\ben
(1)\,\mathcal{B}(u):=u\,\,\mbox{on}\ \pa \Om;\qquad
(2)\,\mathcal{B}(u):=\frac{\pa u}{\pa\nu}\,\,\mbox{on}\ \pa \Om;\qquad
(3)\,\mathcal{B}(u):=\frac{\pa u}{\pa\nu}+\la u\,\,\mbox{on}\ \pa \Om
\enn
corresponding, respectively, to the case when the scatterer $\Omega$ is sound-soft, sound-hard, and of impedance type.
Here, $\nu$ is the unit outward normal to $\pa\Om$ and $\la\in L^{\infty}(\pa \Om)$ is the (complex valued) impedance function such that $\Im(\la)\geq0$ almost everywhere on $\pa \Om$.
Uniqueness of the scattering problems \eqref{HemEquobstacle}--\eqref{Srcmedium} and \eqref{HemEquobstacle}--\eqref{Srcobstacle}
can be shown with the help of Green's theorem, Rellich's lemma and unique continuation principle, see e.g., \cite{CK}. The proof of existence can be done
by variational approaches (cf. \cite{CK, Mclean} for the Dirichlet boundary condition and \cite{CC2014, LiuZhangHu} for other boundary conditions) or by integral equation methods
(cf.\cite{CK, KirschGrinberg,LiuZhangSiam,LiuZhang2012}).


Every radiating solution of the Helmholtz equation has the following asymptotic
behavior at infinity:
\be\label{0asyrep}
u^s(x,\hth)
=\frac{e^{i\frac{\pi}{4}}}{\sqrt{8k\pi}}\left(e^{-i\frac{\pi}{4}}\sqrt{\frac{k}{2\pi}}\right)^{n-2}\frac{e^{ikr}}{r^{\frac{n-1}{2}}}\left\{u^{\infty}(\hat{x},\hth)+\mathcal{O}\left(\frac{1}{r}\right)\right\}\quad\mbox{as }\,r:=|x|\rightarrow\infty
\en
uniformly with respect to all directions $\hx:=x/|x|\in S^{n-1}$, see, e.g., \cite{KirschGrinberg}.
The complex valued function $u^\infty=u^\infty(\hx,\hth)$ defined on the unit sphere $S^{n-1}$
is known as the scattering amplitude or far-field pattern with $\hat{x}\in S^{n-1}$ denoting the observation direction.
Then, the {\em inverse problem} we consider in this paper is to determine $\Om$ from a knowledge of the scattering amplitude $u^{\infty}(\hat{x},\hth)$ for
$\hat{x},\hth\in S^{n-1}$.
It is well known that the scatterer $\Om$ can be uniquely determined by the scattering amplitude $u^\infty(\hx,\hth)$ for all $\hx, \hth\in S^{n-1}$ \cite{CK}.
What we interested, in this paper, is to design a direct sampling approach for shape reconstruction by using the far field measurements.

The indicator functional which will be used for inverse acoustic scattering problems is given as follows,
\be\label{IndicatorLiu}
I_{new}(z):=\Big|\int_{S^{n-1}}e^{-ik\hth\cdot z}\int_{S^{n-1}}u^\infty(\hx,\hth)e^{ik\hx\cdot z}ds(\hx)ds(\hth)\Big|,\quad z\in\R^n.
\en
In practice, the indicator functional is given
by the following discrete form
\ben\label{IndicatorLiudiscrete}
I_{new}(z):=
\left|\left(
    \begin{array}{cccc}
      e^{-ik\hth_1\cdot z}, e^{-ik\hth_2\cdot z}, \cdots, e^{-ik\hth_N\cdot z} \\
    \end{array}
  \right)
  \left(
    \begin{array}{cccc}
      u_{1,1}^\infty\, u_{1,2}^\infty\, \cdots\, u_{1,N}^\infty \\
      u_{2,1}^\infty\, u_{2,2}^\infty\, \cdots\, u_{2,N}^\infty \\
      \vdots\,\quad \vdots\,\quad \ddots\,\quad \vdots \\
      u_{N,1}^\infty\, u_{N,2}^\infty\, \cdots\, u_{N,N}^\infty \\
      \end{array}
  \right)\left(
           \begin{array}{c}
             e^{ik\hx_1\cdot z} \\
             e^{ik\hx_2\cdot z} \\
             \vdots \\
             e^{ik\hx_N\cdot z} \\
           \end{array}
         \right)\right|,\quad z\in\R^n,
\enn
where $u^{\infty}_{i,j}=u^\infty(\hx_j,\hth_i)$ for $1\leq i, j\leq N$ corresponding to $N$ observation directions $\hx_j$ and $N$ incident directions $\hth_i$.
Clearly, only matrix multiplication is involved in the computational implementation.

This paper is organized as follows.
The theoretical foundation of the proposed reconstruction scheme will be established in the next section.
With the help of the inf-criterion characterization obtained by using the factorization of the far field operator, we show a lower bound of
the indicator $I_{new}$ for the sampling points inside the scatterers.
The decay behavior of $I_{new}$ will then be studied for sampling points outside the scatterers.
A stability statement will also be established to reflect the important feature of the reconstruction scheme under consideration.
With the help of well known properties of the scattering amplitudes and the corresponding far field operator, some connections with other sampling methods
are then established.
Some numerical simulations in two dimensions will be presented in Section~\ref{sec3} to indicate the new sampling method is computationally efficient and robust
with respect to data noise.

\section{Theoretical foundation of the proposed sampling method}
\label{sec2}
\setcounter{equation}{0}

The aim of this section is to establish the mathematical basis of our sampling method.
First, we recall the far field operator $F:L^2(S^{n-1})\rightarrow L^2(S^{n-1})$ defined by
\be\label{ffoperator}
(Fg)(\hx)=\int_{S^{n-1}}u^{\infty}(\hx,\hth)\,g(\hth)\,ds(\hth)\,,\quad \hx\in S^{n-1}\,.
\en
The far field operator $F$ plays an essential role in the investigations of the inverse scattering problems, we refer to the monographs of Colton and Kress\cite{CK} and Kirsch and
Grinberg\cite{KirschGrinberg} for a survey on the state of the art of its properties and applications.

For all sampling point $z\in \R^n$, define a test function $\phi_z\in L^2(S^{n-1})$ by
\be
\label{phiz}\phi_z(\vartheta):=e^{-ikz\cdot\vartheta},\quad \vartheta\in S^{n-1}.
\en
Then we may rewrite our indicator $I_{new}$ given by \eqref{IndicatorLiu} in a very simple form
\be\label{IndLiunewform}
I_{new}(z):=|( F\phi_z, \phi_z)|,\quad z\in\R^n.
\en
Here and throughout this paper we denote by $(\cdot, \cdot)$ the inner product of $L^2(S^{n-1})$.

\subsection{Lower bound estimate of $I_{new}(z)$ for $z\in\Om$}
For all $z\in \R^n$, define $A_z\subset L^2(S^{n-1})$ by
\ben
A_z:=\{\psi\in L^2(S^{n-1}):\,(\psi, \phi_z)=1\},
\enn
where again $\phi_z$ is defined by \eqref{phiz}.
We recall that the fundamental solution $\Phi(x,z), x\in \R^n, x\neq z,$ of the Helmholtz equation is given by
\be\label{Phi}
\Phi(x,y):=\left\{
              \begin{array}{ll}
                \frac{ik}{4\pi}h^{(1)}_0(k|x-y|)=\frac{e^{ik|x-y|}}{4\pi|x-y|}, & n=3; \\
                \frac{i}{4}H^{(1)}_0(k|x-y|), & n=2,
              \end{array}
            \right.
\en
where $h^{(1)}_0$ and $H^{(1)}_0$ are, respectively, spherical Hankel function and Hankel function of the first kind and order zero.

Considering the case of scattering by impenetrable scatterers, we have the inf-criterion for sampling points $z$ which is both necessary and sufficient.

\begin{lemma}\label{infcriobstacle}{\rm (See Theorem 1.20 and 2.8 of \cite{KirschGrinberg})}\,
Consider the inverse scattering by impenetrable scatterers. We assume that $k^2$ is not an eigenvalue of $-\Delta$ in $\Om$ with respect to the boundary condition under consideration.
Then $z\in\Om$ if, and only if,
\ben
\inf|\{|(F\psi, \psi)|:\,\psi\in A_z\}>0.
\enn
Furthermore, for $z\in\Om$ we have the estimate
\be\label{estimate}
\inf|\{|(F\psi, \psi)|:\,\psi\in A_z\}\geq \frac{c}{\|\Phi(\cdot,z)\|^2_{H^{1/2}(\pa\Om)}},
\en
for some constant $c>0$ which is independent of $z$.
\end{lemma}

Turning now the case of scattering by an inhomogeneous medium, the analogous results of Lemma \ref{infcriobstacle} is, to our knowledge, still not established. So we proceed by
studying the corresponding inf-criterion for inhomogeneous medium.
We first make the following general assumptions on the contrast function $q$.

\begin{assumption}\label{qassumption}
Let $q\in L^{\infty}(\R^n)$ satisfy
\begin{enumerate}
  \item $q=0$ in $\R^n\ba\ov{\Om}$.
  \item $\Im(q)\geq 0$ and there exists $c_1>0$ with $1+\Re(q)\geq c_1$ for almost all $x\in\Om$.
  \item $|q|$ is locally bounded below, i.e., for every compact subset $D\subset\Om$ there exists $c_2>0$ (depending on $D$) such that $|q|\geq c_2$ for almost all $x\in\Om$.
  \item There exists $t\in [0, \pi]$ and $c_3>0$ such that $\Re[e^{-it}q(x)]\geq c_3|q|$ for almost all $x\in\Om$.
\end{enumerate}
\end{assumption}

For simplicity, we denote by $(\cdot,\cdot)_{\Om}$ the inner product of $L^2(\Om)$. A confusion with the inner product $(\cdot,\cdot)$ of $L^2(S^{n-1})$ is not expected.
We define the weighted space $L^2(\Om, |q|dx)$ as the completion of $L^2(\Om)$ with respect to the norm
$(\phi,\psi)_{L^2(\Om, |q|dx)}=(\phi,|q|\psi)_{\Om}$.
We say that $k^2$ is an interior transmission eigenvalue if there exists $(u,w)\in H^1_{0}(\Om)\times L^2(\Om, |q|dx)$ with  $(u,w)\neq (0,0)$ and a sequence $\{w_j\}$ in
$H^2(\Om)$ with $w_j\rightarrow w$ in $L^2(\Om, |q|dx)$ and $\Delta w_j+k^2w_j=0$ in $\Om$ and
\ben
\int_{\Om}[\nabla u\cdot\nabla\psi-k^2(1+q)u\psi]dx=k^2\int_{\Om}qw\psi\,dx \quad \mbox{for\,all}\,\psi\in H^1(\Om).
\enn
We list now some of the results on the factorization of the far field operator for inhomogeneous medium \cite{KirschGrinberg}.

\begin{lemma}\label{fmedium}{\rm (See Theorems 4.5, 4.6 and 4.8 of \cite{KirschGrinberg})}\,
Let Assumption \ref{qassumption} hold and $F:L^2(S^{n-1})\rightarrow L^2(S^{n-1})$ be the far field operator defined by \eqref{ffoperator}. Then
\begin{enumerate}
  \item We have the following factorization of  the far field operator $F$
        \ben
        F=H^{\ast}TH.
        \enn
        The operator $H:L^2(S^{n-1})\rightarrow L^2(\Om)$ is defined by
        \ben
        (Hg)(x)=\sqrt{|q(x)|}\int_{S^{n-1}}g(\theta)e^{ikx\cdot\theta}ds(\theta),\quad x\in \Om.
        \enn
        The operator $T:L^2(\Om)\rightarrow L^2(\Om)$ is defined by
        \ben
        Tf=k^2(signq)(f+\sqrt{|q|}v_{\Om}),\quad f\in L^2(\Om),
        \enn
        where $signq:=q/|q|$ and $v\in H^1_{loc}(\R^n)$ is the radiating solution of $\Delta v+k^2(1+q)v=-k^2(signq)f$ in $\R^n$.
  \item Let $B_{\eps}(z)\subset \Om$ be some closed ball with center $z$ and radius $\eps>0$ which is completely contained in $\Om$. Choose a function $\chi\in C^{\infty}(\R)$
        with $\chi(t)=1$ for $|t|\geq \eps$ and $\chi(t)=0$ for $|t|\leq\eps/2$ and define $w_0\in C^{\infty}(\R^3)$ by $w_0=\chi(|x-z|)\Phi(x,z)$ in $\R^n$. Now we set
        \be\label{w}
        w=\left\{
              \begin{array}{ll}
                -(\Delta w_0+k^2w_0)/\sqrt{|q|}, & \hbox{in $B_{\eps}(z)$;} \\
                0, & \hbox{in $\Om\ba\ov{B_{\eps}(z)}$.}
              \end{array}
            \right.
        \en
        Then $w\in L^2(\Om)$ and $\phi_z=H^{\ast}w$ where $\phi_z$ is again defined by \eqref{phiz}.
  \item Define $T_0:L^2(\Om)\rightarrow L^(\Om)$ by $T_0f=k^2(signq)f$ for $f\in L^2(\Om)$. Then $T-T_0$ is compact and $\Re[e^{-it}T_0]$ is coercive, i.e.,
        \be\label{T0coercive}
        \Re[e^{-it}(T_0f,f)_{\Om}]\geq c\|f\|^2_{L^2(\Om)},\quad f\in L^2(\Om)
        \en
        for some constant $c>0$.
  \item Assume that $k^2$ is not an interior transmission eigenvalue. Then
       \be
       \Im(Tf, f)_{\Om}>0\quad \forall\,\,f\in\ov{range(H)}, \,f\neq 0.
       \en
\end{enumerate}
\end{lemma}

We will provide the inf-criterion for inverse scattering by inhomogeneous medium with the help of the following lemma.

\begin{lemma}\label{Tcoercive}
In addition to Assumption \ref{qassumption} assume that $k^2$ is not an interior transmission eigenvalue.
Then the middle operator $T:L^2(\Om)\rightarrow L^2(\Om)$ satisfy the coercivity condition, i.e., there exists a constant $c>0$ with
\be\label{Tproperty}
|(Tf, f)_{\Om}|\geq c\|f\|^2_{L^2(\Om)},\quad\forall\,\, f\in range(H).
\en
\end{lemma}
\begin{proof}
If there exists no constant $c>0$ with \eqref{Tproperty} then there exists a sequence $\{f_j\}$ in $range(H)$ such that
\be\label{fj1}
\|f_j\|_{L^2(\Om)}=1\quad\mbox{and}\quad (Tf_j, f_j)_{\Om}\rightarrow 0,\quad\mbox{as} \quad j\rightarrow\infty.
\en
Since the unit ball in $L^2(\Om)$ is weakly compact there exists a subsequence which converge weakly to some $f\in\ov{range(H)}$.
We denote this subsequence again by $\{f_j\}$. By Lemma \ref{fmedium}(3), the difference operator $T-T_0$ is compact,
which implies that
\be\label{TT0}
(T-T_0)f_j\rightarrow (T-T_0)f\quad \mbox{in}\, L^2(\Om),
\en
and thus also
\be\label{TT02}
\big((T-T_0)(f-f_j), f_j\big)_{\Om}\rightarrow 0\quad\mbox{as} \quad j\rightarrow\infty.
\en
By linearity,
\ben
(Tf, f_j)_{\Om}=(Tf_j, f_j)_{\Om}+\big((T-T_0)(f-f_j), f_j\big)_{\Om}+\big(T_0(f-f_j), f\big)_{\Om}-\big(T_0(f-f_j), f-f_j\big)_{\Om}.
\enn
The left hand side converges to $(Tf, f)_{\Om}$, the first three terms on the right hand side converge to zero.
By the definition of $T_0$ and the assumption that $\Im(q)\geq 0$, we deduce that $\Im (\big(T_0(f-f_j), f-f_j\big)_{\Om})\geq 0$.
This fact combines the forth result of Lemma \ref{fmedium} implies that $f=0$.
Therefore, by using the third result of Lemma \ref{fmedium}, we have
\ben
c\|f_j\|^2_{L^2(\Om)}
&\leq& \Re[e^{-it}(T_0f_j,f_j)_{\Om}]\leq |e^{-it}(T_0f_j,f_j)_{\Om}|\cr
&=&|(T_0f_j,f_j)_{\Om}|\leq |\big((T_0-T)f_j,f_j\big)_{\Om}|+|(Tf_j,f_j)_{\Om}|
\enn
which tends to zero as $j\rightarrow \infty$. Therefore, also $f_j\rightarrow 0$ which contradicts to the assumption \eqref{fj1}, i.e., $\|f_j\|_{L^2(\Om)}=1$.
\end{proof}

Now the analogous result of Lemma \ref{infcriobstacle} for inhomogeneous medium, by using Theorem 1.16 from \cite{KirschGrinberg} and the previous two Lemmas
\ref{fmedium} and \ref{Tcoercive}, can be formulated as the following lemma.

\begin{lemma}\label{infcrimedium}
Consider the inverse scattering by inhomogeneous medium.
In addition to Assumption \ref{qassumption} assume that $k^2$ is not an interior transmission eigenvalue.
Then $z\in\Om$ if, and only if,
\ben
\inf|\{|(F\psi, \psi)|:\,\psi\in A_z\}>0.
\enn
Furthermore, for $z\in\Om$ we have the estimate
\be\label{estimatemedium}
\inf|\{|(F\psi, \psi)|:\,\psi\in A_z\}\geq \frac{c}{\|w(\cdot,z)\|^2_{L^{2}(\Om)}},
\en
for some constant $c>0$ which is independent of $z$. Here $w$ is defined by \eqref{w}.
\end{lemma}

Lemmas \ref{infcriobstacle} and \ref{infcrimedium} are satisfactory from the theoretical point of view.
However, there is a major drawback with respect to the computational point of view since it is very time consuming to solve a minimization
problem for every sampling point $z$.
What is interesting is that the estimates \eqref{estimate} and \eqref{estimatemedium} given, respectively, in Lemmas \ref{infcriobstacle} and \ref{infcrimedium}
provide some insight to our indicator $I_{new}(z)$ for sampling points $z\in\Om$.
Actually, a straightforward calculation shows that
\be\label{gama}
\g:=(\phi_z, \phi_z)=\int_{S^{n-1}}|\phi_z|^2ds=\int_{S^{n-1}}1ds=\left\{
                                                                    \begin{array}{ll}
                                                                      2\pi, & \hbox{$n=2$;} \\
                                                                      4\pi, & \hbox{$n=3$.}
                                                                    \end{array}
                                                                  \right.
\en
This implies $\psi_z:=\phi_z/\g\in A_z$, and therefore, by noting the linearity of the far field operator $F$ and using the estimate \eqref{estimate} or \ref{estimatemedium}, we have
\ben
I_{new}(z)
&=&|( F\phi_z, \phi_z)|\cr
&=&\g|( F\psi_z, \phi_z)|\cr
&\geq&\g\inf|\{|(F\psi, \psi)|:\,\psi\in A_z\}\cr
&\geq& \frac{c\g}{M_z},\quad z\in\Om
\enn
for some constant $c>0$ which is independent of $z$. Here $M_z$ is defined by
\be\label{Mz}
M_z:=\left\{
             \begin{array}{ll}
               \|\Phi(\cdot,z)\|^{2}_{H^{1/2}(\pa\Om)}, & \hbox{for the scattering by impenetrable scatterers;} \\
               \|w(\cdot,z)\|^2_{L^{2}(\Om)}, & \hbox{for the scattering by inhomogeneous medium,}
             \end{array}
           \right.
\en
with $w$ defined by \eqref{w}.
We formulate this result as the main result of this subsection.

\begin{theorem}\label{Ilowbound}
Under the assumptions of Lemmas \ref{infcriobstacle} and \ref{infcrimedium}, we have
\be
I_{new}(z)\geq \frac{c\g}{M_z},\quad z\in\Om
\en
for some constant $c>0$ which is independent of $z$. Here, $M_z$ is defined by \eqref{Mz}, $\g$ is defined by \eqref{gama}.
\end{theorem}

Theorem \ref{Ilowbound} provides a lower bound of the indicator $I_{new}$ for sampling points in $\Om$.
We finally remark that the assumptions in Theorem \ref{Ilowbound} are only used for theoretical analysis.

\subsection{Resolution analysis for the sampling points go away from the boundary $\pa\Om$}
This subsection is devoted to the study of the behavior of the indicator $I_{new}$ for sampling points outside the scatterer.

Let $Y_\alpha^\beta(\cdot)$ for $\alpha\in\mathbb{N}\cup\{0\}$ and $\beta=-\alpha,\ldots, \alpha$ ($\beta=\pm\alpha$ in two-dimensional case) be the
spherical harmonics which form a complete orthonormal system in $L^2(\mathbb{S}^{n-1})$ (cf. \cite{CK}).
In particular, we recall the spherical harmonics $Y_\alpha^\beta(\hat{x})$ of order $\al=0,1$, for $\hat{x}=(\hat{x}^l)_{l=1}^n\in\mathbb{S}^{n-1}$. In the three-dimensional case,
\ben
Y^0_0(\hat{x})=\sqrt{\frac{1}{4\pi}},\
Y^{-1}_1(\hat{x})=\sqrt{\frac{3}{8\pi}}(\hat{x}^1-i\hat{x}^2),\
Y^{0}_1(\hat{x})=\sqrt{\frac{3}{4\pi}}\hat{x}^3,\
Y^{1}_1(\hat{x})=\sqrt{\frac{3}{8\pi}}(\hat{x}^1+i\hat{x}^2).
\enn
In the two-dimensional case, $Y_1^0$ does not exist and
\ben
Y^0_0(\hat{x})=\sqrt{\frac{1}{2\pi}},\quad
Y^{-1}_1(\hat{x})=\sqrt{\frac{1}{2\pi}}(\hat{x}^1-i\hat{x}^2),\quad
Y^{1}_1(\hat{x})=\sqrt{\frac{1}{2\pi}}(\hat{x}^1+i\hat{x}^2).
\enn

\begin{lemma}\label{lemma_funkheck}
\ben\label{FunkHeckgeneralize}
\int_{S^{n-1}}e^{-ik\hx\cdot p}ds(\hx) &=& \mu_{0}f_{0}(k|p|),\quad p\in\R^n, \\
\int_{S^{n-1}}\hx e^{-ik\hx\cdot p}ds(\hx) &=& \left\{
                                                 \begin{array}{ll}
                                                   0, & \hbox{$p=0$;} \\
                                                   \mu_{1}\hat{p}f_1(k|p|), & \hbox{$p\in\R^n, p\neq 0$,}
                                                 \end{array}
                                               \right.
\enn
where
\ben
\hat{p}=p/|p|,\quad
\mu_{\alpha}=\left\{
               \begin{array}{ll}
                 \frac{2\pi}{i^{\alpha}}, & \hbox{$n=2$;} \\
                 \frac{4\pi}{i^{\alpha}}, & \hbox{$n=3$}
               \end{array}
             \right.
\qquad\mbox{and}\qquad
f_{\alpha}(t)=\left\{
               \begin{array}{ll}
                 J_{\alpha}(t), & \hbox{$n=2$;} \\
                 j_{\alpha}(t), & \hbox{$n=3$}
               \end{array}
             \right.
\enn
with $J_{\alpha}$ and $j_{\alpha}$ being the Bessel functions and spherical Bessel functions of order $\alpha$, respectively.
\end{lemma}
\begin{proof}
The lemma follows by the well known Funk-Hecke formula
\ben\label{FunkHeckformula}
\int_{S^{n-1}}e^{-ikz\cdot\hx }Y_\alpha^\beta(\hx)ds(\hx)=\mu_{\alpha}f_{\alpha}(k|z|)Y_\alpha^\beta(\hat{z}),
\enn
and the fact that
\ben
\hx = (\hat{x}^1, \hat{x}^2, \cdots, \hat{x}^n)
    = \left\{
        \begin{array}{ll}
          \sqrt{\pi/2}\Big(Y_1^1(\hx)+Y_1^{-1}(\hx),\,\,iY_1^{-1}(\hx)-iY_1^{1}(\hx)\Big), & \hbox{$n=2$;} \\
          \sqrt{2\pi/3}\Big(Y_1^1(\hx)+Y_1^{-1}(\hx),\,\, iY_1^{-1}(\hx)-iY_1^{1}(\hx),\,\, \sqrt{2}Y_1^{0}(\hx)\Big), & \hbox{$n=3$.}
        \end{array}
      \right.
\enn
\end{proof}

\begin{figure}[htbp]
  \centering
  \subfigure[\textbf{Bessel function $J_0$}]{
    \includegraphics[width=3in]{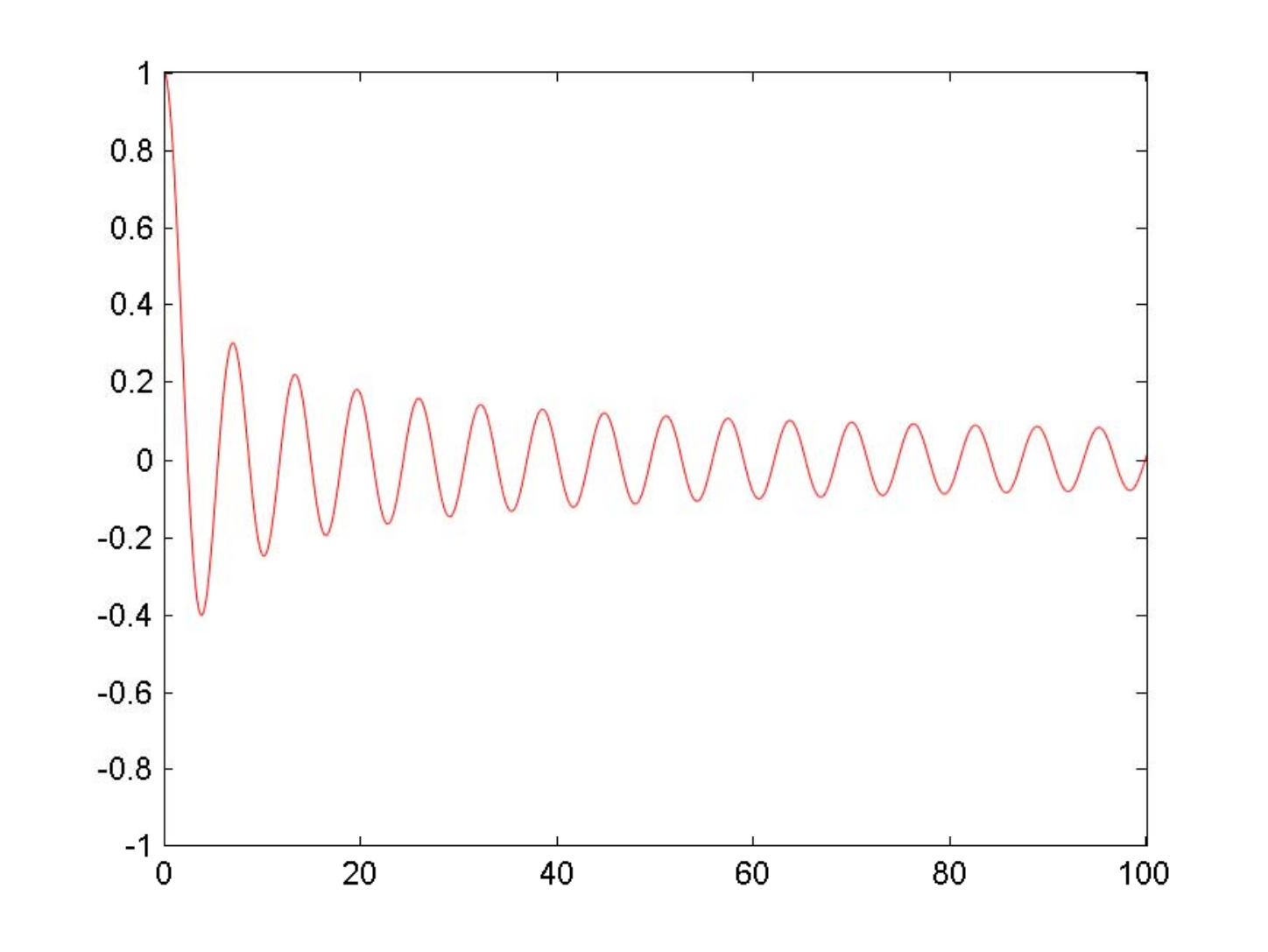}}
  \subfigure[\textbf{Bessel function $J_1$}]{
    \includegraphics[width=3in]{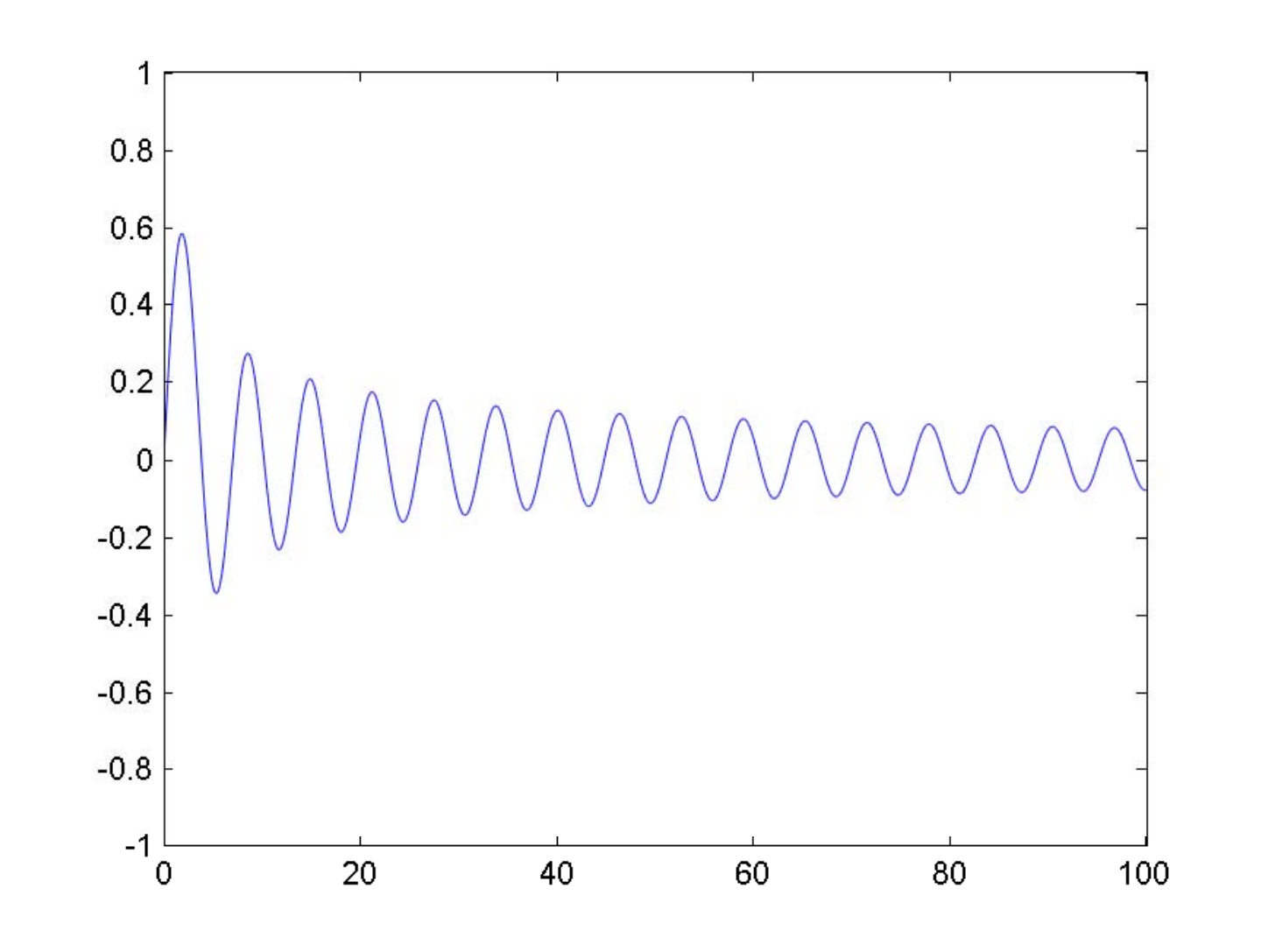}}
  \subfigure[\textbf{Spherical Bessel function $j_0$}]{
    \includegraphics[width=3in]{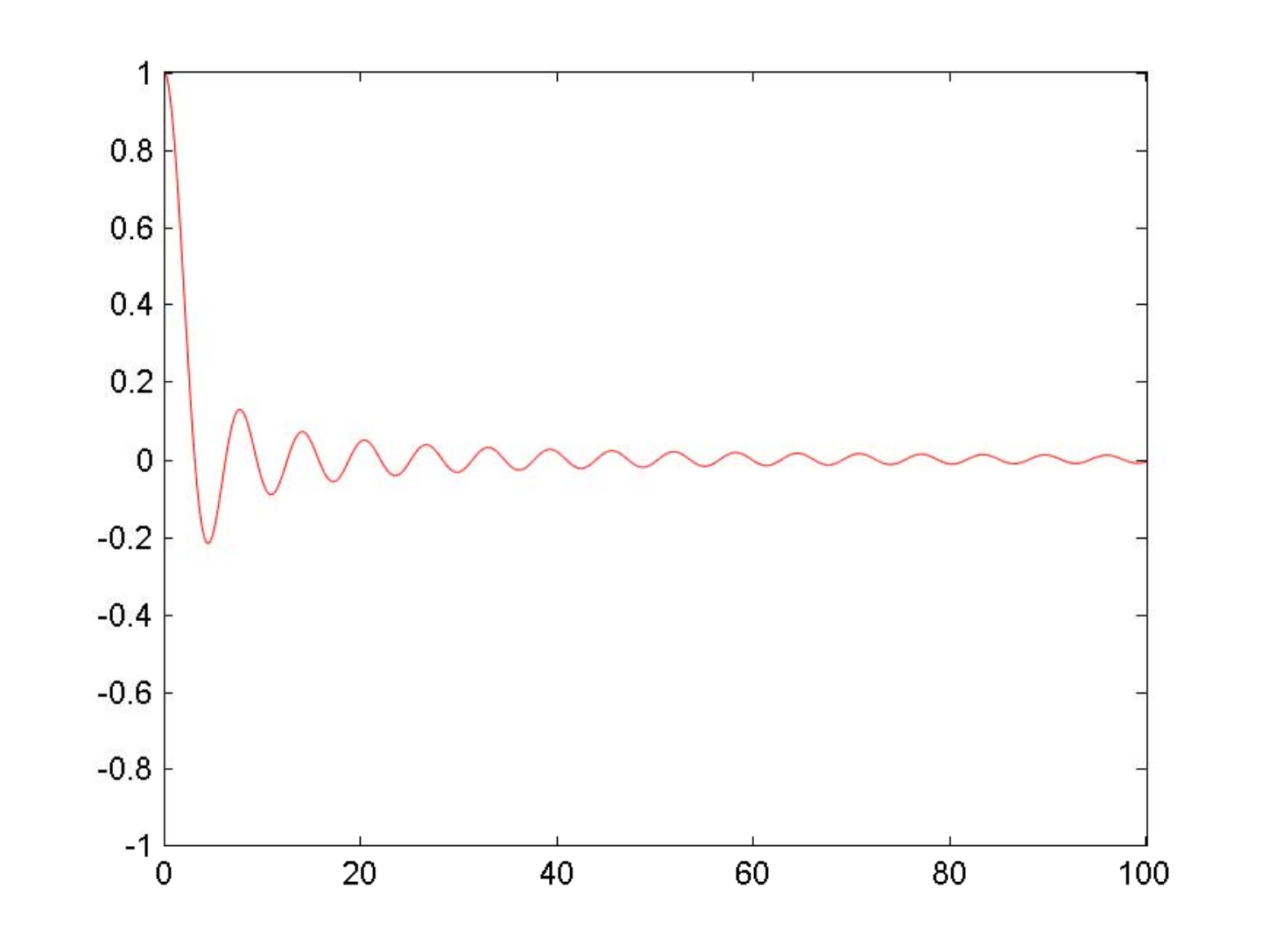}}
  \subfigure[\textbf{Spherical Bessel function $j_1$}]{
    \includegraphics[width=3in]{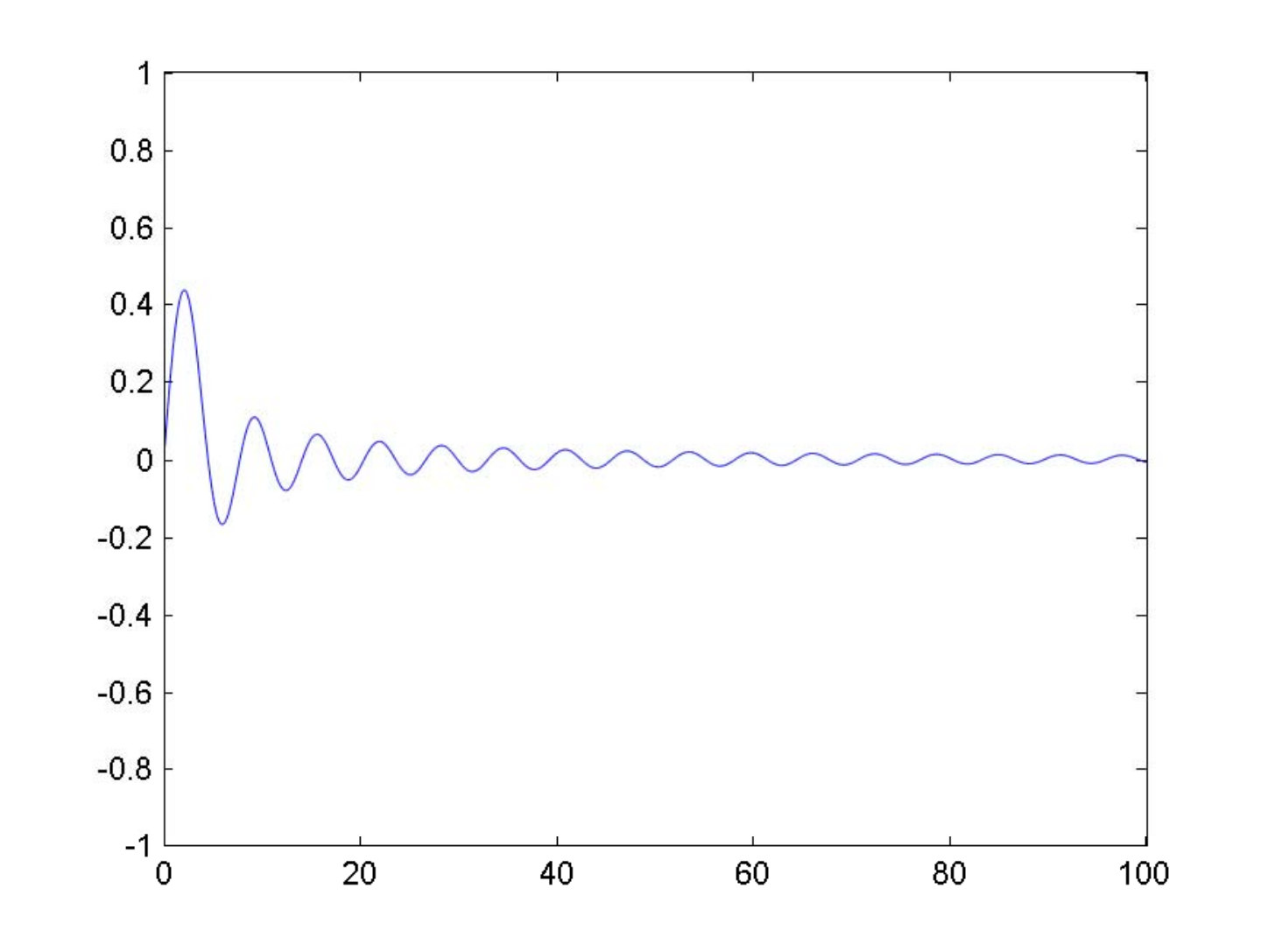}}
    \caption{Decay behaviors of $f_{\alpha}$ in two dimensions (a)-(b), and in three dimensions (c)-(d) for $\alpha=0,1$.}
\label{falphadecay}
\end{figure}

For both the scattering problems \eqref{HemEqumedium}-\eqref{Srcmedium} and \eqref{HemEquobstacle}-\eqref{Srcobstacle}, it is well known that the scattering amplitude
has the following form (cf. \cite{KirschGrinberg})
\ben
u^{\infty}(\hx,\hth)=\int_{\pa\Om}\left\{u^s(y,\hth)\frac{\pa e^{-ik\hx\cdot y}}{\pa\nu(y)}-\frac{\pa u^s}{\pa\nu}(y,\hth)e^{-ik\hx\cdot y}\right\}ds(y),\quad \hx\in S^{n-1}.
\enn
Inserting this into our indicator \eqref{IndicatorLiu} yields
\be\label{Inewz2}
&&I_{new}(z)\cr
&=&\Big|\int_{S^{n-1}}\int_{S^{n-1}}\int_{\pa\Om}\left\{u^s(y,\hth)\frac{\pa e^{-ik\hx\cdot (y-z)}}{\pa\nu(y)}-\frac{\pa u}{\pa\nu}(y,\hth)e^{-ik\hx\cdot (y-z)}\right\}ds(y)ds(\hx)e^{-ik\hth\cdot z}ds(\hth)\Big|\cr
&=&\Big|\int_{S^{n-1}}\int_{\pa\Om}\int_{S^{n-1}}\left\{-iku^s(y,\hth)\nu(y)\cdot\hx e^{-ik\hx\cdot (y-z)}-\frac{\pa u^s}{\pa\nu}(y,\hth)e^{-ik\hx\cdot (y-z)}\right\}ds(\hx)ds(y)\cr
&& \qquad\qquad\qquad\qquad  e^{-ik\hth\cdot z}ds(\hth)\Big|\cr
&:=&\Big|\int_{S^{n-1}}G(z, \hth)e^{-ik\hth\cdot z}ds(\hth)\Big|
\en
with
\ben
G(z,\hth)&:=&\int_{\pa\Om}\Big\{-iku^s(y,\hth)\nu(y)\cdot \int_{S^{n-1}}\hx e^{-ik\hx\cdot (y-z)}ds(\hx)\cr
&&  \qquad\quad       -\frac{\pa u^s}{\pa\nu}(y,\hth)\int_{S^{n-1}}e^{-ik\hx\cdot (y-z)}ds(\hx)\Big\}ds(y).
\enn
By the well known Riemann-Lebesgue Lemma, we obtain that
\be\label{Inewz2}
I_{new}(z)\rightarrow 0\quad\mbox{as}\,|z|\rightarrow\infty.
\en
From Lemma \ref{lemma_funkheck} we deduce that
\ben
G(z,\hth)=\int_{\pa\Om}\left\{-ik\mu_1u^s(y,\hth)\nu(y)\cdot\frac{y-z}{|y-z|}f_1(k|y-z|)-\mu_0\frac{\pa u^s}{\pa\nu}(y,\hth)f_0(k|y-z|)\right\}ds(y).
\enn
This means that $G(z,\hth)$ and thus also $I_{new}(z)$ are superpositions of the Bessel functions $f_0$ and $f_1$.
For large argument, we have the following asymptotic formulas for the Bessel and spherical Bessel functions
\ben
j_0(t)&=&\frac{\sin t}{t}\left\{1+O\Big(\frac{1}{t}\Big)\right\},\quad t\rightarrow\infty,\cr
j_1(t)&=&\frac{\cos t}{t}\left\{-1+O\Big(\frac{1}{t}\Big)\right\},\quad t\rightarrow\infty,\cr
J_0(t)&=&\frac{\cos t+\sin t}{\sqrt{\pi t}}\left\{1+O\Big(\frac{1}{t}\Big)\right\},\quad t\rightarrow\infty,\cr
J_1(t)&=&\frac{\cos t-\sin t}{\sqrt{\pi t}}\left\{-1+O\Big(\frac{1}{t}\Big)\right\},\quad t\rightarrow\infty.
\enn
We refer the readers to Figure \ref{falphadecay} for a visual display of the behavior of these four functions.
Thus, we expect that the $I_{new}(z)$ decays as the sampling point $z$ goes away from the boundary $\pa\Om$.
Actually, this phenomenon have been verified in a lot of the numerical simulations, see Section 3 in this paper.

We end this subsection by a stability statement, which reflects an important feature of the reconstruction scheme under consideration.
\begin{theorem}{\rm (Stability statement).}\label{stablity}
\be\label{stabilityestimate}
I_{new}(z)-I_{new}^{\delta}(z)\leq c\|u^{\infty}-u_{\delta}^{\infty}\|_{L^2(S^{n-1}\times S^{n-1})},\quad z\in\R^n.
\en
where $I_{new}^{\delta}(z)$ is the indicator functional with $u^{\infty}$ replaced by $u_{\delta}^{\infty}$, $c$ is a constant independent of sampling point $z$.
\end{theorem}
\begin{proof}
\ben
&&I_{new}(z)-I_{new}^{\delta}(z)\cr
&:=&\Big|\int_{S^{n-1}}e^{-ik\hth\cdot z}\int_{S^{n-1}}u^\infty(\hx,\hth)e^{ik\hx\cdot z}ds(\hx)ds(\hth)\Big|\cr
&&-\Big|\int_{S^{n-1}}e^{-ik\hth\cdot z}\int_{S^{n-1}}u_{\delta}^\infty(\hx,\hth)e^{ik\hx\cdot z}ds(\hx)ds(\hth)\Big|\cr\vspace{2cm}
&\leq&\Big|\int_{S^{n-1}}e^{-ik\hth\cdot z}\int_{S^{n-1}}[u^\infty(\hx,\hth)-u_{\delta}^\infty(\hx,\hth)]e^{ik\hx\cdot z}ds(\hx)ds(\hth)\Big|\cr
&\leq&c\|u^\infty-u_{\delta}^\infty\|_{L^{2}(S^{n-1}\times S^{n-1})}.
\enn
where we have used the Triangle Inequality in first inequality and the Cauchy-Schwarz Inequality in the second inequality.
\end{proof}

\subsection{Comparisons with other sampling methods}
In this subsection, we will study some connections with some other sampling methods. We mainly consider the Orthogonality Sampling Method proposed by Potthast in \cite{Potthast2010}
and the Reverse Time Migration proposed by Chen et.al. in \cite{CCHuang}.

In 2010, Potthast proposed a direct sampling method in \cite{Potthast2010} which he called Orthogonality Sampling based on the indicator
\be\label{IndicatorPotthast1}
I_{small}(z,\hth):=\Big|\int_{S^{n-1}}u^\infty(\hx,\hth)e^{ik\hx\cdot z}ds(\hx)\Big|,\quad z\in\R^n.
\en
This indicator actually is an evaluation of the modulus of the $L^2(S^{n-1})$ inner product of the far field measurements $( u^\infty(\hx,d)$ and suitably chosen function
$e^{-ik\hx\cdot z}$,
thus is very easy and simple for numerical simulation.
For small objects (whose size are much smaller than the wavelength), theoretical analysis shows that the indicator behaves like the Bessel functions.
Numerical examples (see also \cite{LiZou}) show the feasibility and effectiveness of the indicator $I_{small}$ for location reconstruction, in particular for small objects.
However, $I_{small}$ does not work for shape reconstruction, at least for extended scatterers.
To solve this difficulty, Potthast suggested in \cite{Potthast2010} the superpositions of $I_{small}$
with respect to the all the incident directions, i.e., by using the indicator
\be\label{IndicatorPotthast2}
I^{(\rho)}_{OSM}(z)
&:=&\int_{S^{n-1}} [I_{small}(z,\hth)]^{\rho} ds(\hth)\cr
&=&\int_{S^{n-1}}\Big|\int_{S^{n-1}}u^\infty(\hx,\hth)e^{ik\hx\cdot z}ds(\hx)\Big|^{\rho} ds(\hth),\quad z\in\R^n,
\en
with $\rho=1$ or $\rho=2$.
Numerical examples show that it is indeed a good indicator for shape reconstruction for extended scatterers. However, no theoretical analysis is established for the behavior
of $I^{(\rho)}_{OSM}$.
As pointed out by Potthast in \cite{Potthast2010}, the modulus as used in \eqref{IndicatorPotthast2} before further integral is very important for reconstruction. Recall our
indicator \eqref{IndicatorLiu2}, that is,
\be\label{IndicatorLiu2}
I_{new}(z):=\Big|\int_{S^{n-1}}e^{-ik\hth\cdot z}\int_{S^{n-1}}u^\infty(\hx,\hth)e^{ik\hx\cdot z}ds(\hx)ds(\hth)\Big|,\quad z\in\R^n,
\en
where the modulus is taken after the intergral but a weight function $e^{-ik\hth\cdot z}$ is used in the integral w.r.t. $\hth\in S^{n-1}$.
Another work close to our method is the Reverse Time Migration method which has been studied by Chen et.al. in \cite{CCHuang}.
The authors considered the point source as the incident field and the scattered field as the measurement. Based on similar arguments,
for the case of incident plane waves and far field measurements, one may consider the following indicator functional
\be\label{RTM}
I_{RTM}(z):=\Im\int_{S^{n-1}}e^{-ik\hth\cdot z}\int_{S^{n-1}}u^\infty(\hx,\hth)e^{ik\hx\cdot z}ds(\hx)ds(\hth),\quad z\in\R^n.
\en

In the following, we will show that
\be\label{IndicatorsRelations}
c I^{2}_{OSM}(z)\leq I_{RTM}(z)\leq I_{new}(z) \leq C \sqrt{I^{2}_{OSM}(z)},
\en
for two positive constants $c$ and $C$ which are independent of the sampling point $z\in \R^n$.
This implies these indicators  $I^{2}_{OSM}(z)$ in \eqref{IndicatorPotthast2},  $I_{new}(z)$ in \eqref{IndicatorLiu2} and
$I_{RTM}(z)$ in \eqref{RTM} are equivalent in some sense.

To show \eqref{IndicatorsRelations}, we recall the reciprocity relation of the scattering amplitudes and some properties of the far field operator $F$ given by \eqref{ffoperator}.

\begin{lemma}\label{Fproperties}
(1) The scattering amplitude satisfies the reciprocity relation
\be\label{reciprocityrelation}
u^{\infty}(\hx,\hth)=u^{\infty}(-\hth, -\hx)\quad\forall\, \hx,\hth\in S^{n-1}.
\en\\
(2) The far field operator satisfies
\be\label{FFastR}
F-F^{\ast}-\frac{i}{4\pi}\Big(\frac{k}{2\pi}\Big)^{n-2}F^{\ast}F=2iR,
\en
where $F^{\ast}$ denotes the $L^2-$adjoint of $F$ and $R: L^2(S^{n-1})\rightarrow L^2(S^{n-1})$ is some self-adjoint non-negative operator.
The operator $R$ vanishes for the cases of Dirichlet or Neumann boundary conditions.
For the impedance boundary conditions, the operator $R$ is given by
\be\label{Rimp}
(Rh)(\hx):=\int_{S^{n-1}}\Big(\int_{\pa\Om}\Im(\la)u(y,\hth)\ov{u(y,\hx)}ds(y)\Big)h(\hth)ds(\hth),\quad \hx\in S^{n-1}.
\en
For the case of inhomogeneous medium, the operator $R$ is given by
\be\label{Rmed}
(Rh)(\hx):=\int_{S^{n-1}}\Big(\int_{\Om}k^2\Im(q)u(y,\hth)\ov{u(y,\hx)}dy\Big)h(\hth)ds(\hth),\quad \hx\in S^{n-1},
\en
where $u(\cdot,\hth)$ is the total field in $\Om$ corresponding to the incident plane wave $u^i(\cdot,\hth)$ with incident direction $\hth$.
Clearly, the operator $R$ also vanishes for the cases of $\Im(\la)=0$ or $\Im(q)=0$.
\end{lemma}
\begin{proof}
The proof of three-dimensional case can be found in Theorems 1.8, 2.5 and 4.4 of \cite{KirschGrinberg} (see also Theorem 2.1 in \cite{KirschLiuMMAS}).
By literally the same proofs, one can show the two-dimensional case with suitably modified coefficient. Thus we omit the proof here.
\end{proof}

By interchanging the roles of $\hx$ and $\hth$, with the help of the reciprocity relation \eqref{reciprocityrelation}, we have
\be\label{Fphiz2}
I^{(2)}_{OSM}(z)
&=&\int_{S^{n-1}}\Big|\int_{S^{n-1}}u^{\infty}(\hx,\hth)e^{ikz\cdot\hx}ds(\hx)\Big|^2ds(\hth)\cr
&=&\int_{S^{n-1}}\Big|\int_{S^{n-1}}u^{\infty}(-\hth,-\hx)e^{ikz\cdot\hx}ds(\hx)\Big|^2ds(\hth)\cr
&=&\int_{S^{n-1}}\Big|\int_{S^{n-1}}u^{\infty}(\hth,\hx)e^{-ikz\cdot\hx}ds(\hx)\Big|^2ds(\hth)\cr
&=&\int_{S^{n-1}}\Big|\int_{S^{n-1}}u^{\infty}(\hx,\hth)e^{-ikz\cdot\hth}ds(\hth)\Big|^2ds(\hx)\cr
&=&\|F\phi_z\|^2_{L^2(S^{n-1})}.
\en

By using the operator identity \eqref{FFastR}, we have
\ben
\Im(Fg, g)=\frac{1}{8\pi}\Big(\frac{k}{2\pi}\Big)^{n-2}(Fg,Fg)+(Rg,g),\quad \forall\,\, g\in L^{2}(S^{n-1}).
\enn
In particular, taking $g=\phi_z$ yields that
\be\label{lowb}
|(F\phi_z,\phi_z)|
&\geq& \Im(F\phi_z, \phi_z)\cr
&=& \frac{1}{8\pi}\Big(\frac{k}{2\pi}\Big)^{n-2}\|F\phi_z\|^2_{L^2(S^{n-1})}+(R\phi_z,\phi_z)\cr
&\geq& \frac{1}{8\pi}\Big(\frac{k}{2\pi}\Big)^{n-2}\|F\phi_z\|^2_{L^2(S^{n-1})},
\en
where we have used the fact that $R$ is an non-negative operator.
On the other hand, by using the Cauchy-Schwartz inequality, we have
\be\label{upb}
|(F\phi_z,\phi_z)|^2
\leq \|F\phi_z\|^2_{L^2(S^{n-1})}\|\phi_z\|^2_{L^2(S^{n-1})}
= 2^{n-1}\pi\|F\phi_z\|^2_{L^2(S^{n-1})}.
\en
Combing the previous two inequalities \eqref{lowb} and \eqref{upb} yields
\be
\frac{1}{8\pi}\Big(\frac{k}{2\pi}\Big)^{n-2} \|F\phi_z\|^2_{L^2(S^{n-1})}\leq \Im(F\phi_z, \phi_z)\leq |(F\phi_z,\phi_z)| \leq \sqrt{\pi}2^{\frac{n-1}{2}}\|F\phi_z\|_{L^2(S^{n-1})},
\en
which actually implies \eqref{IndicatorsRelations} with $c=\frac{1}{8\pi}\Big(\frac{k}{2\pi}\Big)^{n-2}$ and $C=\sqrt{\pi}2^{\frac{n-1}{2}}$.

We have established the theory foundation of the reconstruction scheme by using the indicator $I_{new}$ in the previous subsections.
Thus similar results can be shown for the indicators $I^{(2)}_{OSM}$ in Orthogonal Sampling method and $I_{RTM}$ in Reverse Time Migration method.

Finally, we want to remark that similar indicators also been proposed in \cite{BaoHuangLiZhao} based on the idea of the MUSIC algorithm and in
\cite{BellisBonnetCakoni} by using the topological derivative of far field measurements-based $L^{2}$ cost functionals. The contributions of our paper
are new theory basis of such kind sampling methods and extensive numerical experiments shown in the next section.

\section{Numerical examples and discussions}
\label{sec3}
\setcounter{equation}{0}

Now we turn to present a variety of numerical examples in two dimensions to illustrate the applicability and effectiveness of our sampling method.
All the programs in our experiments are written in Matlab and run on a Core i5-5200U 2.2GHz PC.

There are totally seven groups of numerical tests to be considered, and they are
respectively referred to as {\bf Comparison, Dirichlet, OtherPhyPro, MixedType, MultiScalar, ResolutionLimit}  and {\bf HighResolution}.
The boundaries of the scatterers used in our numerical experiments are parameterized as follows:
\be
\label{circle}&\mbox{\rm Circle:}&\quad x(t)\ =(a,b)+r\ (\cos t, \sin t),\quad 0\leq t\leq2\pi,\\
\label{peanut}&\mbox{\rm Peanut:}&\quad x(t)\ =(a,b)+\ \sqrt{3\cos^2 t+1}(\cos t, \sin t),\quad 0\leq t\leq2\pi,\\
\label{pear}&\mbox{\rm Pear:}&\quad x(t)\ =(a,b)+(2+0.3\cos 3t)\ (\cos t, \sin t),\quad 0\leq t\leq2\pi,\\
\label{kite}&\mbox{\rm Kite:}&\quad x(t)\ =(a,b)+\ (\cos t+0.65\cos 2t-0.65, 1.5\sin t),\quad 0\leq t\leq2\pi,
\en
with $(a,b)$ be the location of the scatter which may be different in different examples.

In our simulations, we used the boundary integral equation method to compute the scattering amplitudes $u_\Om^\infty(\theta_j, \theta_l)$  with $\theta_j = 2\pi j/N$,
for $N$ equidistantly distributed incident directions and $N$ observation directions.
These data are then stored in the matrices $F_\Om \in \C^{N \times N}$.
We further perturb
$F_\Om$ by random noise using
\ben
F_{\Om}^{\delta}\ =\ F_{\Om} +\delta\|F_{\Om}\|\frac{R_1+R_2 i}{\|R_1+R_2 i\|},
\enn
where $R_1$ and $R_2$ are two $N \times N$ matrixes containing pseudo-random values
drawn from a normal distribution with mean zero and standard deviation one. The
value of $\delta$ used in our code is $\delta:=\|F_{\Om}^{\delta} -F_{\Om}\|/\|F_{\Om}\|$ and so presents the relative error.

In the simulations, we used a grid $\mathcal{G}$ of $M\times M$ equally spaced sampling points on some rectangle $[-c,c]\times[-c,c]$.
For each point $z \in \mathcal{G}$, we define the indicator function
\ben
W(z)\ :=|\phi_z^{\ast}F_{\Om}^{\delta}\phi_z|^{\rho},
\enn
where $\phi_z=(e^{-ik\theta_1\cdot z}, e^{-ik\theta_2 \cdot z},\dots, e^{-ik\theta_{N}\cdot z})^\top \in \C^{N}$.
Clearly, the indicator is independent of any a priori information of the unknown scatterers. \\

{\bf Example Comparison.}
This example is designed to compare the reconstructions by using different sampling methods. We set the wave number $k=5$, $\rho=1$ and consider a sound soft kite shaped domain.
The research domain is $[-4,4]\times[-4,4]$ with $151\times 151$ equally spaced sampling points.
We take $N=64$, i.e., the scattering amplitude is collected in 64 observation directions and 64 incident directions.
Figure \ref{comparison} show the reconstruction by using the Factorization method, Orthogonal Sampling method, Reverse Time Migration method and Our novel method.
Clearly, the reconstructions are comparable to each other.
Different to the classical Factorization method, in the later three sampling methods, the indicator decays like the bessel functions as the sampling points away from
the boundary. This verifies our theory analyses and can be observed in later examples.
\begin{figure}[htbp]
  \centering
\subfigure[\textbf{Factorization method}]{
    \includegraphics[width=3in]{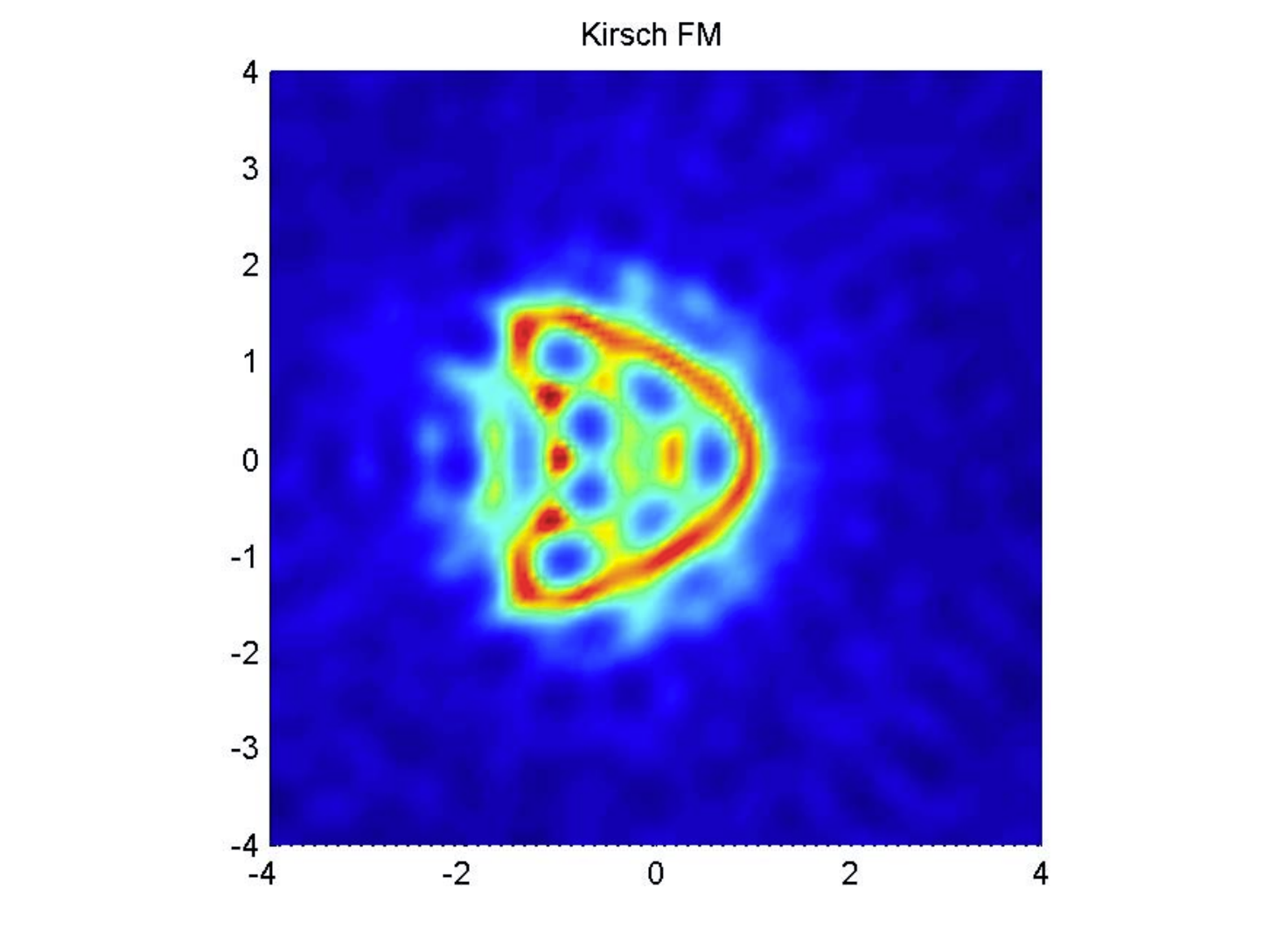}}
  \subfigure[\textbf{Orthogonal Sampling method}]{
    \includegraphics[width=3in]{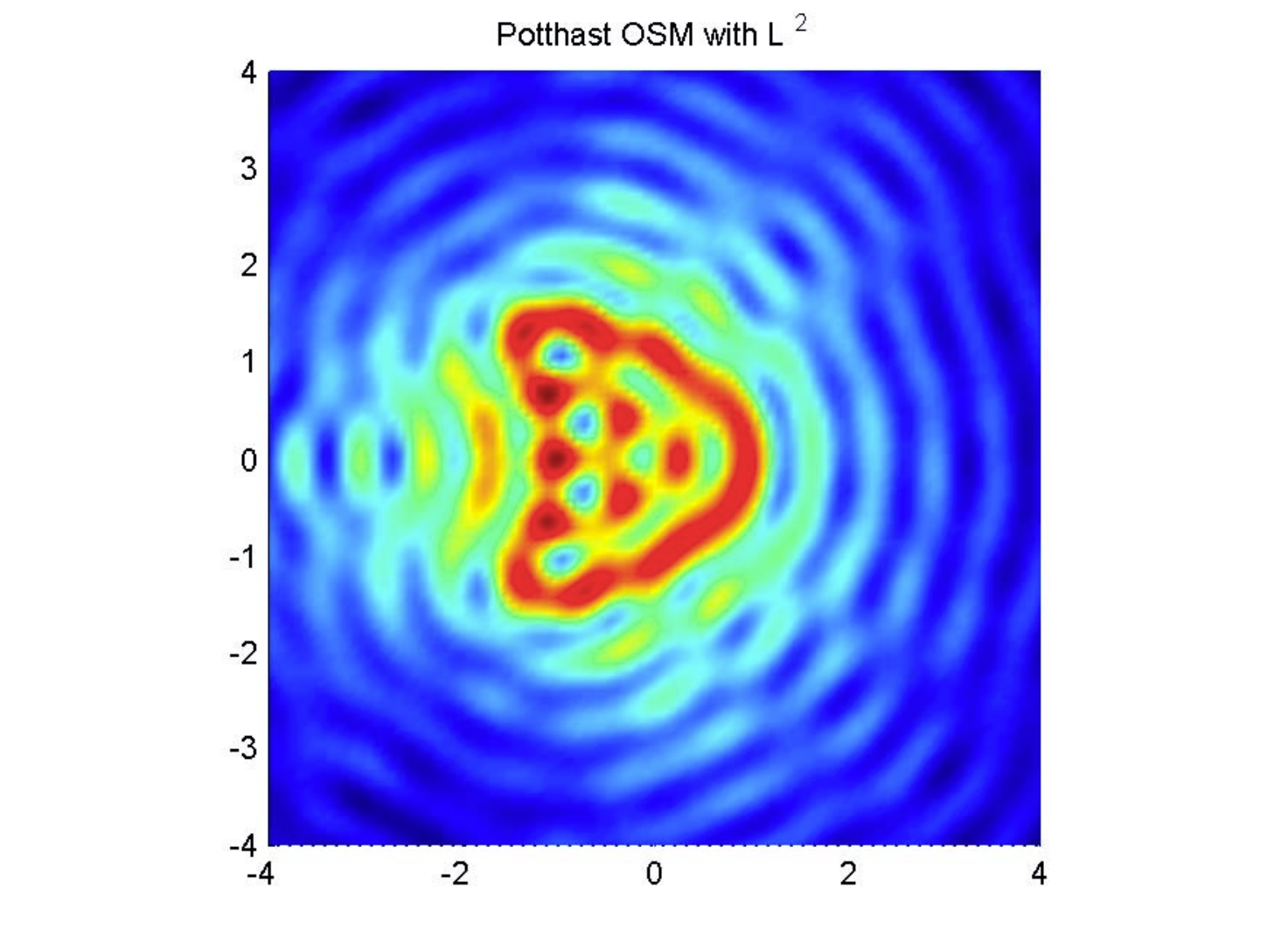}}
  \subfigure[\textbf{Reverse Time Migration method}]{
    \includegraphics[width=3in]{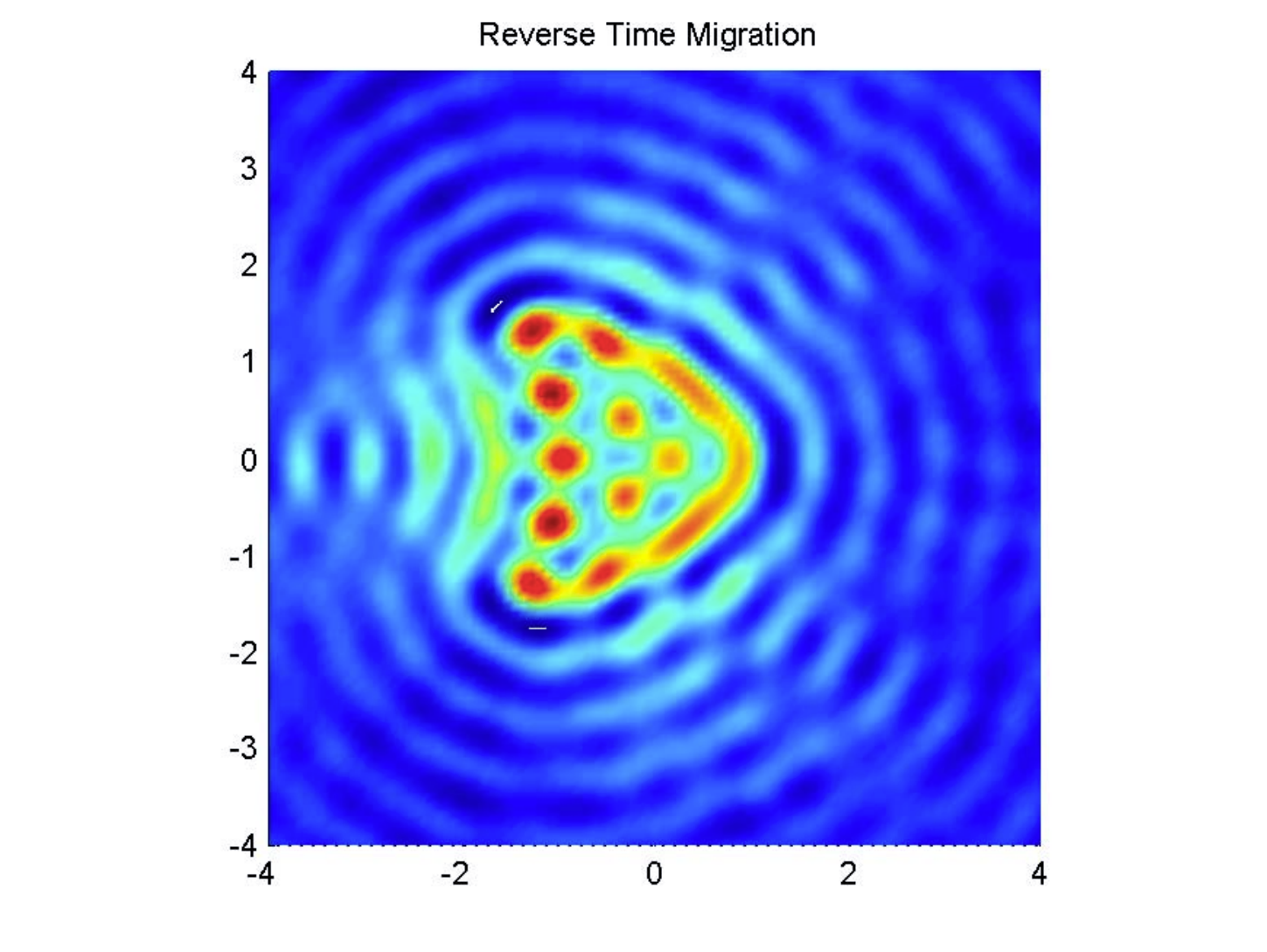}}
  \subfigure[\textbf{Our novel method}]{
    \includegraphics[width=3in]{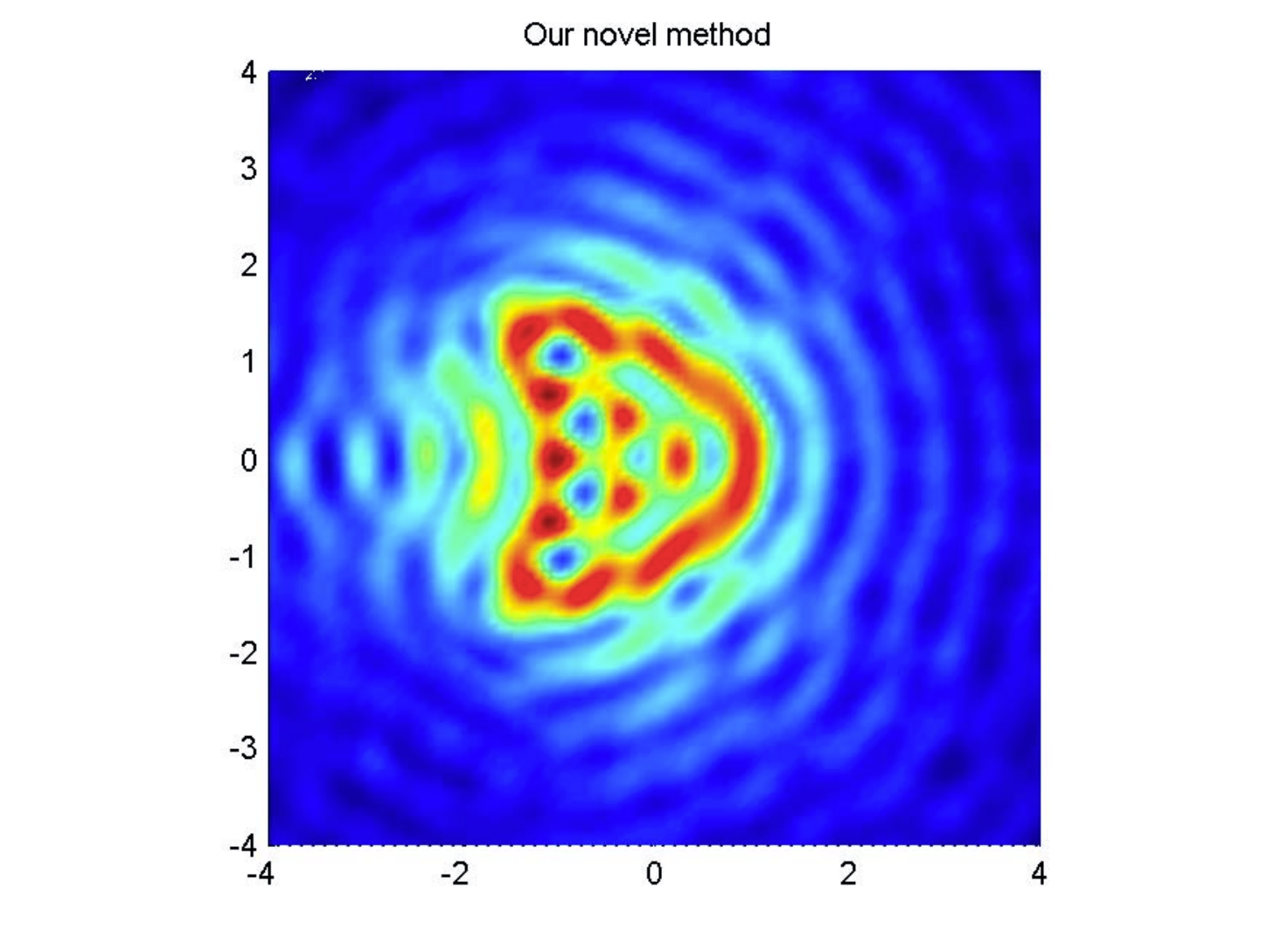}}
\caption{{\bf Example Comparison.}\, Reconstruction of Kite shaped domain with different sampling methods and $30\%$ noise. }
\label{comparison}
\end{figure}

\medskip

\noindent

To enhance the contrast, we take $\rho=2$ in the next five examples.

{\bf Example Dirichlet.}
This example is designed to check the validity of our method for scatterers with different shapes. For simplicity, we impose Dirichlet boundary condition on the
underlying scatterers. The other boundary conditions will be considered in the subsequent examples.
The same as previous example, we set $k=5$, $N=64$.
The research domain is $[-4,4]\times[-4,4]$ with $151\times 151$ equally spaced sampling points.
For the circle given in \eqref{circle}, we take the radius $r$ to be $2$.
Figures \ref{softkite}-\ref{softpear} show the reconstructions of kite, circle, peanut and pear shaped domain, respectively. All the results show that our method is very robust
to noise. Surprisingly, we observe that, even up to $90\%$ noise is added, both the location and shape of the underlying scatterer can be roughly reconstructed.
\begin{figure}[htbp]
  \centering
\subfigure[\textbf{Kite}]{
    \includegraphics[width=3in]{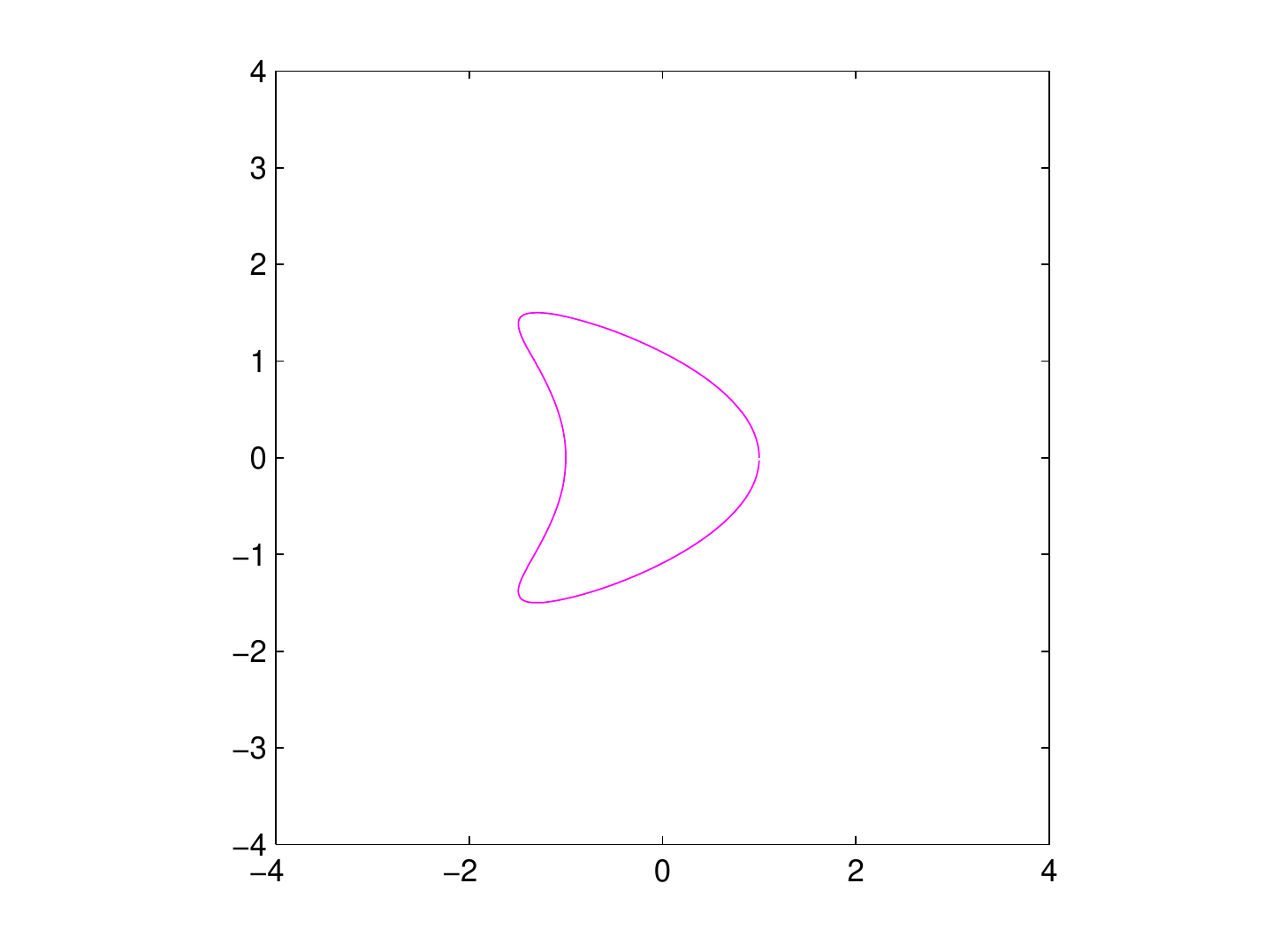}}
  \subfigure[\textbf{10\% noise}]{
    \includegraphics[width=3in]{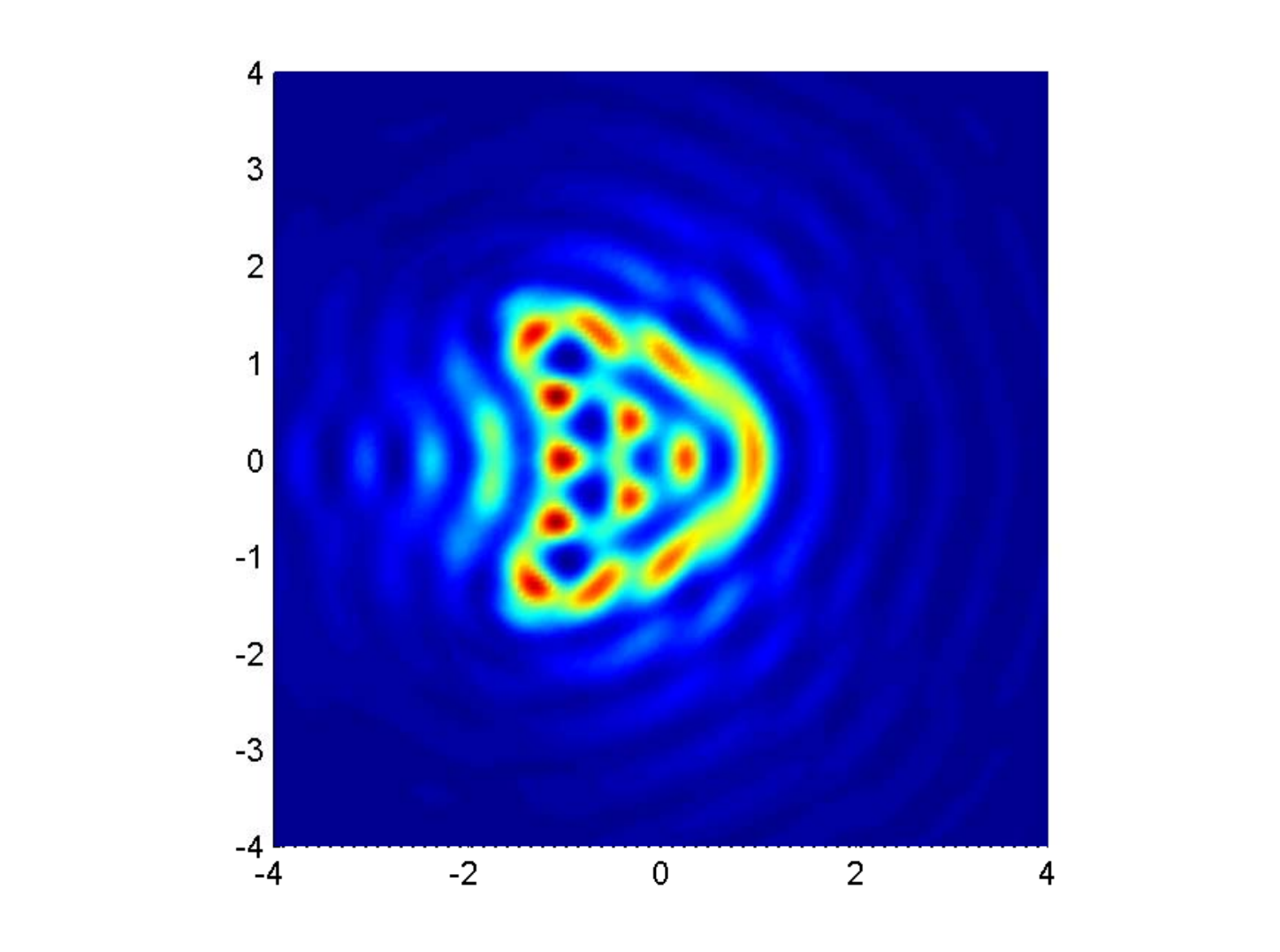}}
  \subfigure[\textbf{30\% noise}]{
    \includegraphics[width=3in]{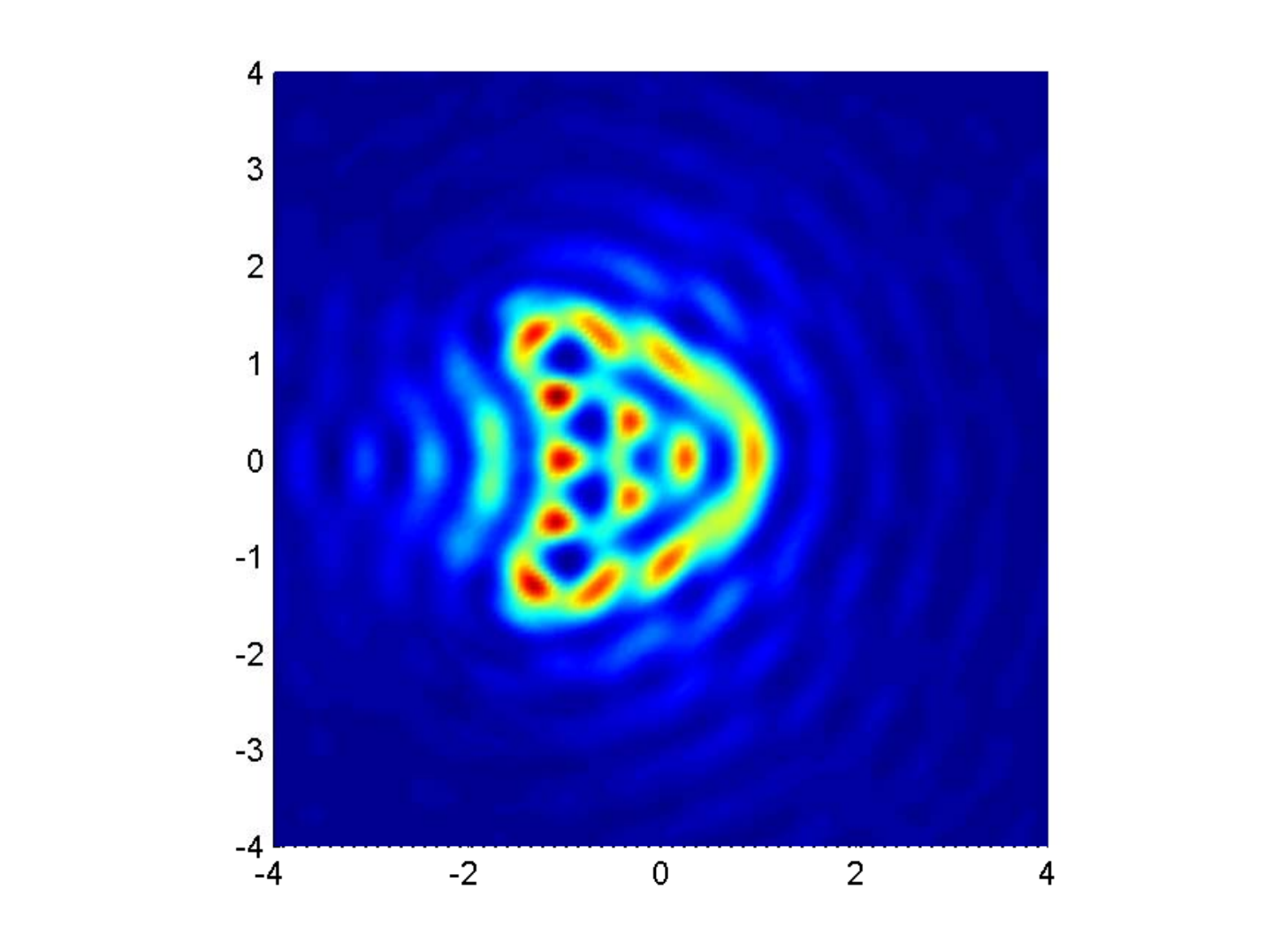}}
  \subfigure[\textbf{90\% noise}]{
    \includegraphics[width=3in]{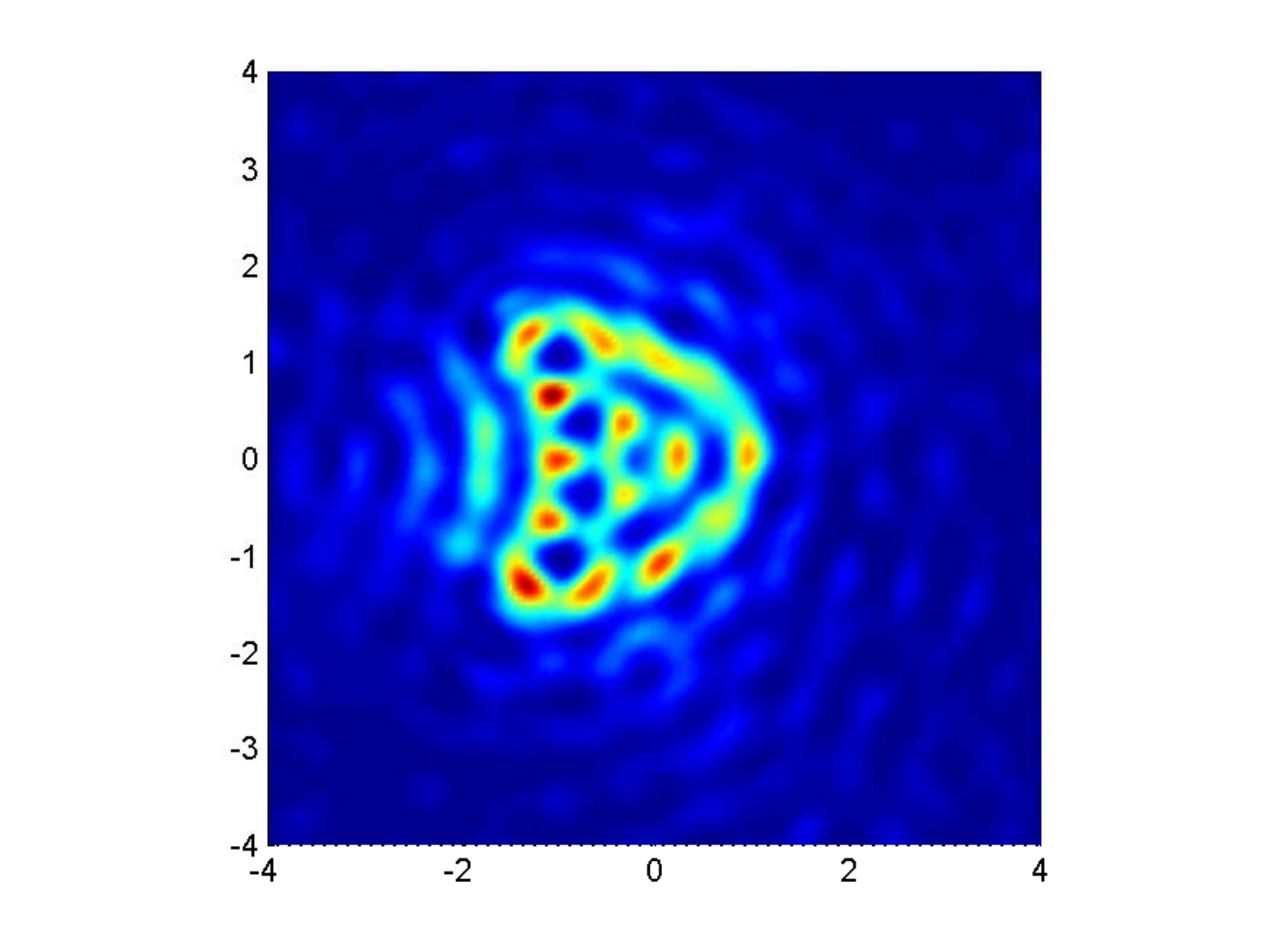}}
\caption{{\bf Example Soft.}\, Reconstruction of Kite shaped domain with different noise.}
\label{softkite}
\end{figure}
\begin{figure}[htbp]
  \centering
  \subfigure[\textbf{Circle}]{
    \includegraphics[width=3in]{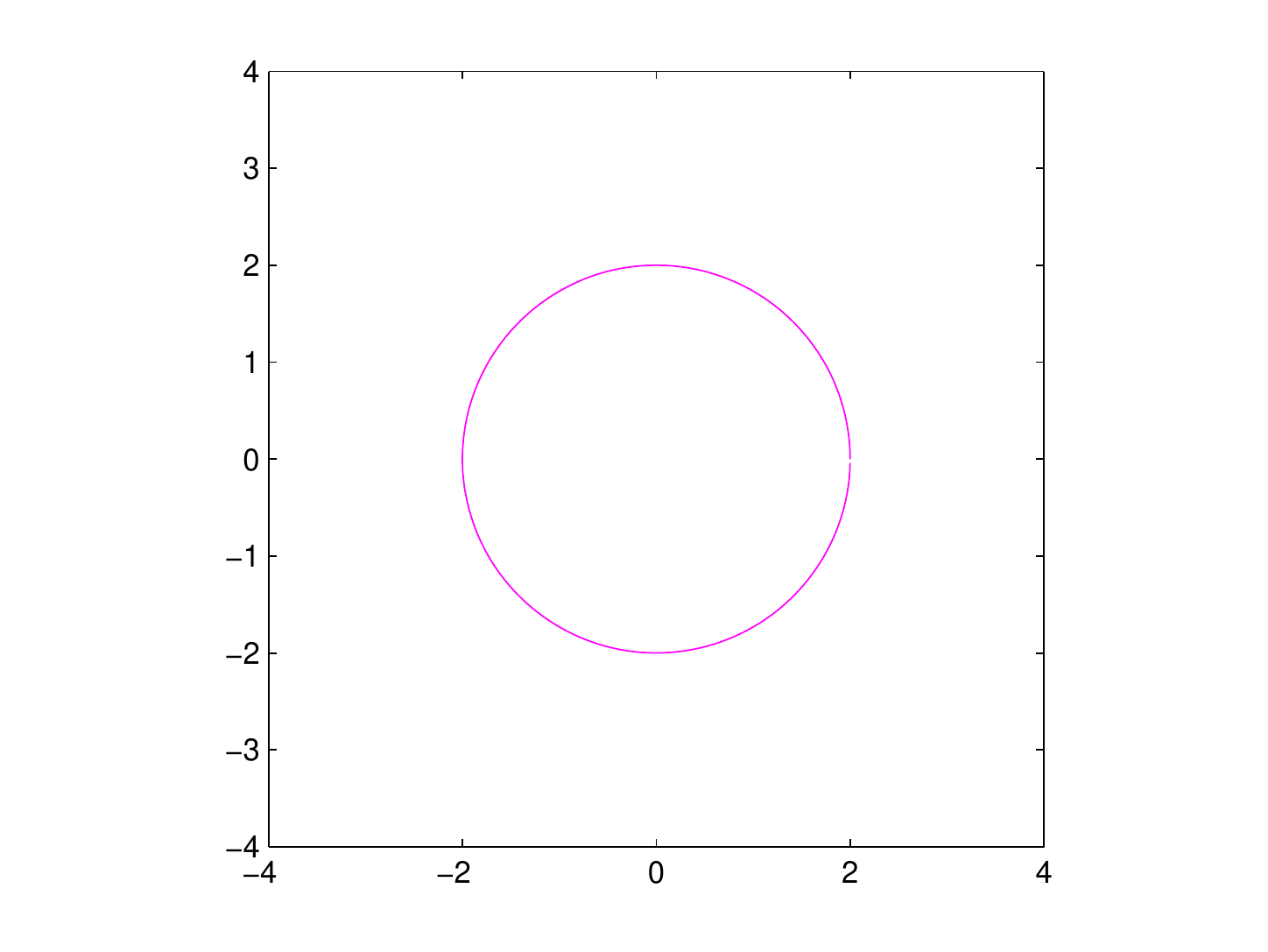}}
  \subfigure[\textbf{10\% noise}]{
    \includegraphics[width=3in]{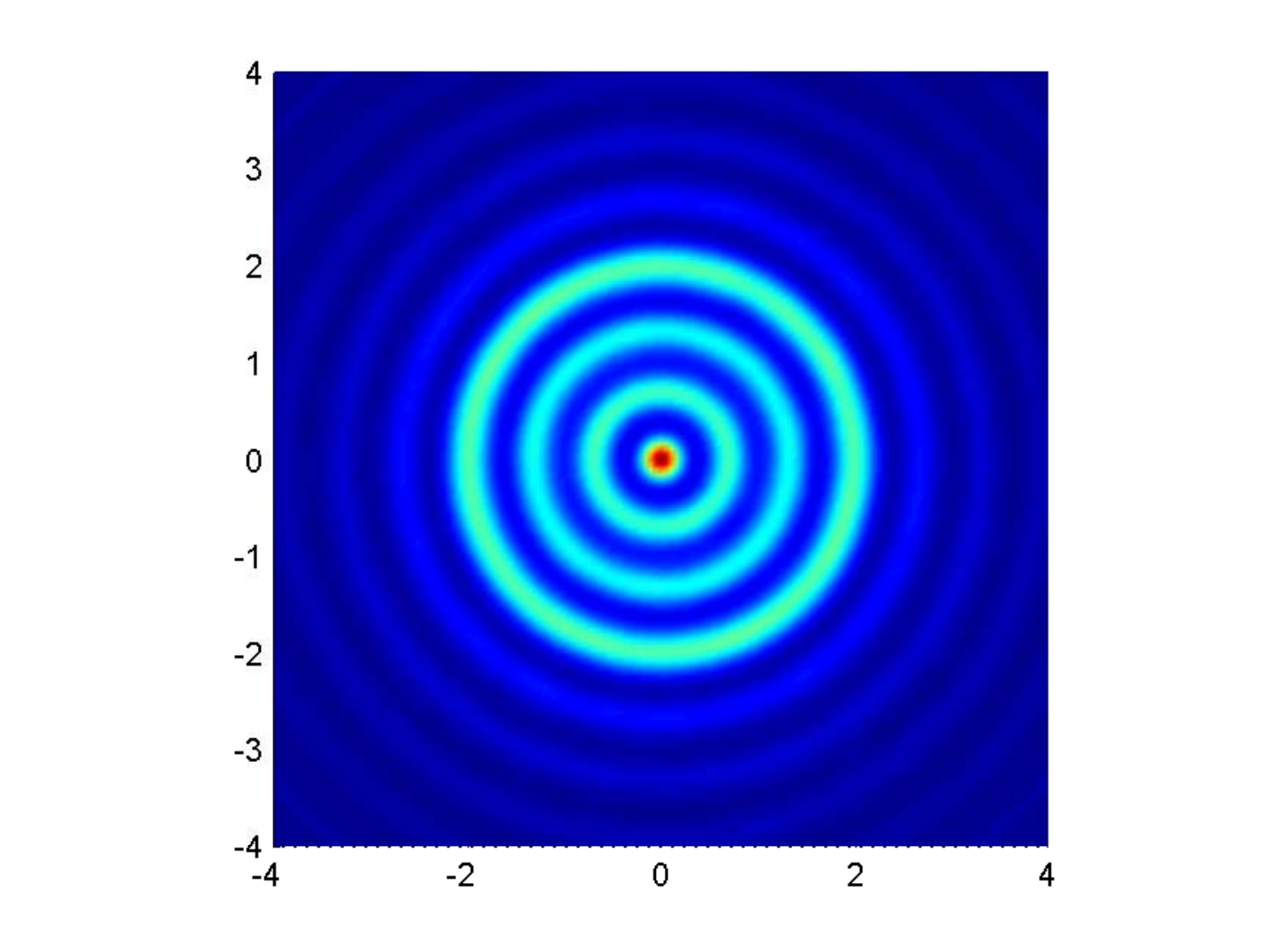}}
  \subfigure[\textbf{30\% noise}]{
    \includegraphics[width=3in]{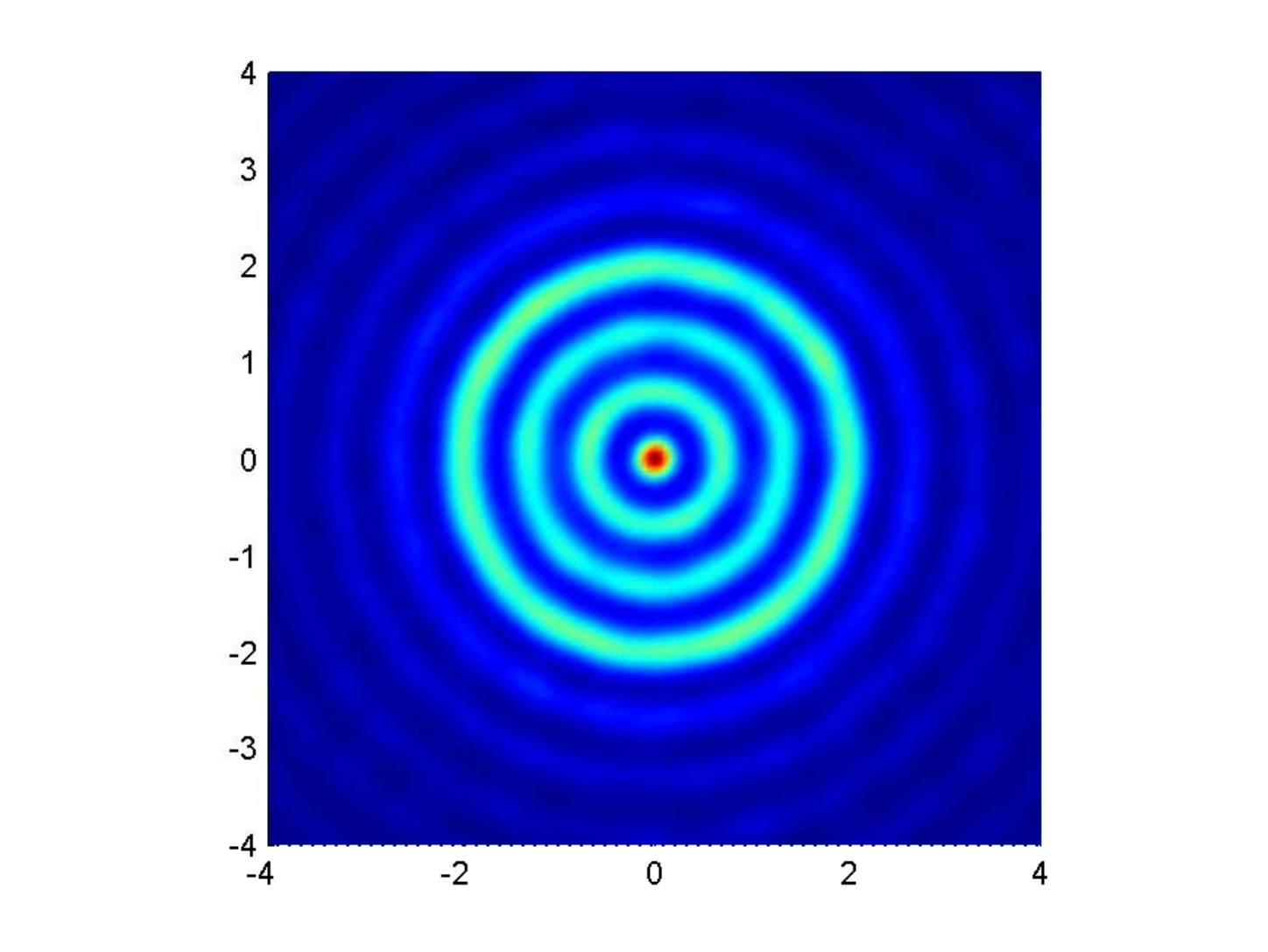}}
  \subfigure[\textbf{90\% noise}]{
    \includegraphics[width=3in]{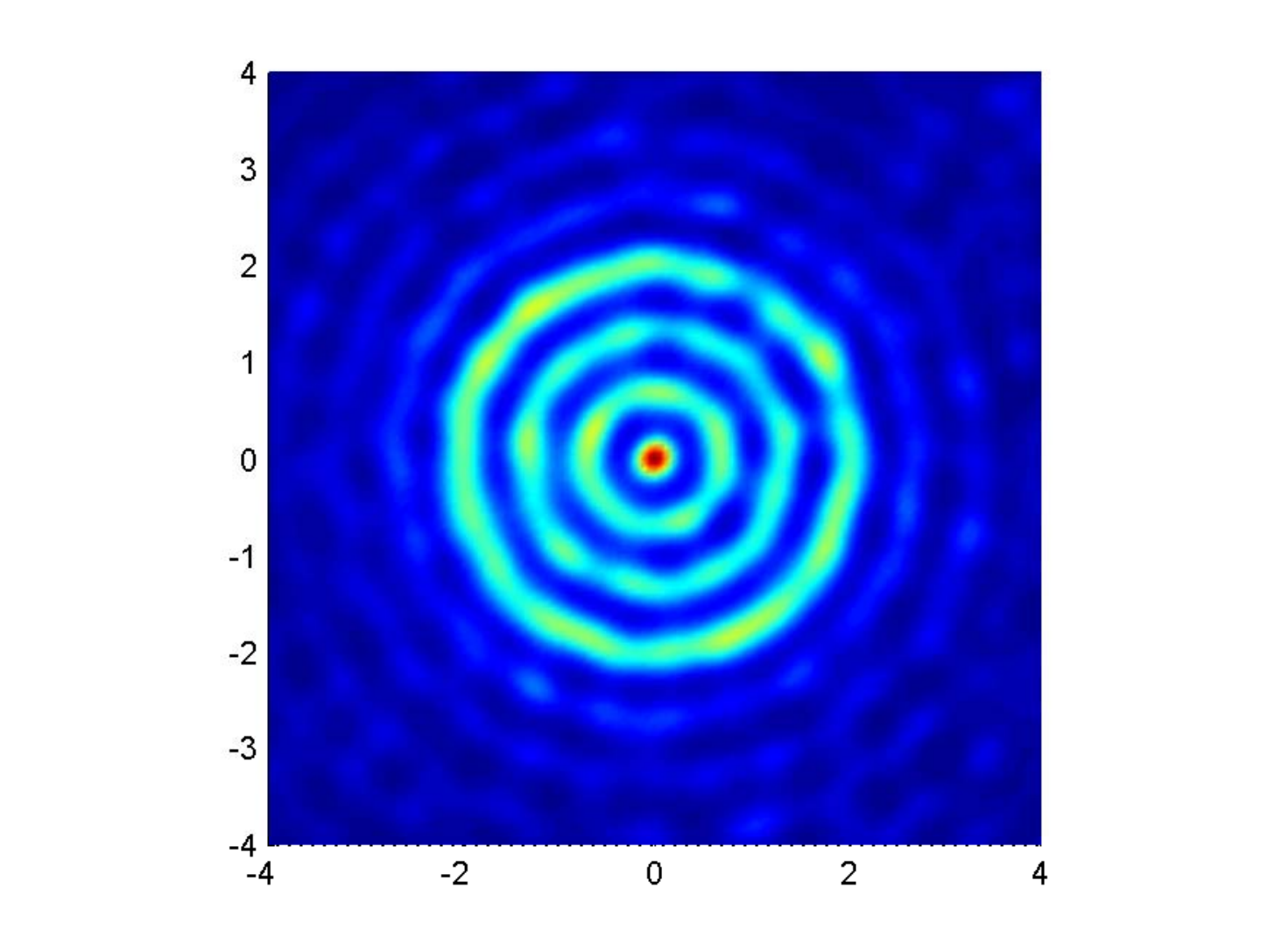}}
\caption{{\bf Example Soft.}\, Reconstruction of a disk with different noise.}
\label{softcircle}
\end{figure}
\begin{figure}[htbp]
  \centering
\subfigure[\textbf{Peanut}]{
    \includegraphics[width=3in]{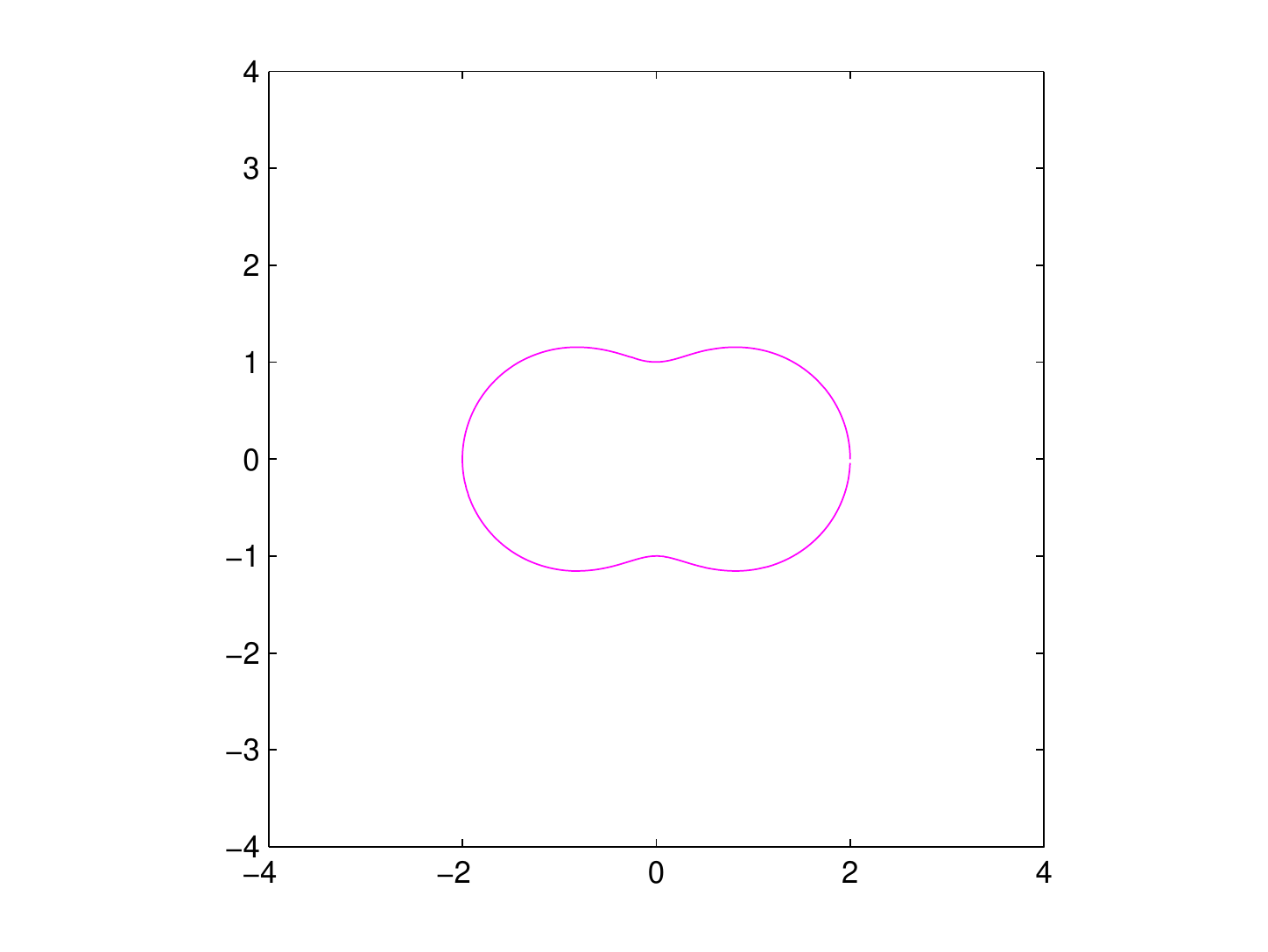}}
  \subfigure[\textbf{10\% noise}]{
    \includegraphics[width=3in]{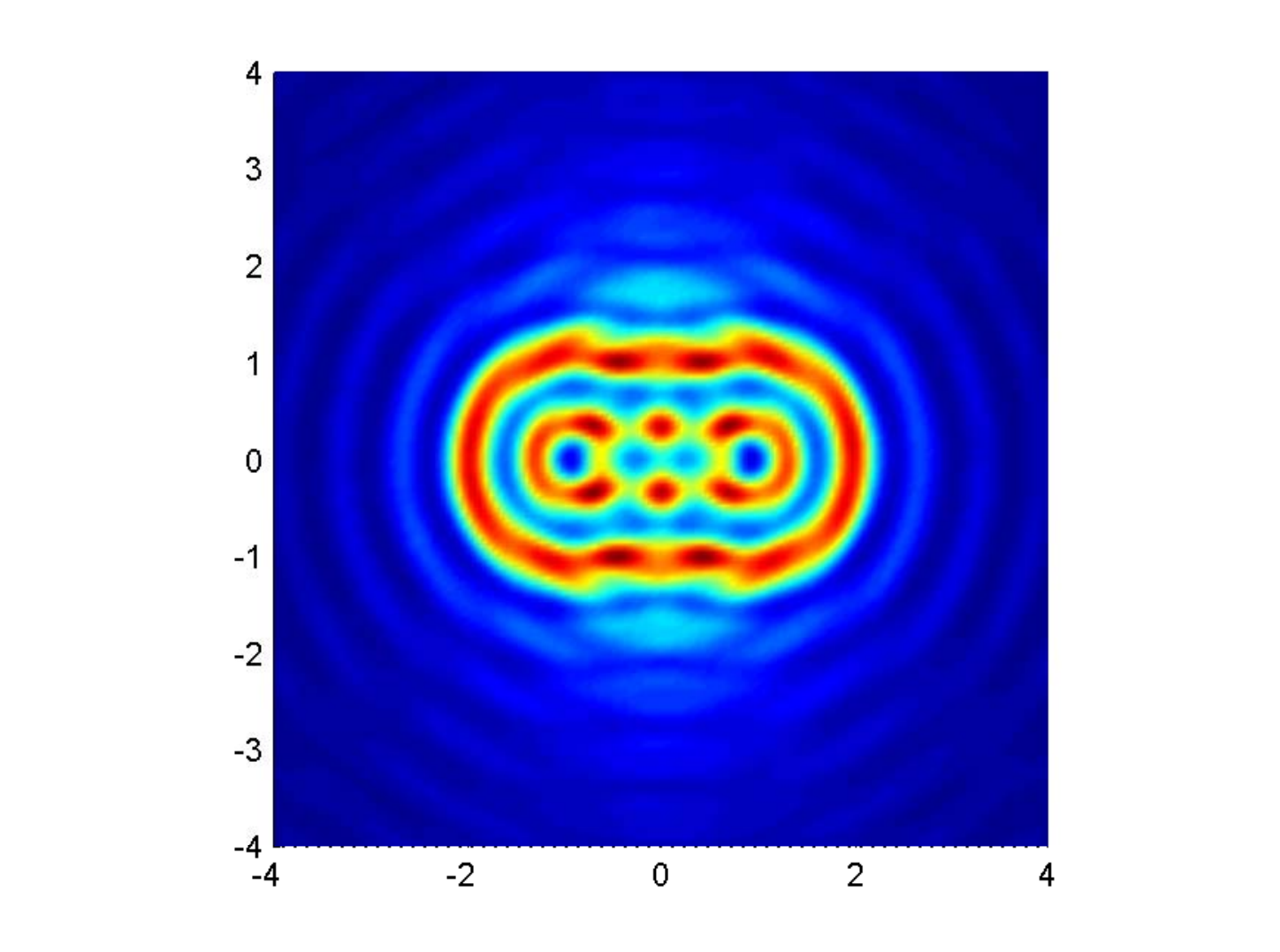}}
  \subfigure[\textbf{30\% noise}]{
    \includegraphics[width=3in]{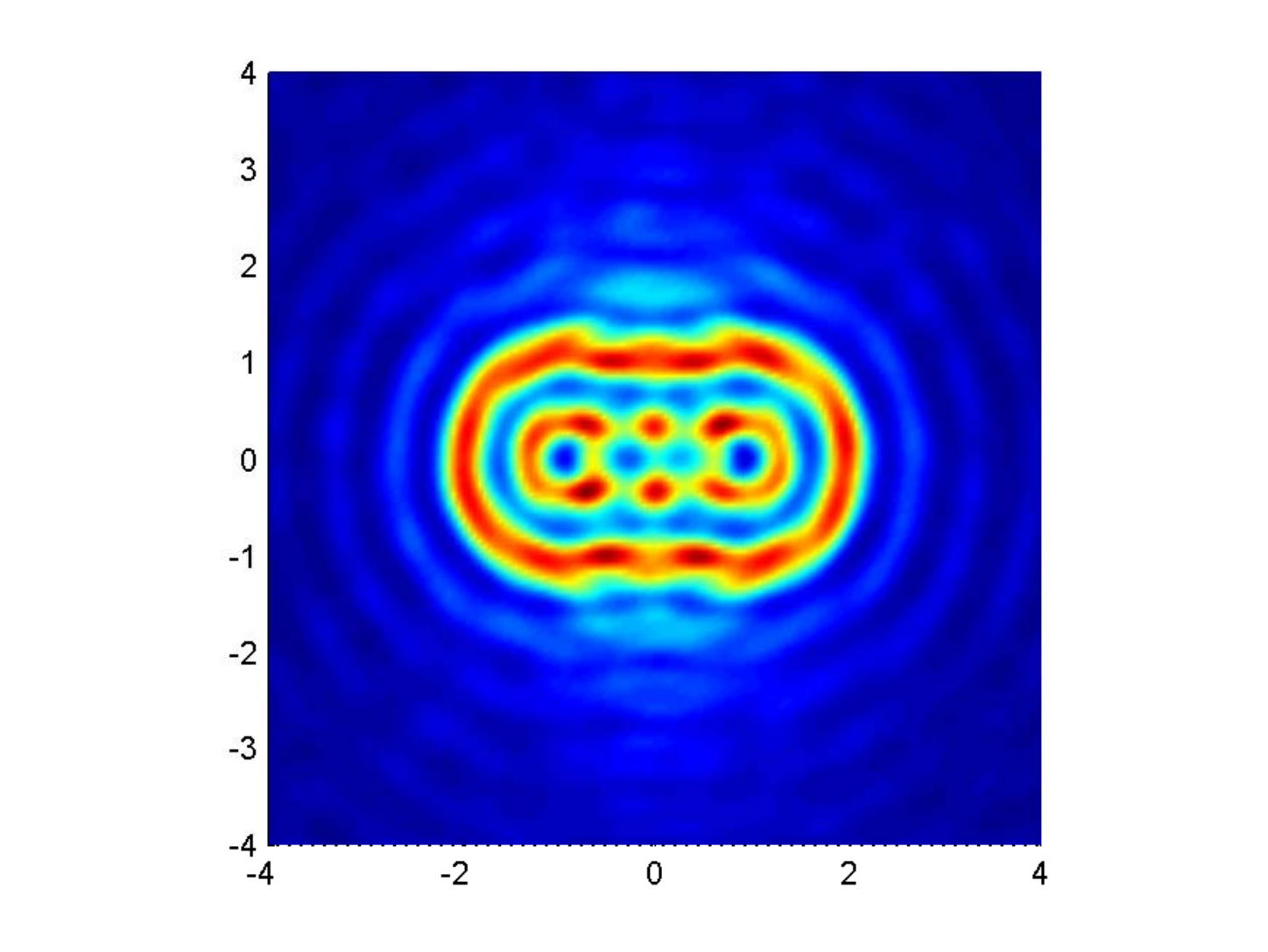}}
  \subfigure[\textbf{90\% noise}]{
    \includegraphics[width=3in]{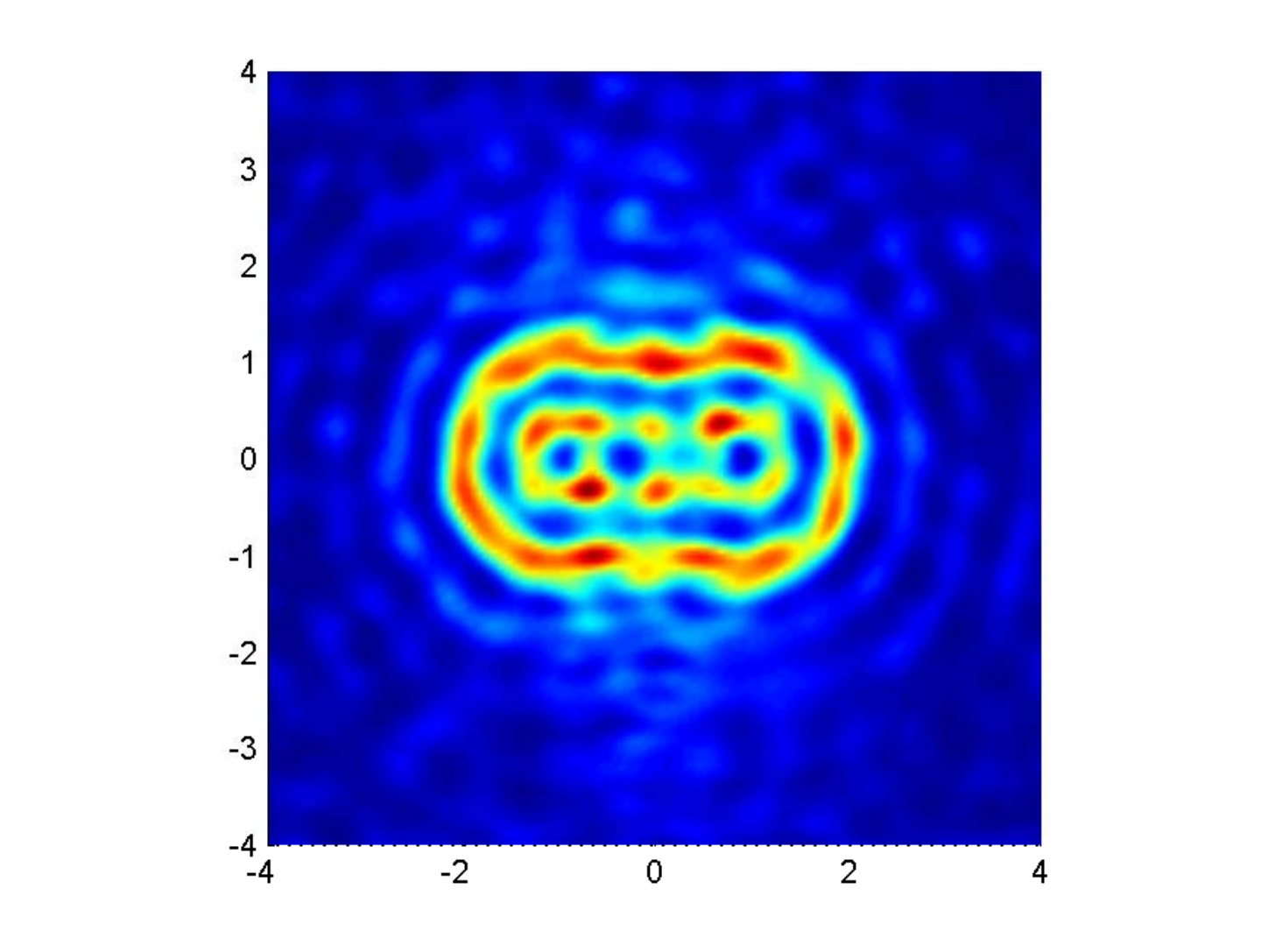}}
\caption{{\bf Example Soft.}\, Reconstruction of Peanut shaped domain with different noise.
}\label{softpeanut}
\end{figure}
\begin{figure}[htbp]
  \centering
  \subfigure[\textbf{Pear}]{
    \includegraphics[width=3in]{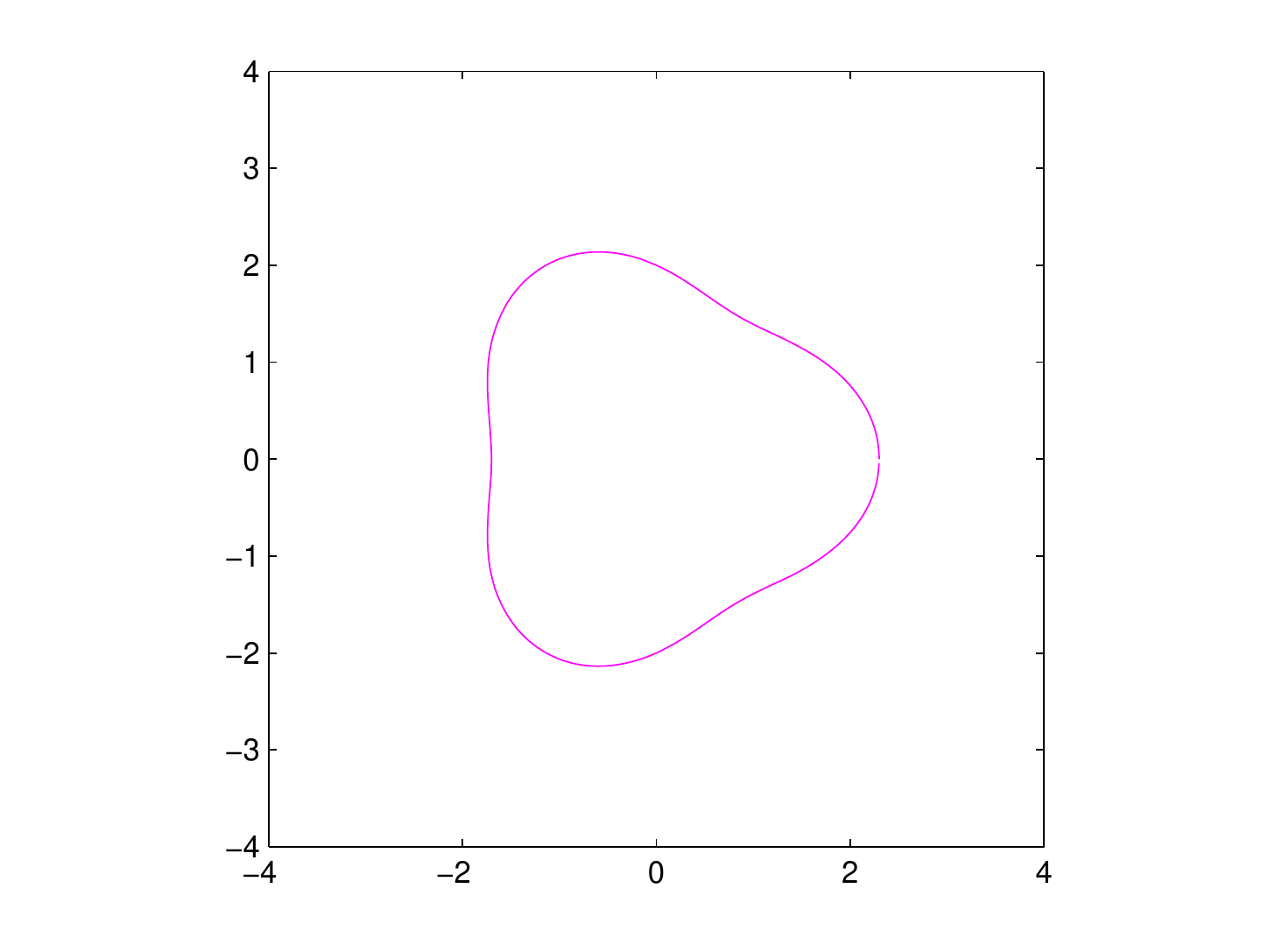}}
  \subfigure[\textbf{10\% noise}]{
    \includegraphics[width=3in]{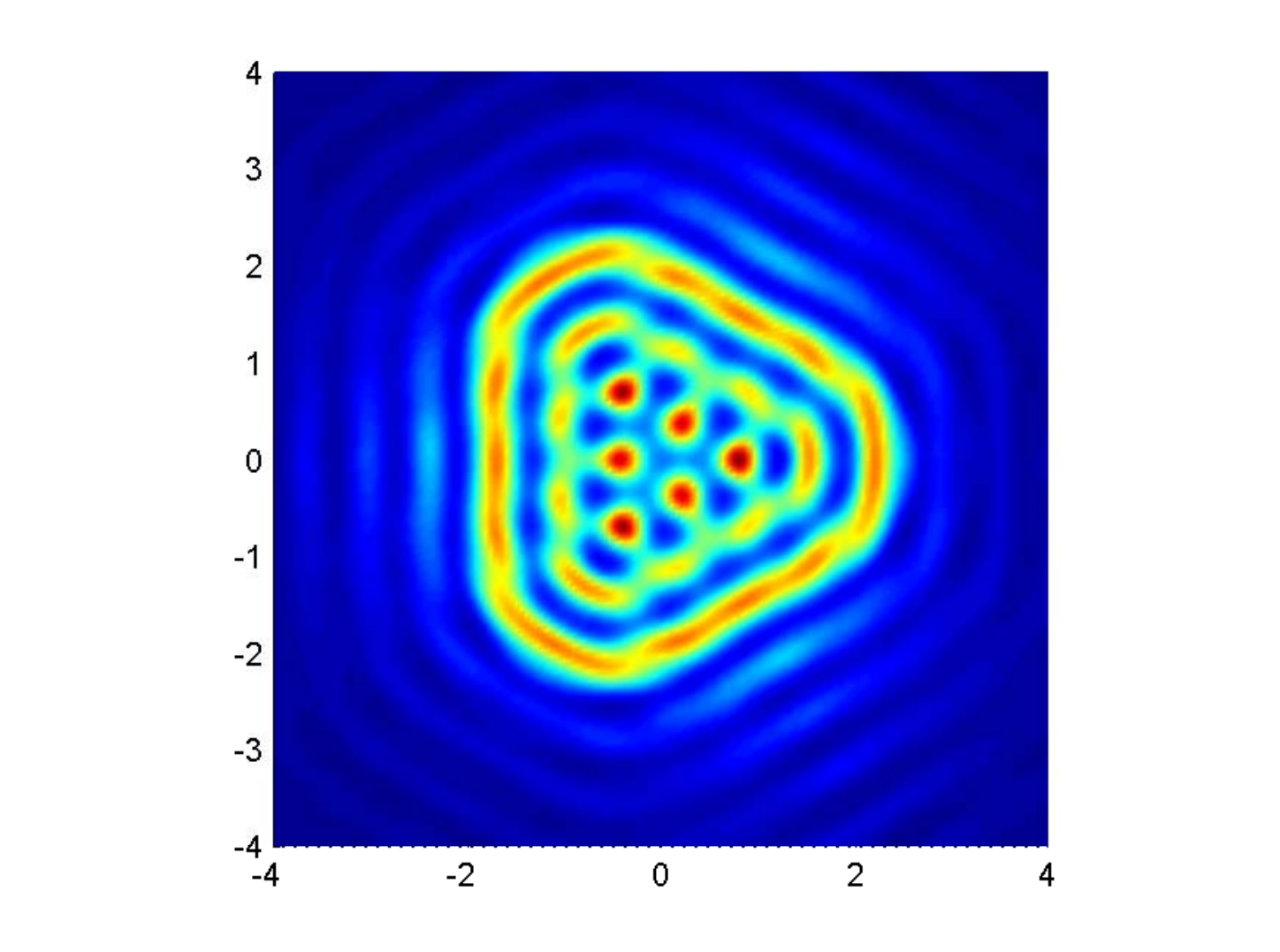}}
  \subfigure[\textbf{30\% noise}]{
    \includegraphics[width=3in]{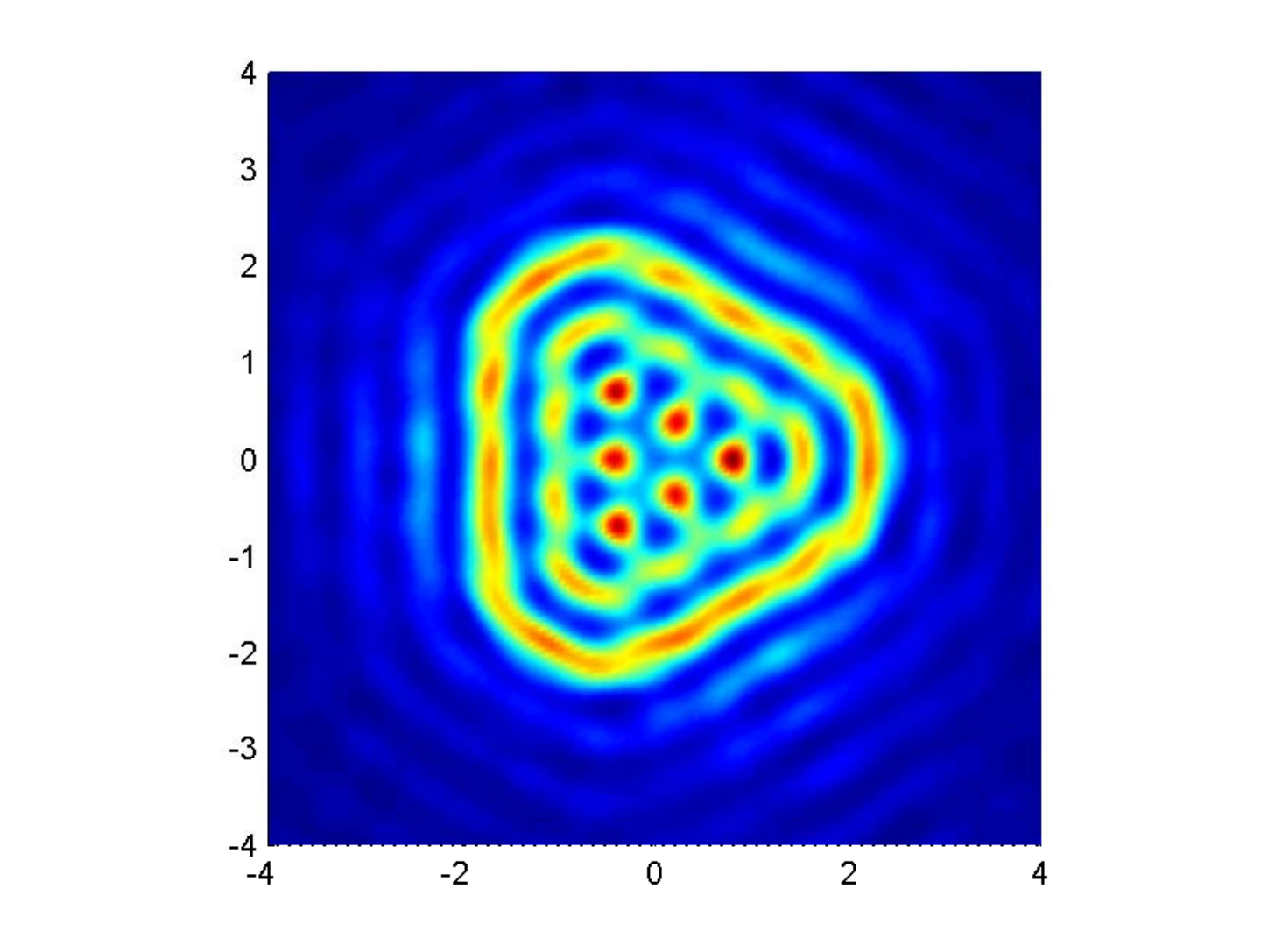}}
  \subfigure[\textbf{90\% noise}]{
    \includegraphics[width=3in]{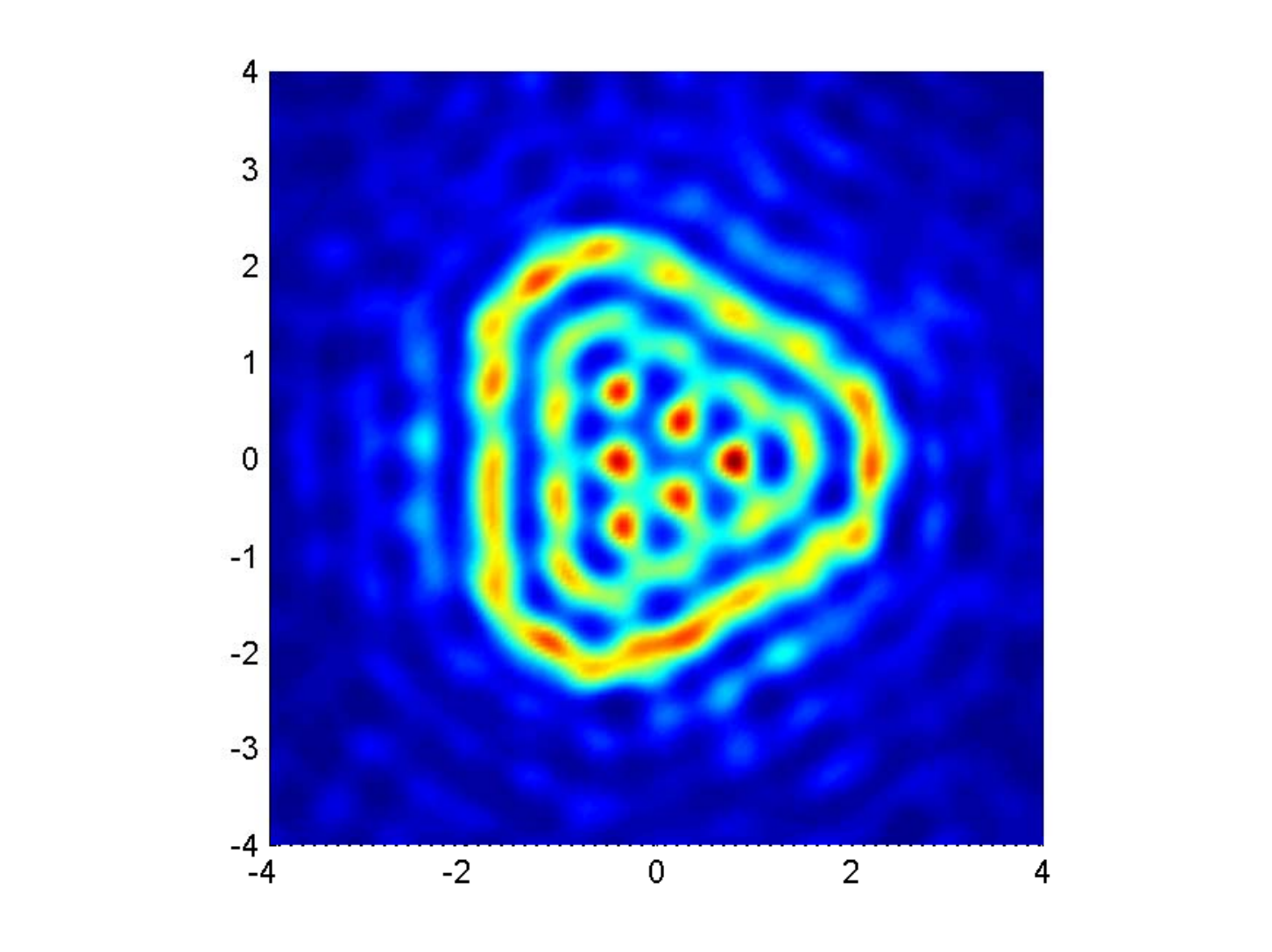}}
\caption{{\bf Example Soft.}\, Reconstruction of Pear shaped domain with different noise.
}\label{softpear}
\end{figure}

\medskip

\noindent

{\bf Example OtherPhyPro.}
For comparison, we consider the the other boundary conditions and the penetrable inhomogeneous medium. We take the pear shaped domain as an example.
Figure \ref{impedancepear} shows the results for impedance boundary conditions with different impedance functions $\la$.
Figure \ref{penetrablepear} shows the results for penetrable inhomogeneous medium with different contrast functions $q$.
From the reconstructions shown in Figures \ref{softpear}-\ref{penetrablepear}, we observe that the shape of the pear shaped domain can always be roughly captured,
it only changes slightly for different physical properties. This further show that our method is independent of the physical properties of the underlying scatterer.\\
\begin{figure}[htbp]
  \centering
  \subfigure[\textbf{$\lambda=0$}]{
    \includegraphics[width=3in]{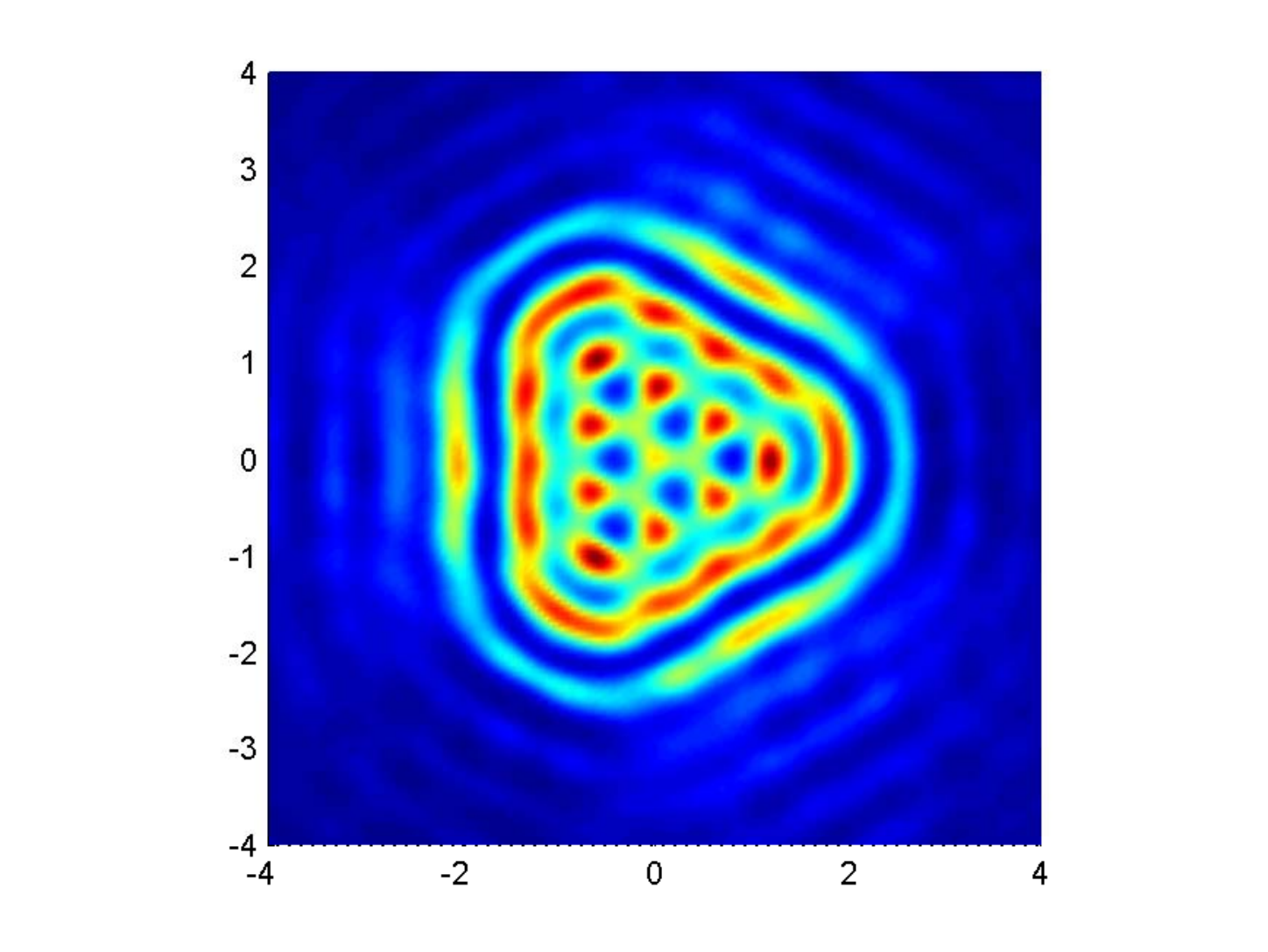}}
  \subfigure[\textbf{$\lambda=1$}]{
    \includegraphics[width=3in]{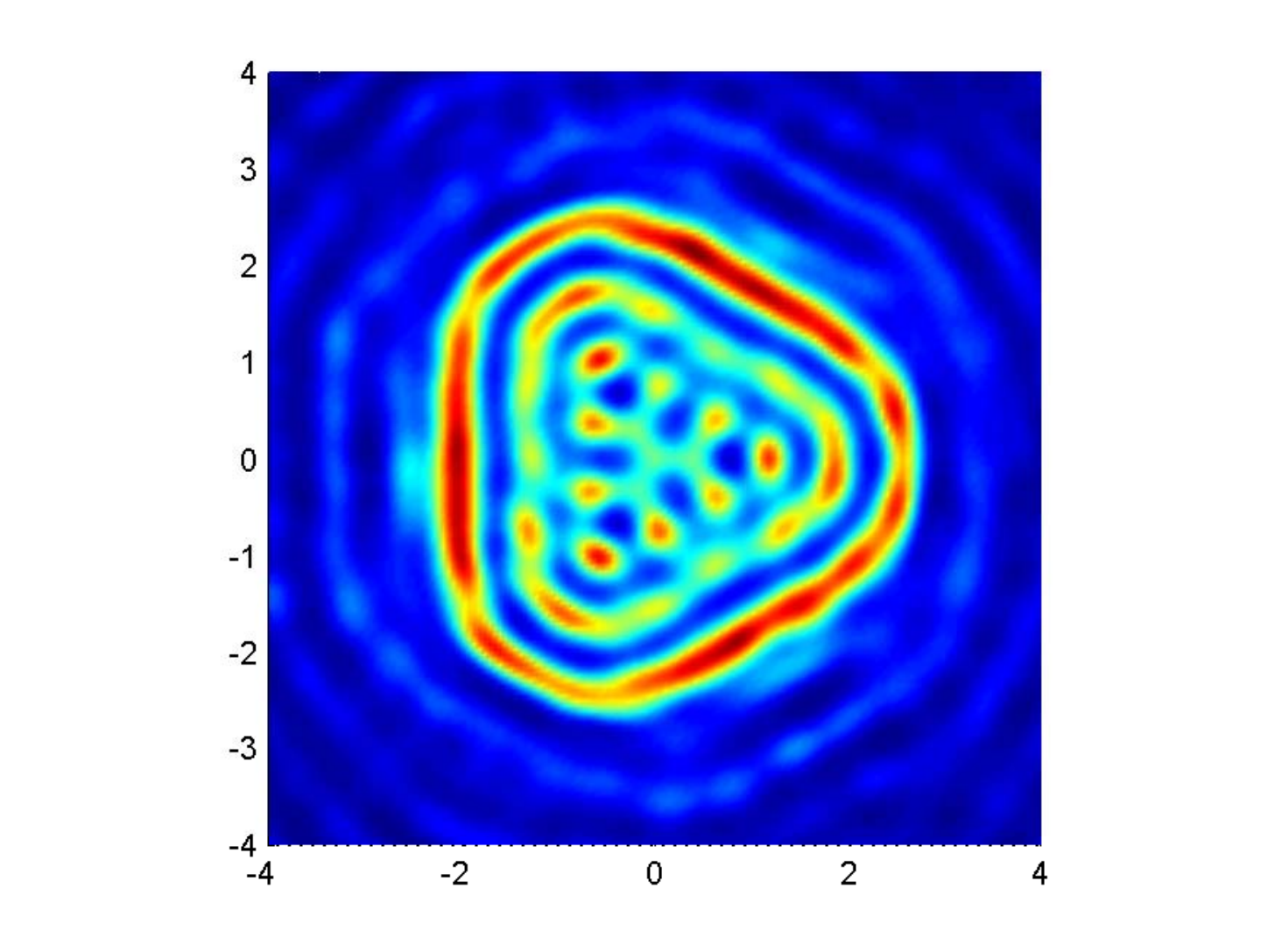}}
  \subfigure[\textbf{$\lambda=i$}]{
    \includegraphics[width=3in]{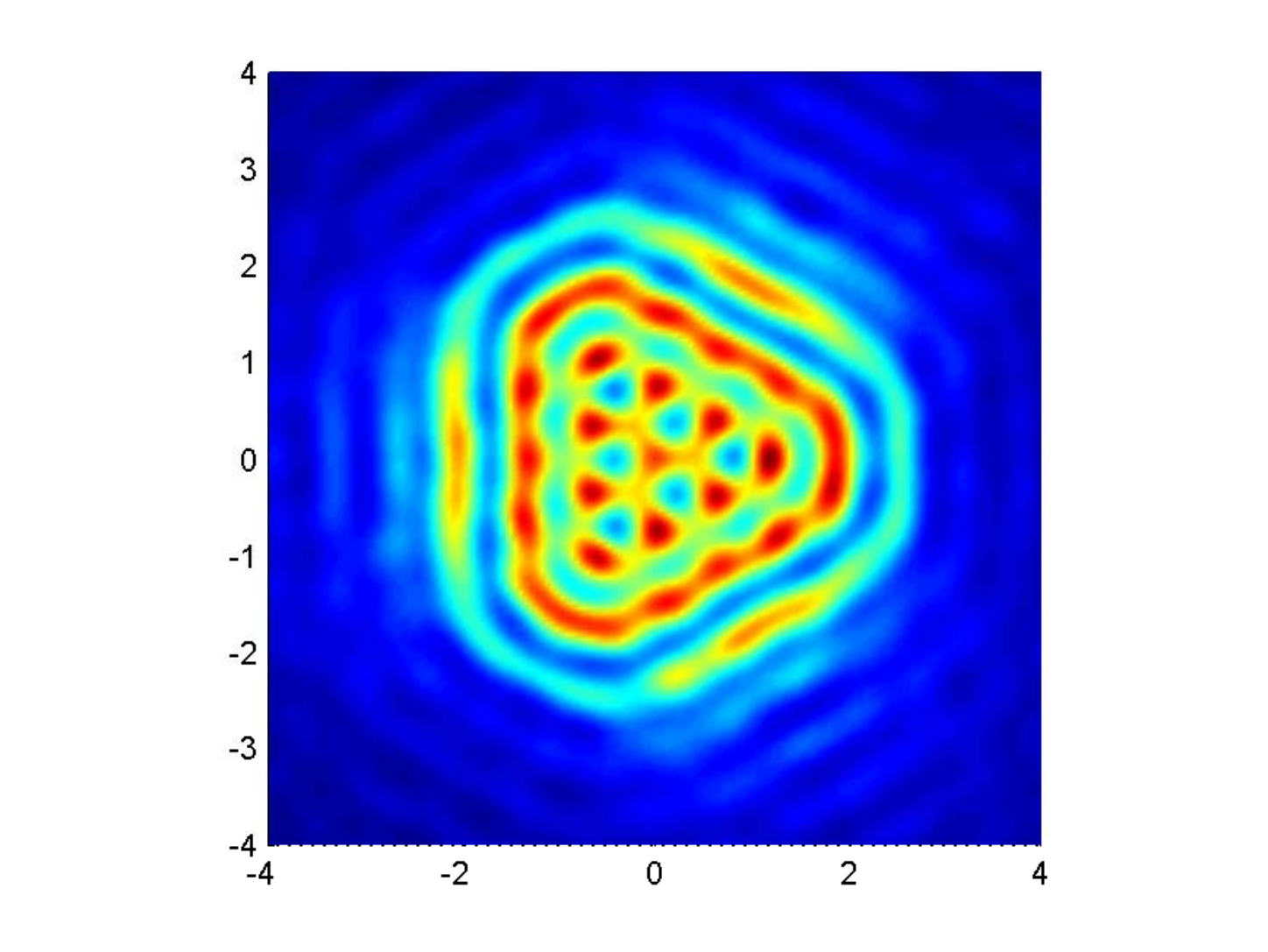}}
  \subfigure[\textbf{$\lambda=1+i$}]{
    \includegraphics[width=3in]{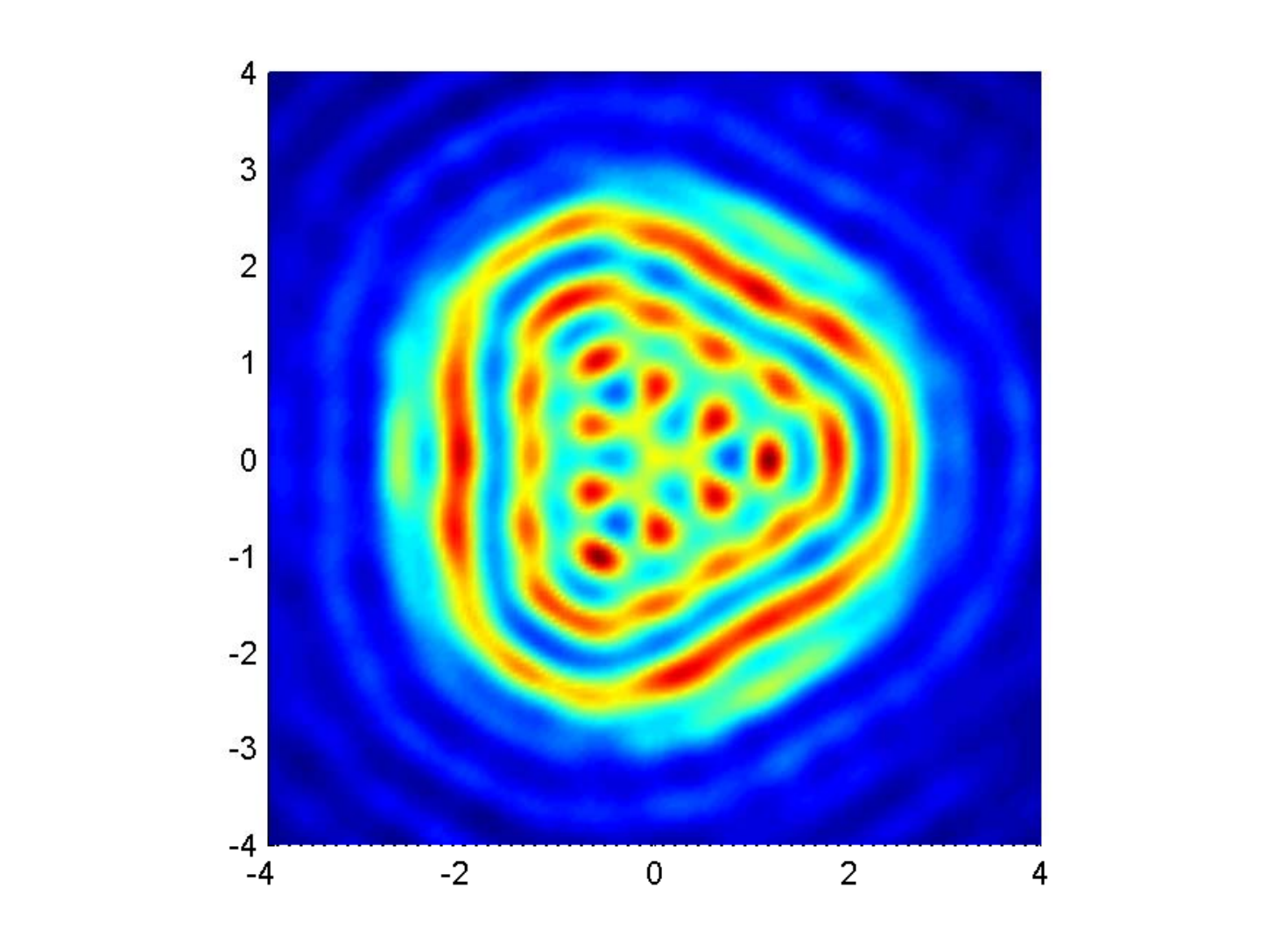}}
    \caption{{\bf Example OtherPhyPro.}\, Reconstruction of Pear shaped domain with different impedance conditions and $30\%$ noise.}
\label{impedancepear}
\end{figure}

\begin{figure}[htbp]
  \centering
  \subfigure[\textbf{$q=0.5$}]{
    \includegraphics[width=3in]{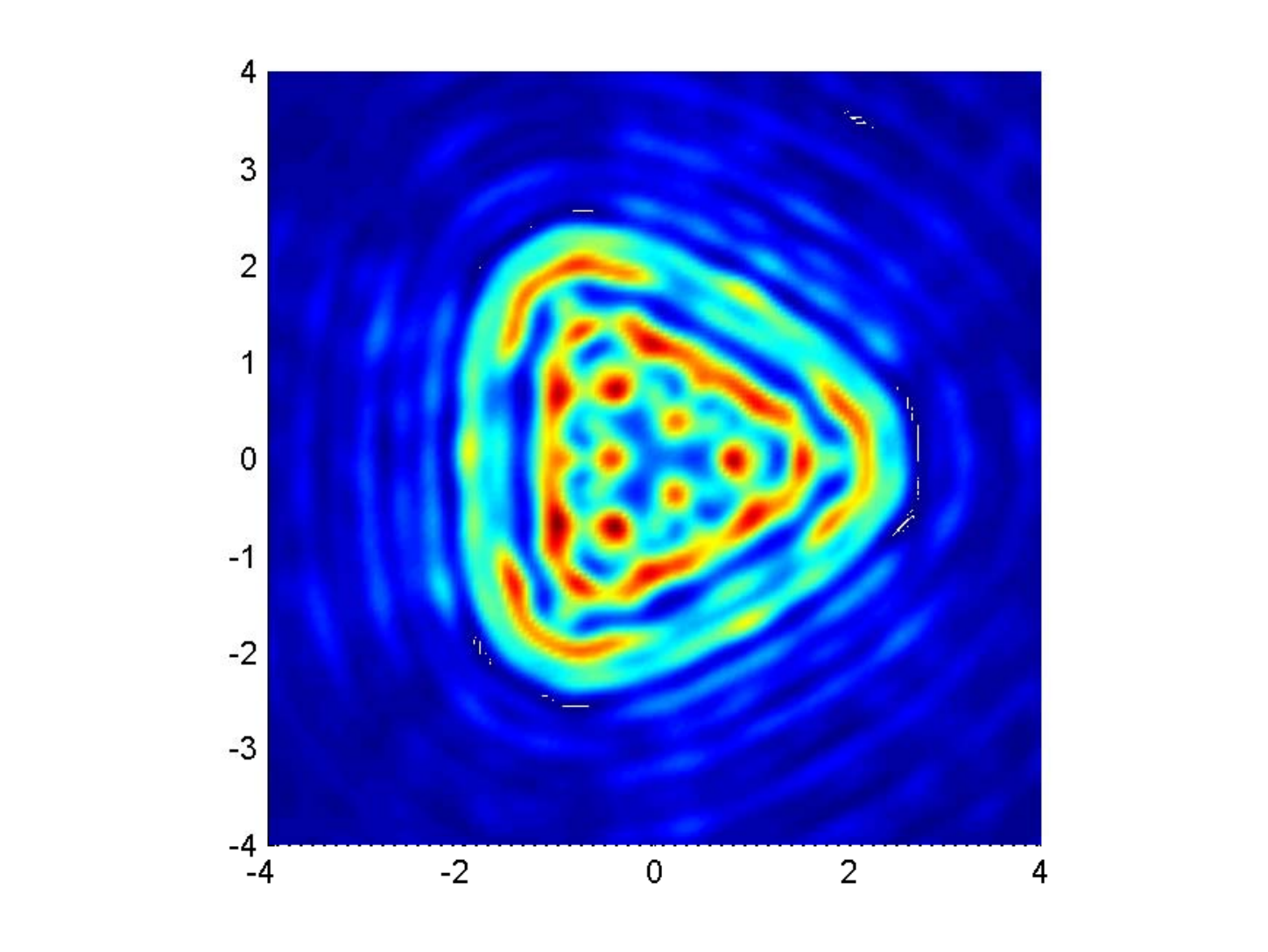}}
  \subfigure[\textbf{$q=0.5+0.5i$}]{
    \includegraphics[width=3in]{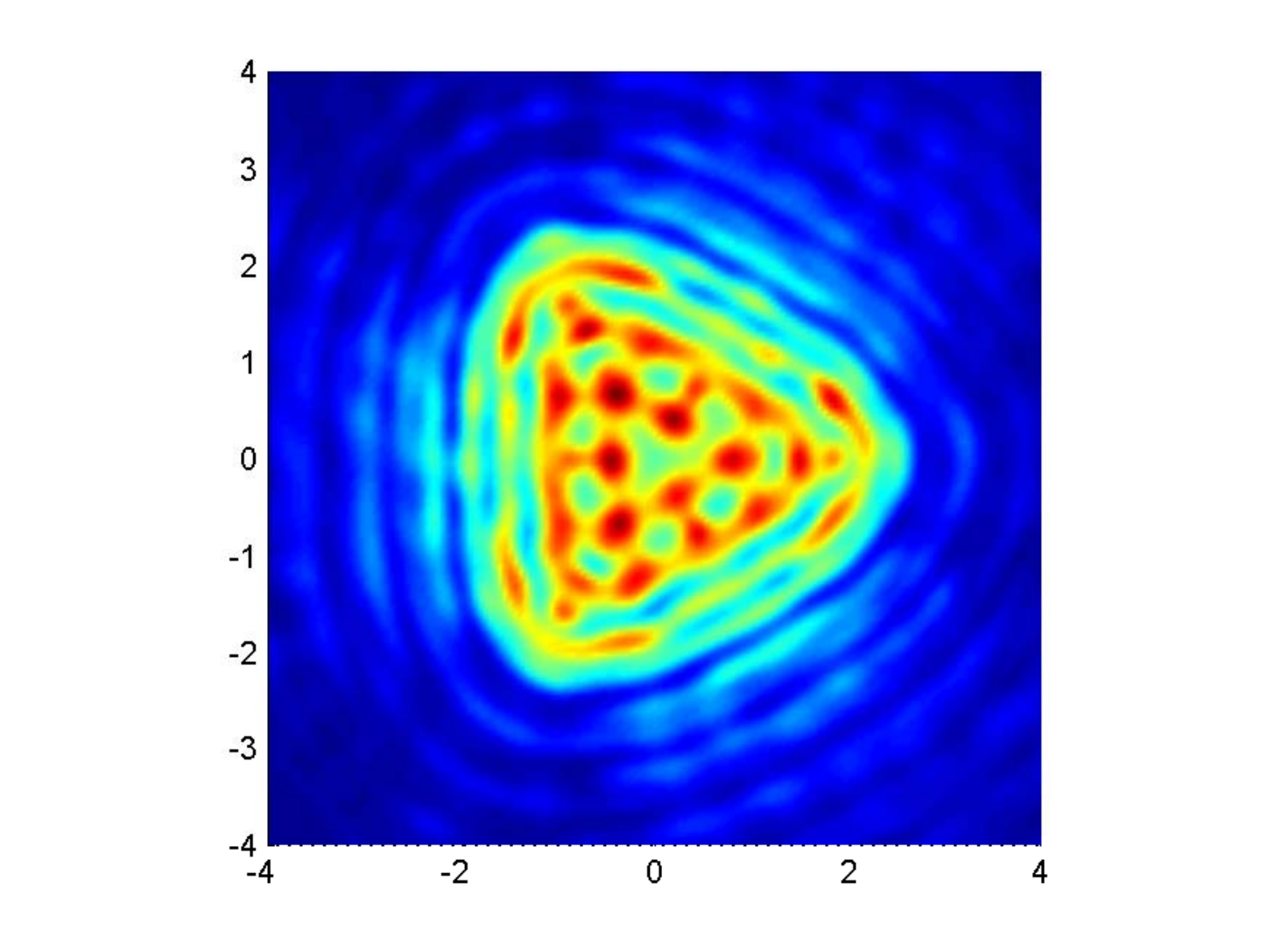}}
  \subfigure[\textbf{$q=-0.5+0.5i$}]{
    \includegraphics[width=3in]{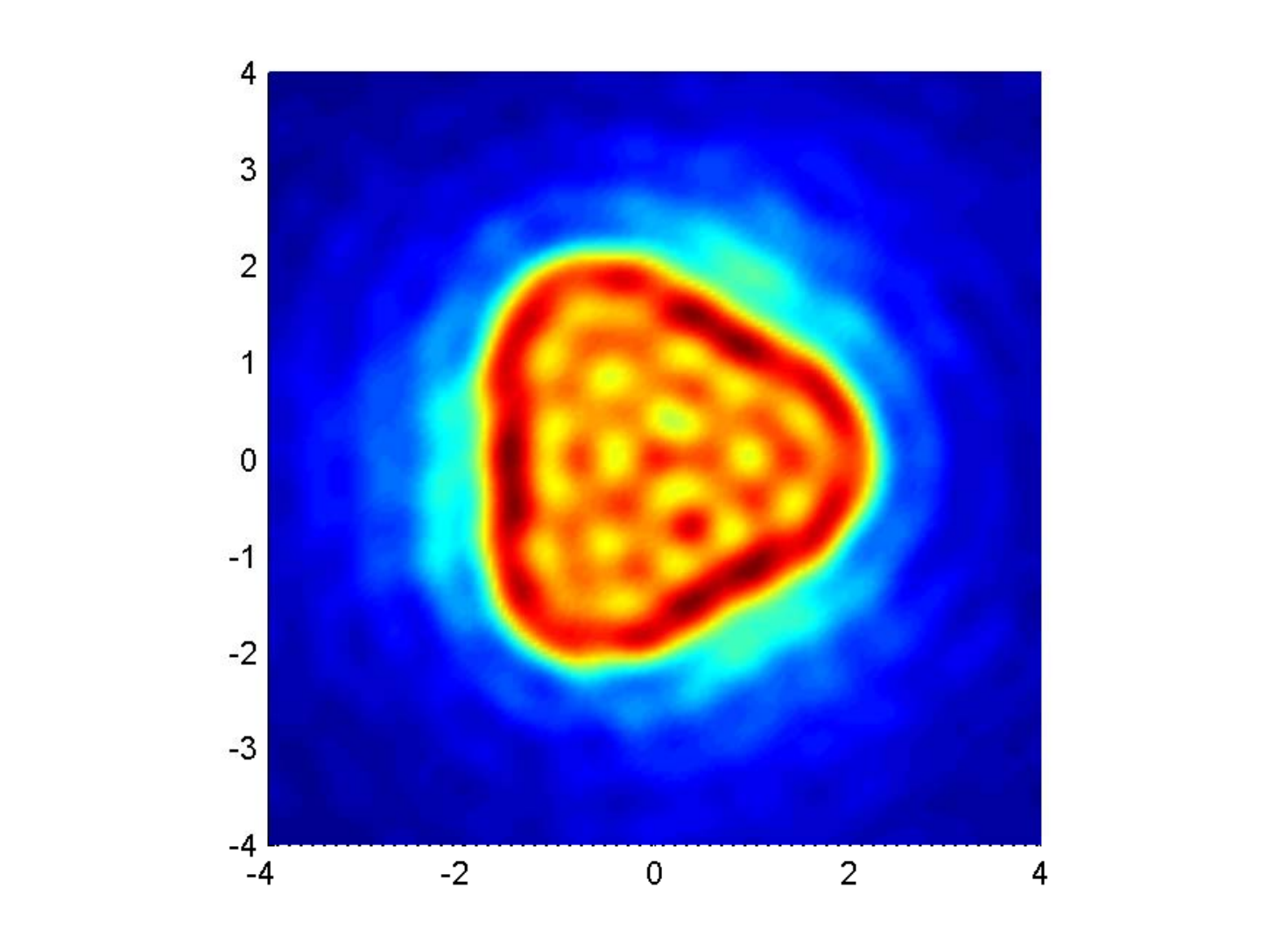}}
  \subfigure[\textbf{$q=-0.5$}]{
    \includegraphics[width=3in]{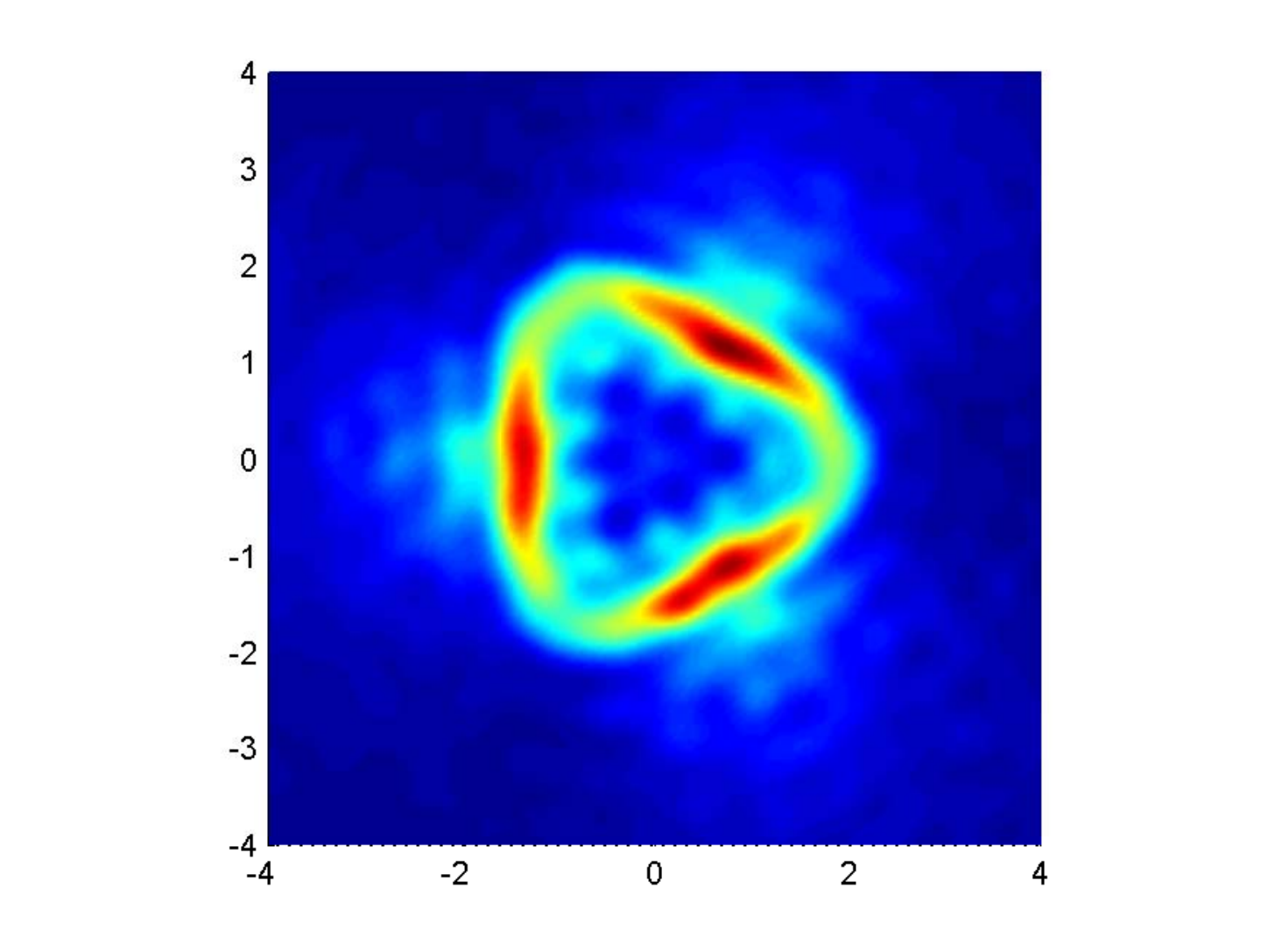}}
    \caption{{\bf Example OtherPhyPro.}\, Reconstruction of Pear shaped domain with different inhomogeneities and $30\%$ noise.}
\label{penetrablepear}
\end{figure}

\medskip

\noindent

{\bf Example MixedType.}
We consider the scattering by a scatterer with two disjoint components. We set the scatterer to be a combination of a peanut shaped domain centered at
$(-3,3)$ and a kite shaped domain centered at $(3,-3)$. Further, we impose different physical properties on each component.
The search domain is the rectangle $[-7,7]\times[-7,7]$ with $301\times 301$ equally spaced sampling points.
Figure \ref{multiscalar1} shows the reconstructions for scatterers with different physical property components.

\begin{figure}[htbp]
  \centering
  \subfigure[\textbf{Original domain}]{
    \includegraphics[width=3in]{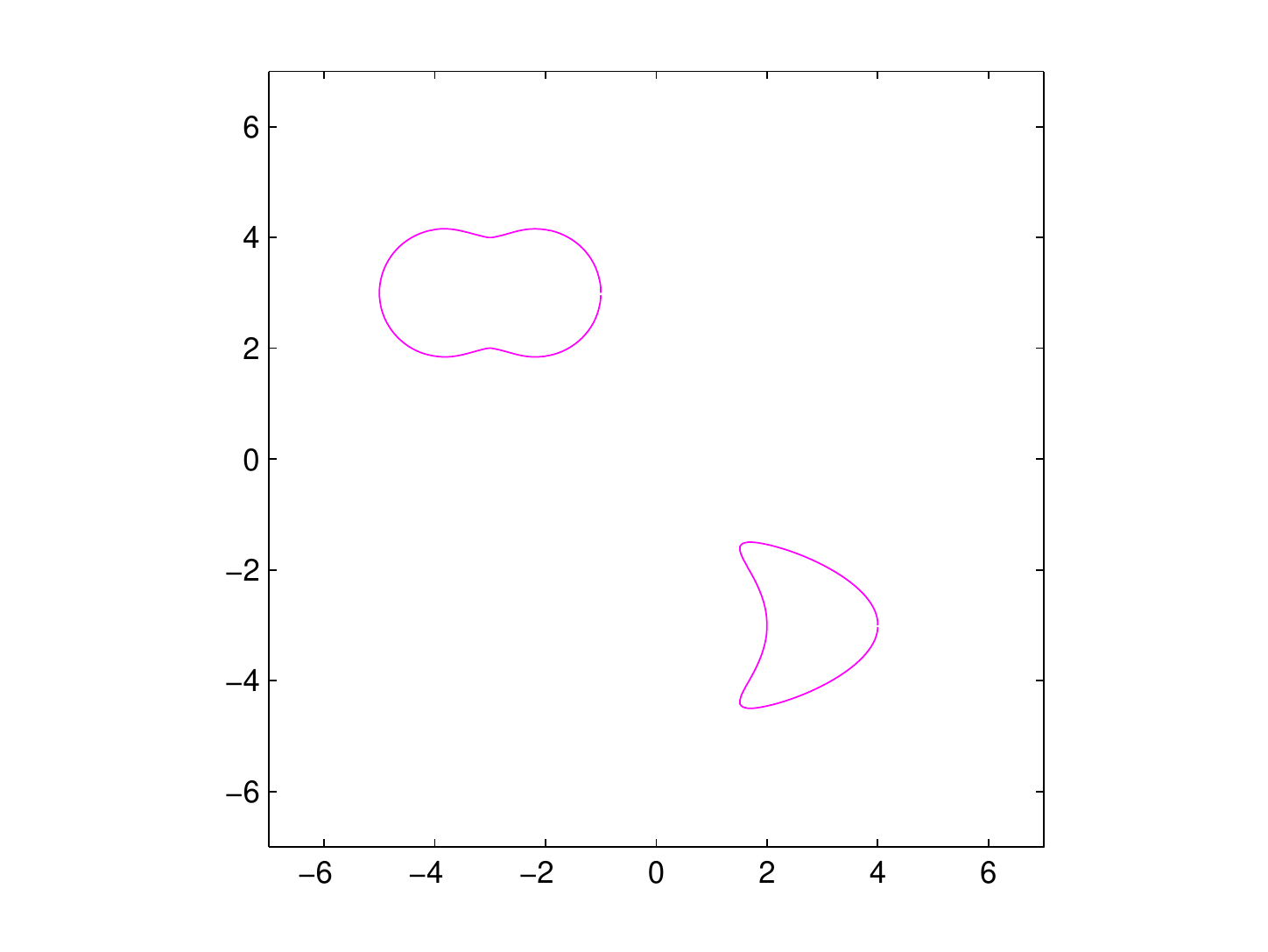}}
  \subfigure[\textbf{Dirichlet Kite, Neumann Peanut}]{
    \includegraphics[width=3in]{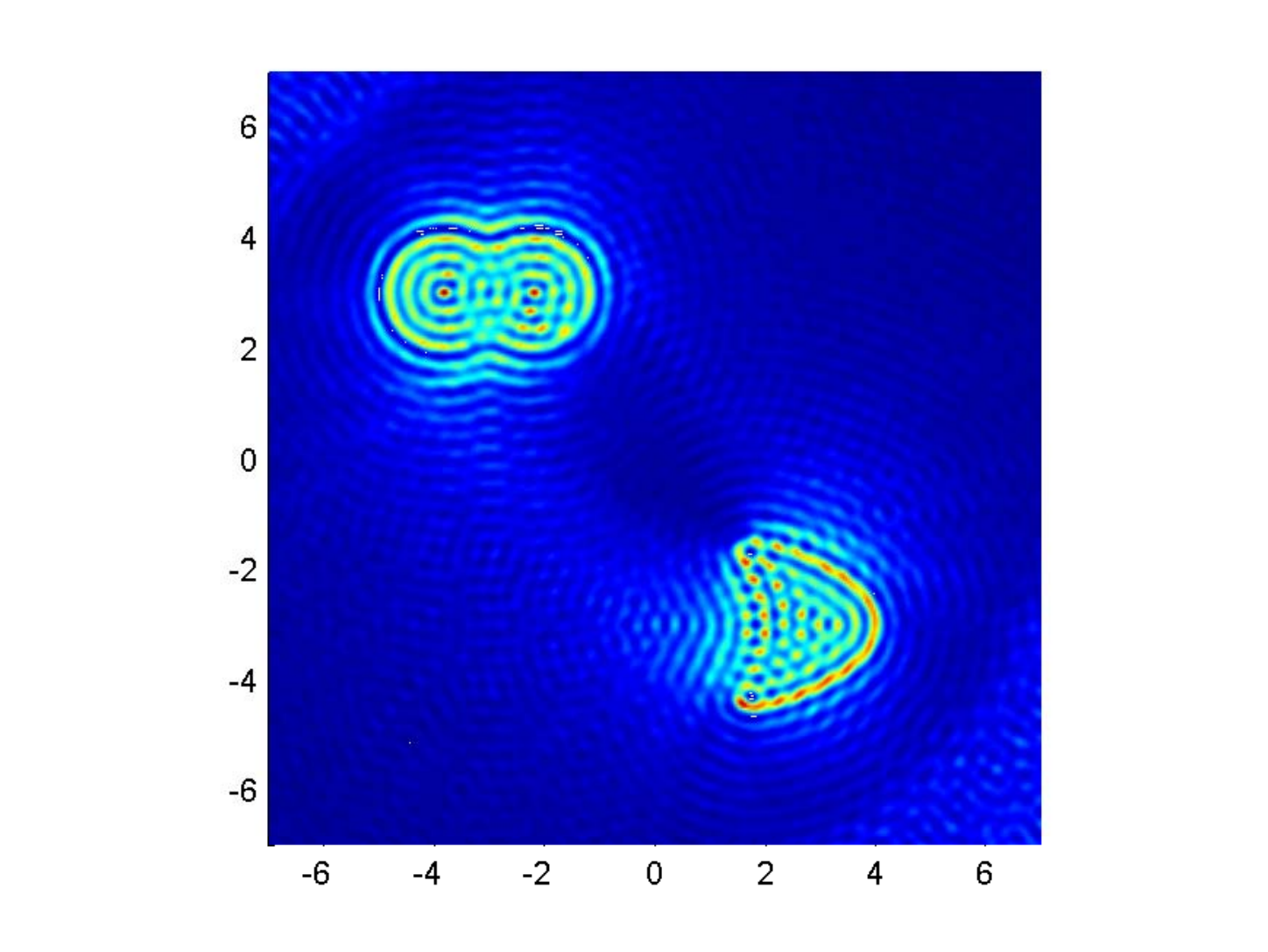}}
  \subfigure[\textbf{Dirichlet Peanut, Neumann Kite}]{
    \includegraphics[width=3in]{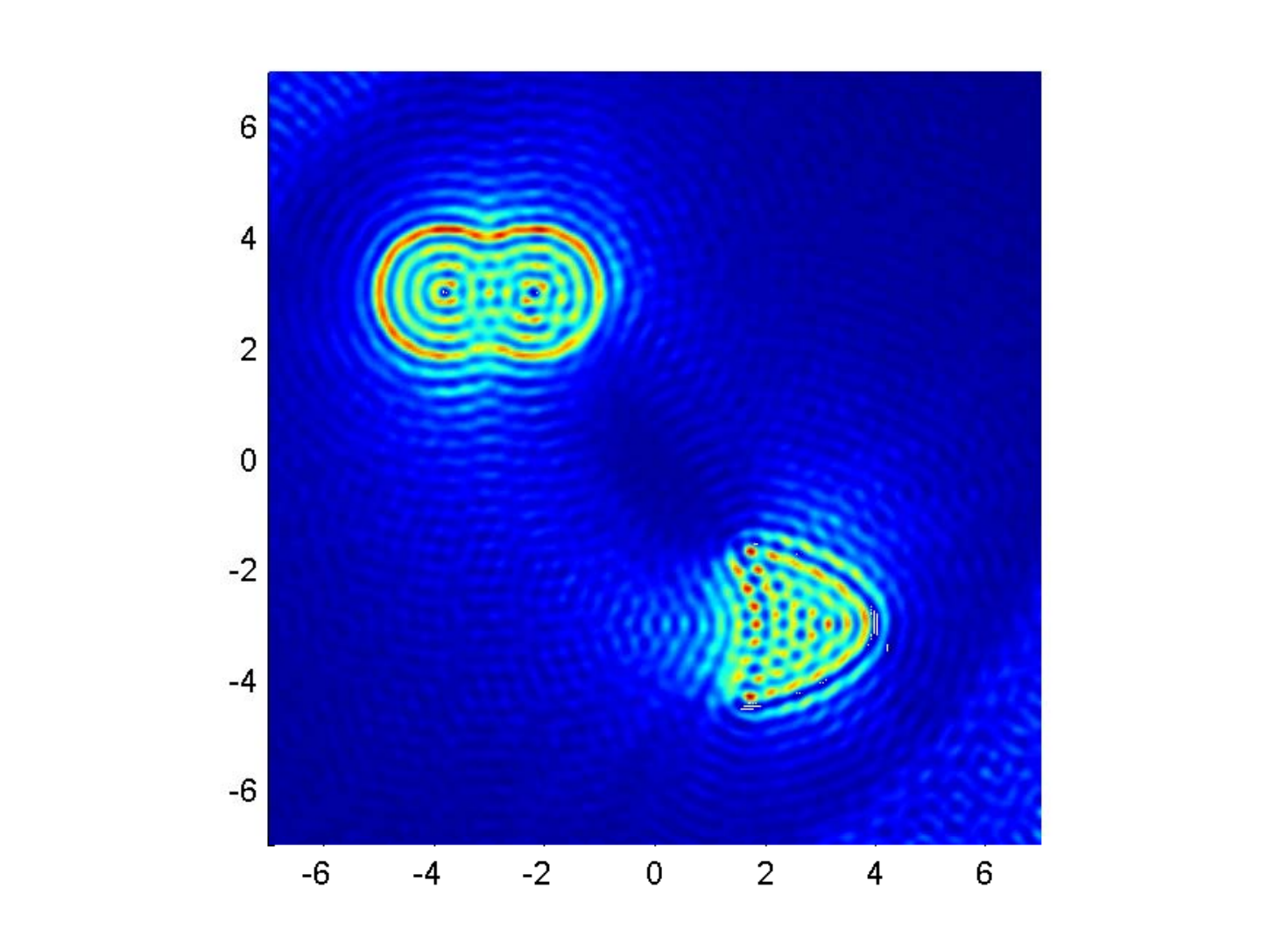}}
  \subfigure[\textbf{Dirichlet Peanut, Penetrable Kite}]{
    \includegraphics[width=3in]{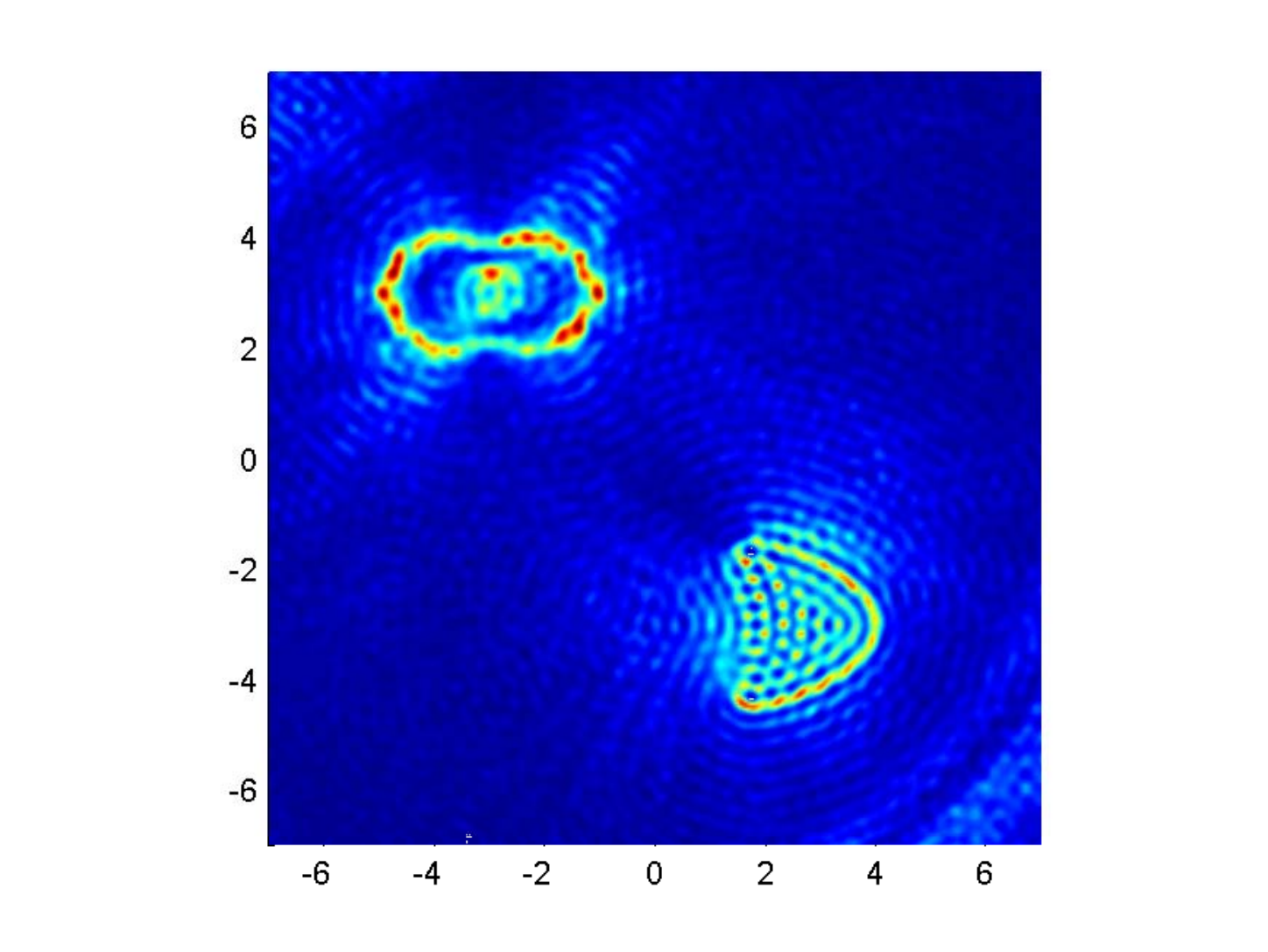}}
    \caption{{\bf Example MixedType.}\, Reconstruction of mixed type scatterers with $30\%$ noise. We take $k=10$, $N=128$, $q=-0.5$.}
\label{mixedtype}
\end{figure}

\medskip

\noindent

{\bf Example MultiScalar.}
In this example, the underlying scatterer is a combination of a big pear shaped domain and a small disk with radius $r=0.1$ or $r=0.2$.
The Dirichlet boundary condition is imposed on the boundary  of the scatterer.
The research domain is $[-6,6]\times[-6,6]$ with $301\times 301$ equally spaced sampling points.
The scattering amplitude is collected in 256 observation directions and 256 incident directions.
We observe, from Figures \ref{multiscalar1} and \ref{multiscalar2}, that the boundary of the scatterer is clearly located.
Figures \ref{multiscalar1} and \ref{multiscalar2} also show that the resolution can be improved with higher wave number.
\\
\begin{figure}[htbp]
  \centering
  \subfigure[\textbf{Original domain}]{
    \includegraphics[width=2in]{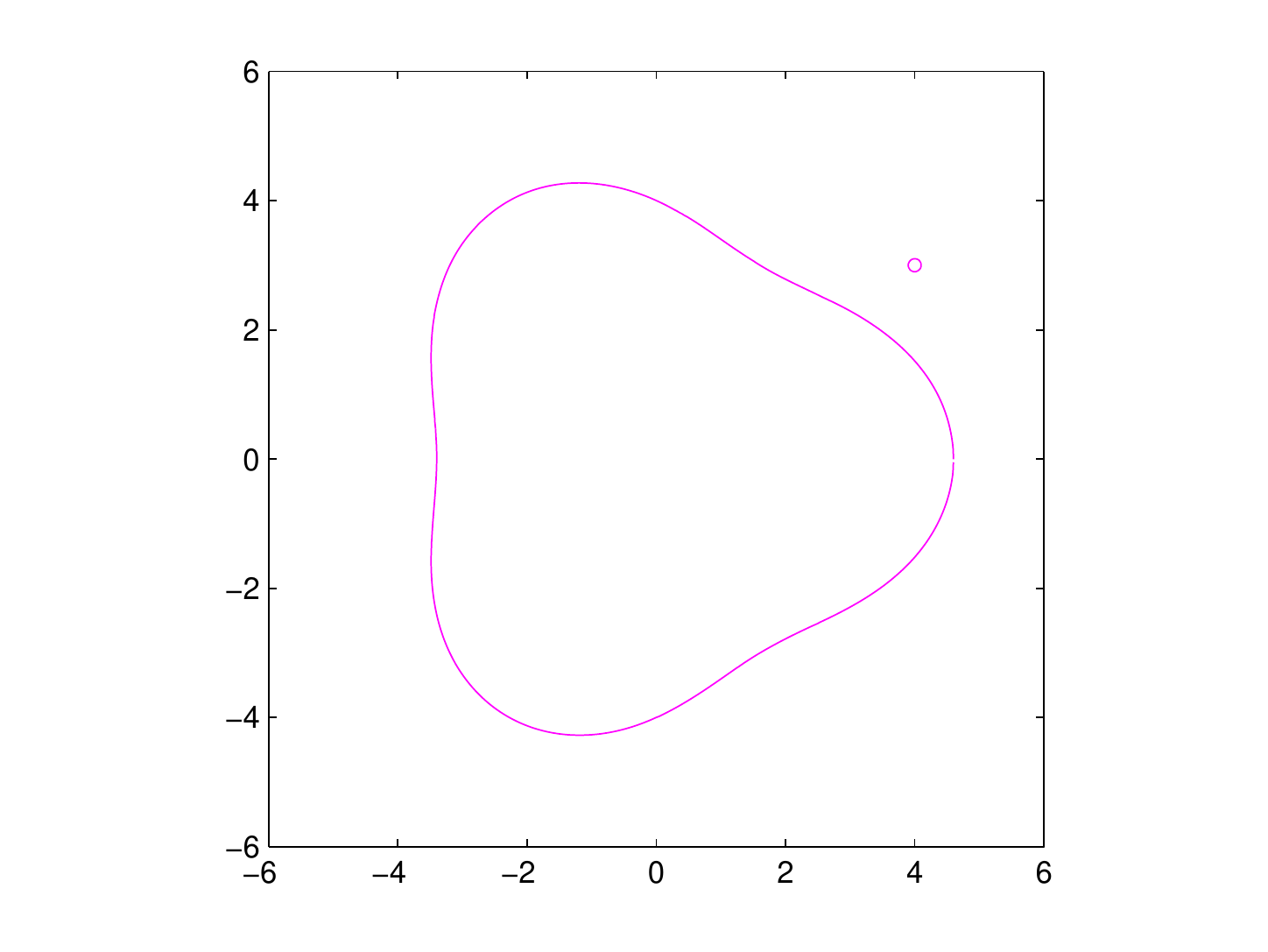}}
  \subfigure[\textbf{$k=5$}]{
    \includegraphics[width=2in]{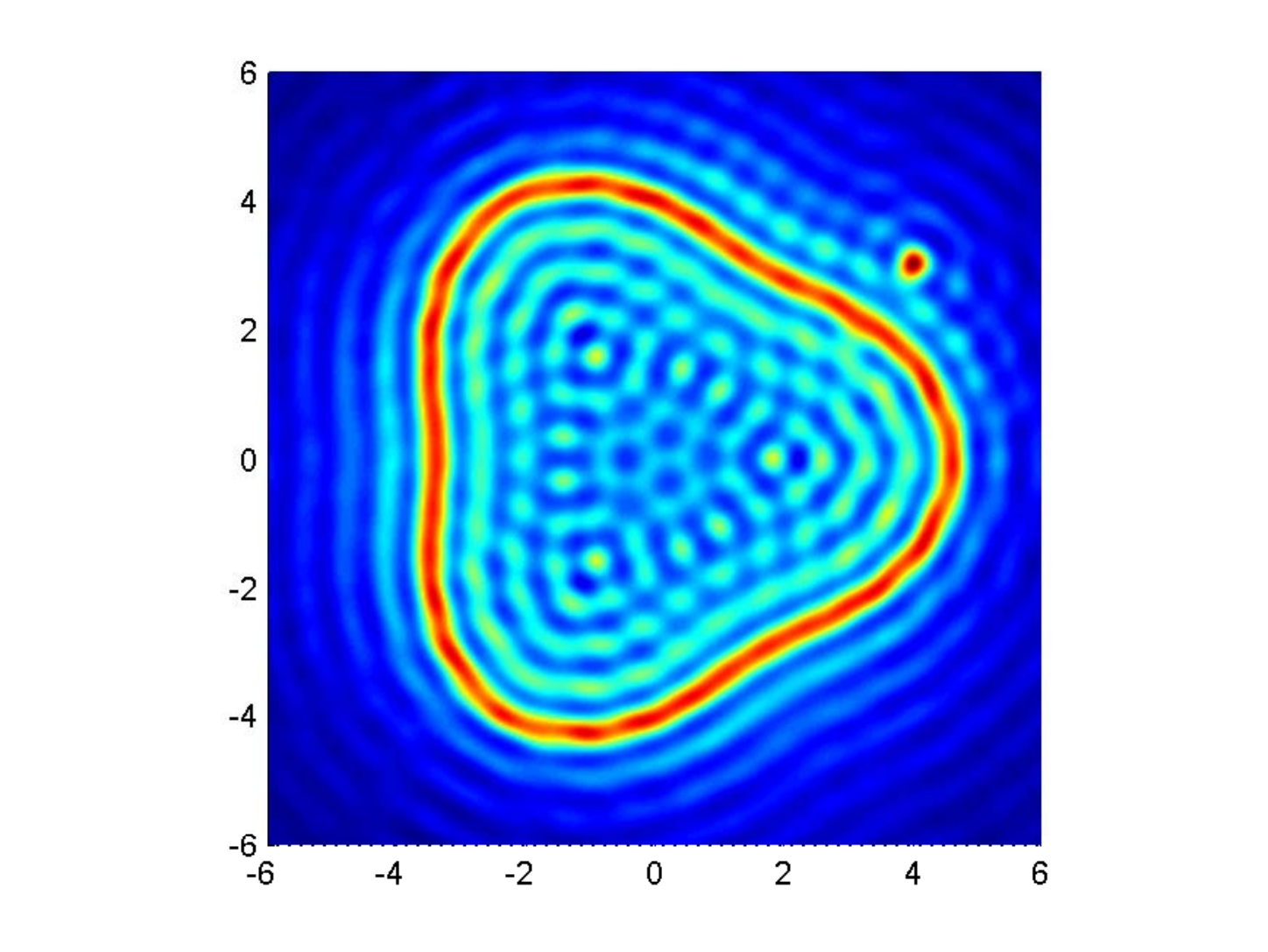}}
  \subfigure[\textbf{$k=10$}]{
    \includegraphics[width=2in]{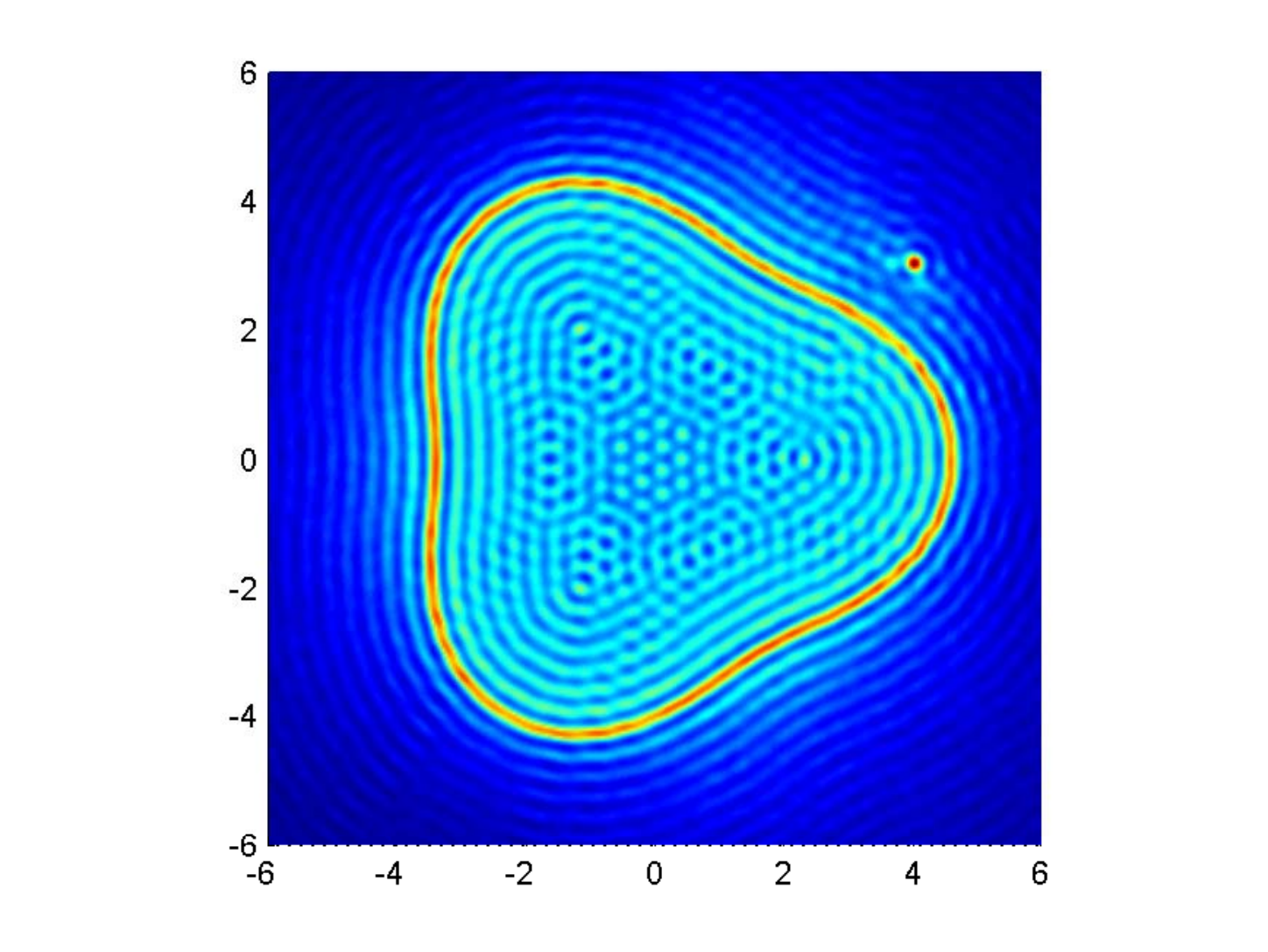}}
\caption{{\bf Example MultiScalar.}\, Reconstruction of a big pear shaped domain and a small disk with radius  $r=0.1$ and 30\% noise. }
\label{multiscalar1}
\end{figure}

\begin{figure}[htbp]
  \centering
  \subfigure[\textbf{Original domain}]{
    \includegraphics[width=2in]{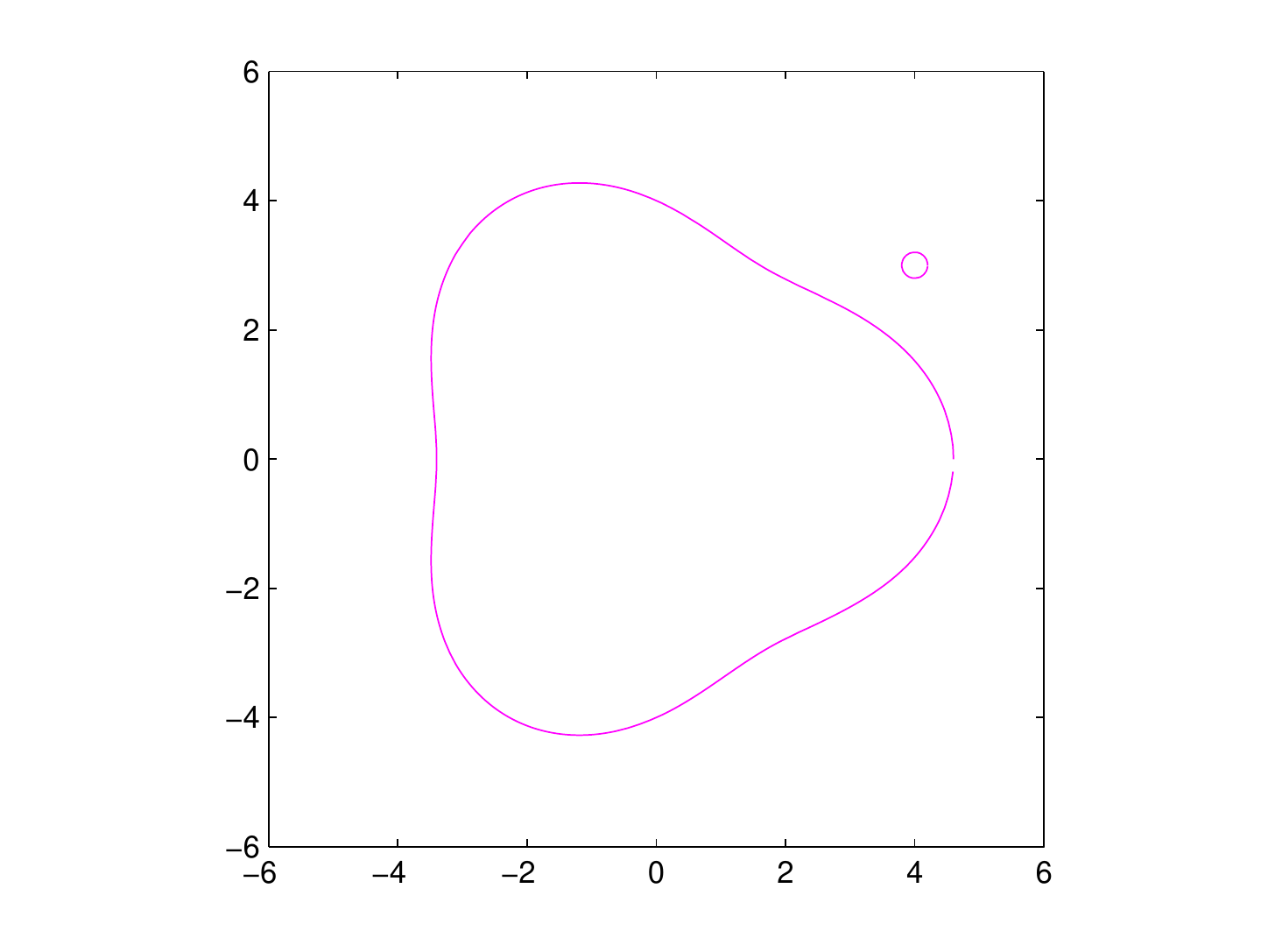}}
  \subfigure[\textbf{$k=5$}]{
    \includegraphics[width=2in]{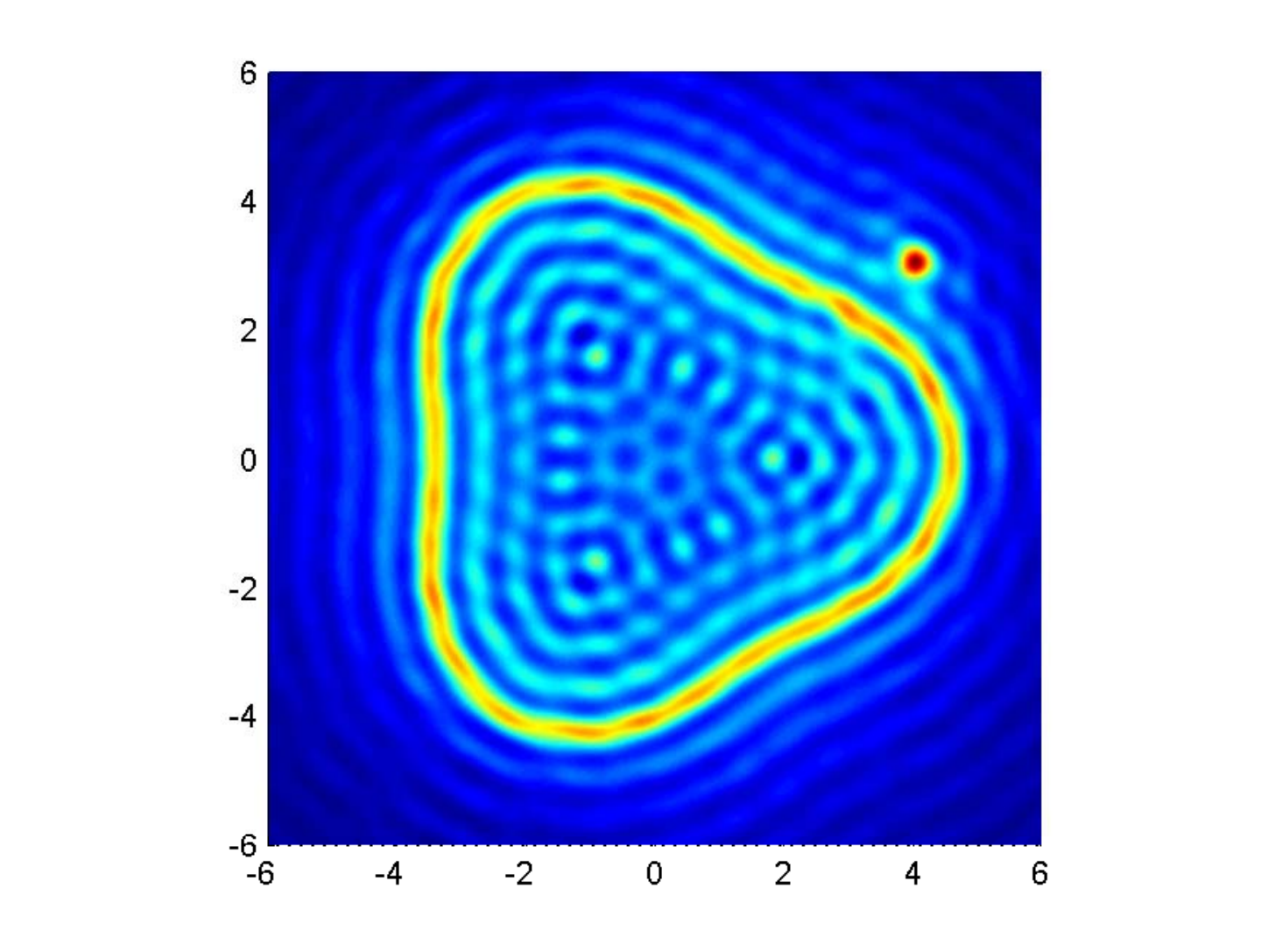}}
  \subfigure[\textbf{$k=10$}]{
    \includegraphics[width=2in]{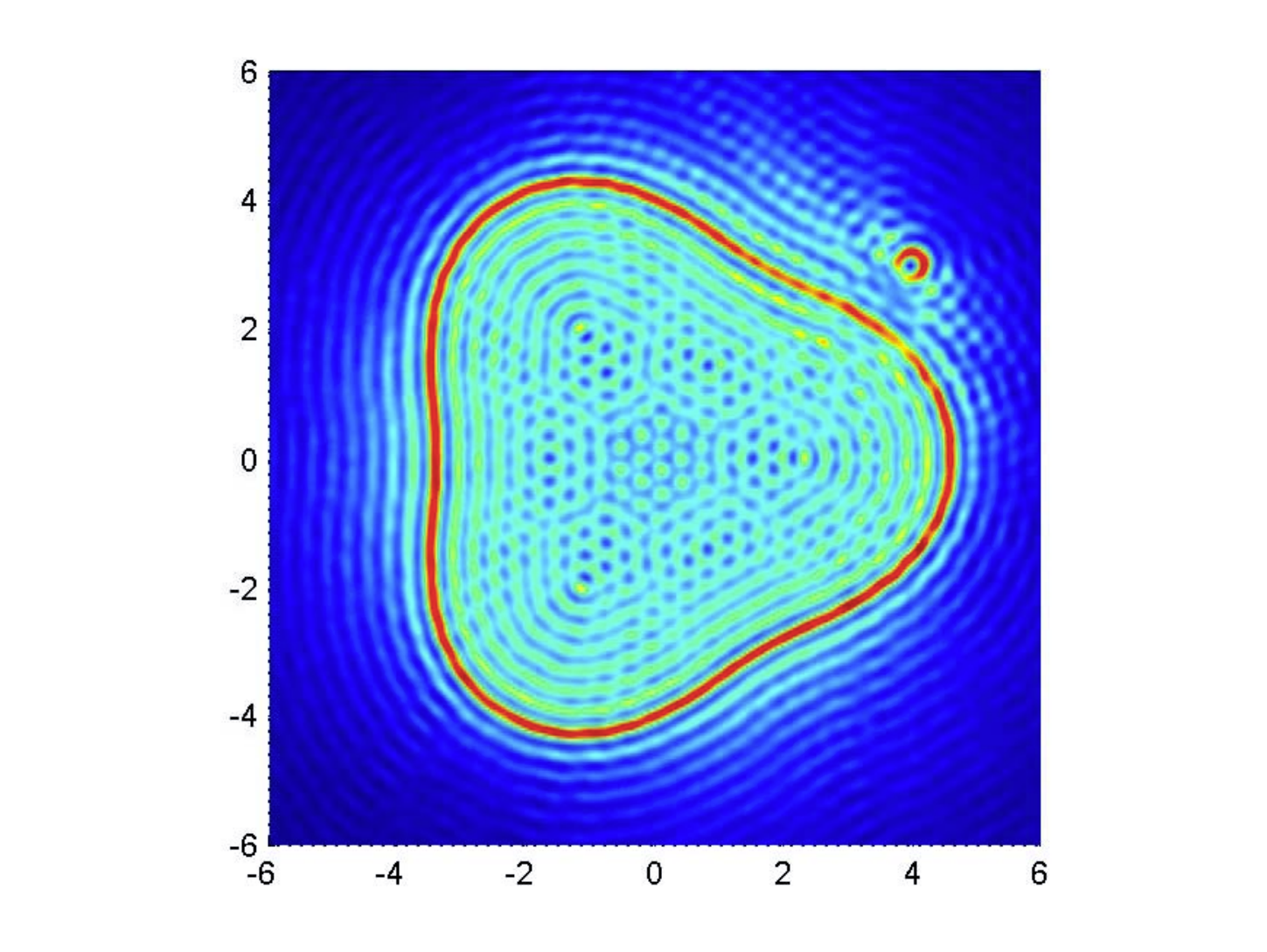}}
\caption{{\bf Example MultiScalar.}\, Reconstruction of a big pear shaped domain and a small disk with radius  $r=0.2$ and 30\% noise. }
\label{multiscalar2}
\end{figure}

\medskip

\noindent

{\bf Example ResolutionLimit.}
As shown in Figure \ref{resolutionlimit}(a), the underling scatterer $\Om$ is given as the union of two disjoint obstacles
$\Om=\Om_{1}\cup\Om_{2}$ where $\Om_{1}$ is a disk with radius $r=4$ and center $(a,b)=(-1,0)$, while $\Om_{2}$ is
kite shaped domain centered at $(a,b)=(5,0)$. These two disjoint components are very close to each other.
Again, we impose Dirichlet boundary condition on the boundary $\pa \Om$.
The search domain is the rectangle $[-7,7]\times[-7,7]$ with $301\times 301$ equally spaced sampling points.
To make the gap available, we have used 360 incident directions and 360 observation directions in our simulations.
The distance between these two scatterers is about 0.8.
The relation between the wave number $k$ and the probe wavelength $l$ is given by $kl=2\pi$.
We test two wave numbers $k=4$ and $k=8$, which means that the corresponding wavelengths are $l=2\pi/4\approx1.6$ and $l=2\pi/8\approx0.8$.
The results are shown in Figure \ref{resolutionlimit}, from which we observe that the gap appears clearly even if it is just about half a wavelength and gets more visible by
using a higher wave number.\\

\begin{figure}[htbp]
  \centering
  \subfigure[\textbf{Original domain}]{
    \includegraphics[width=2in]{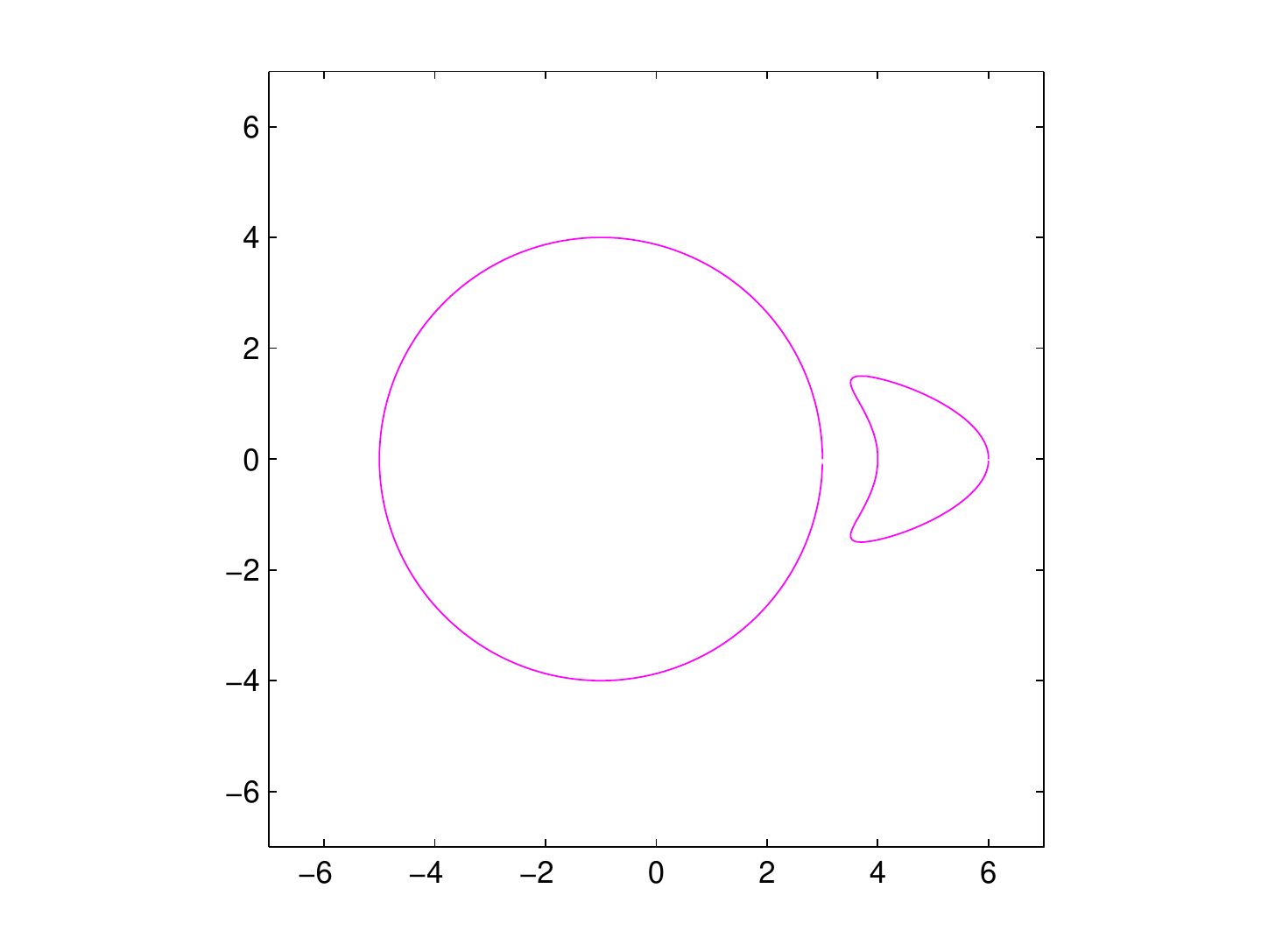}}
  \subfigure[\textbf{$k=4$}]{
    \includegraphics[width=2in]{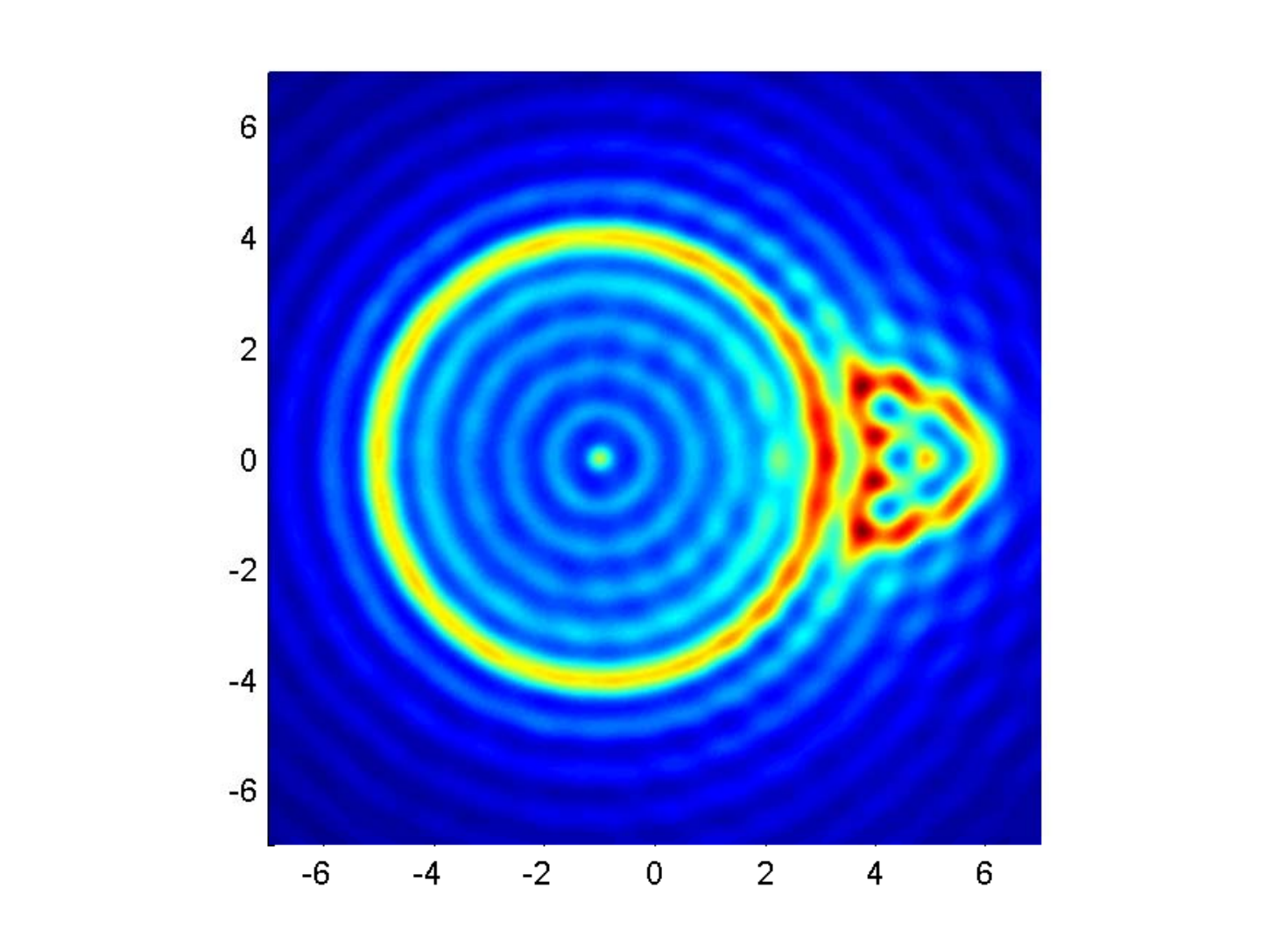}}
  \subfigure[\textbf{$k=8$}]{
    \includegraphics[width=2in]{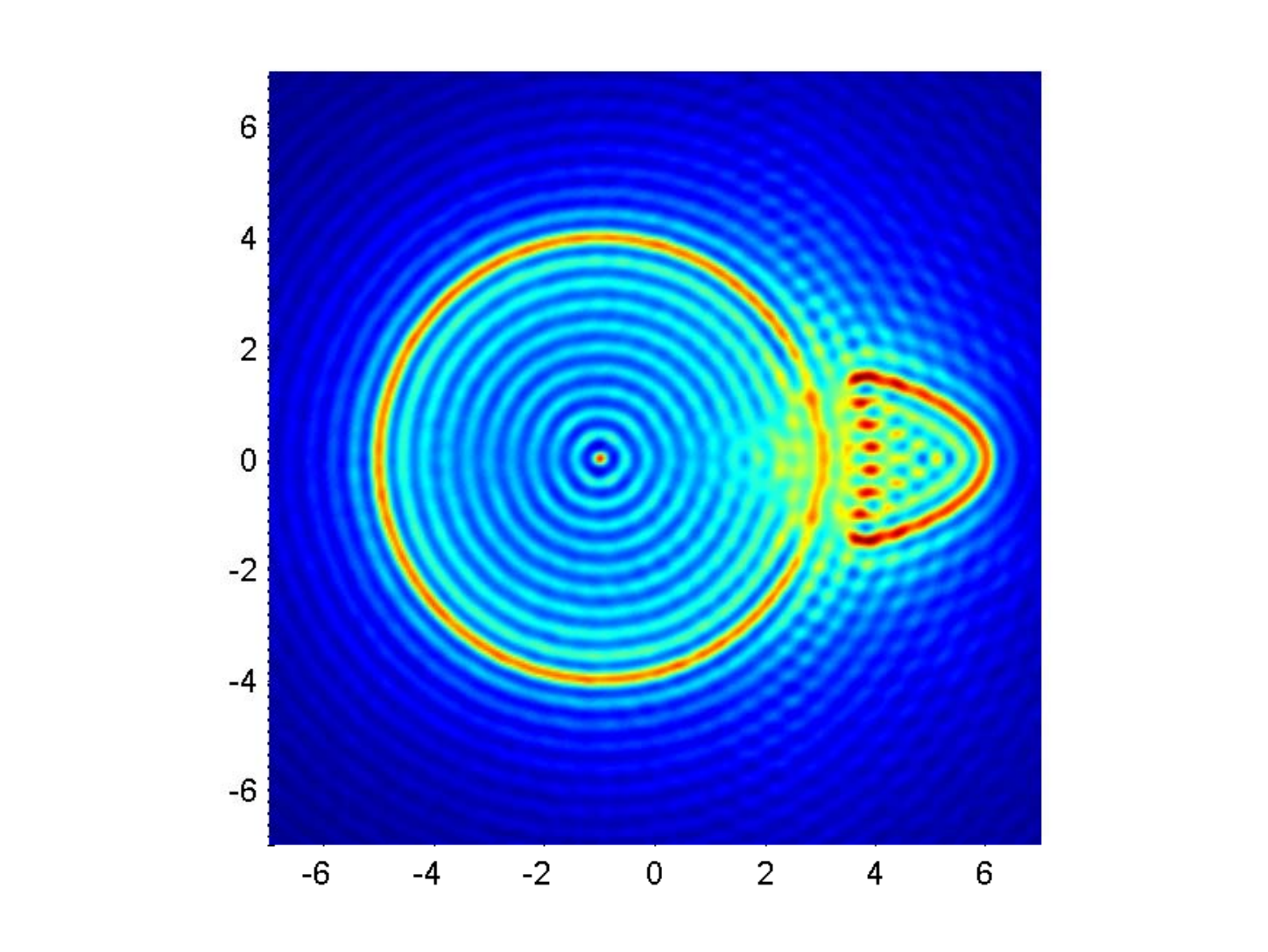}}
\caption{{\bf Example ResolutionLimit.}\, Reconstruction of two objects that are close to each other with $30\%$ noise by using different wave numbers.}
\label{resolutionlimit}
\end{figure}

\medskip

\noindent

{\bf Example HighResolution.}
In the previous two Examples {\bf MultiScalar} and {\bf ResolutionLimit}, we have found that the resolution can be improved with the increase of the wave number $k$.
We will further verify this fact in this example. The second contribution of this example is to show that the resolution can also be improved by a higher power $\rho$.
We choose a kite shaped domain with Dirichlet boundary condition as an unknown object.
We set three different wave numbers $k=5,\,10$ and $15$ and three different powers $\rho=1,\,2$ and $\rho=8$. The corresponding observation and incident direction number $N=16*k$.
The research domain is $[-4,4]\times[-4,4]$ with $151\times 151$ equally spaced sampling points.
Figure \ref{highresolution} shows the corresponding reconstructions.
Obviously, the shadows are greatly reduced with the increase of the power $\rho$.
To our surprise, it seems that the indicator always takes its maximum on the boundary of the scatterer, which results in a shaper reconstruction of the boundary of the scatterer.
However, there is no general theory for this fact.
\\

\begin{figure}[htbp]
  \centering
  \subfigure[\textbf{$k=5,\,\rho=1$}]{
    \includegraphics[width=2in]{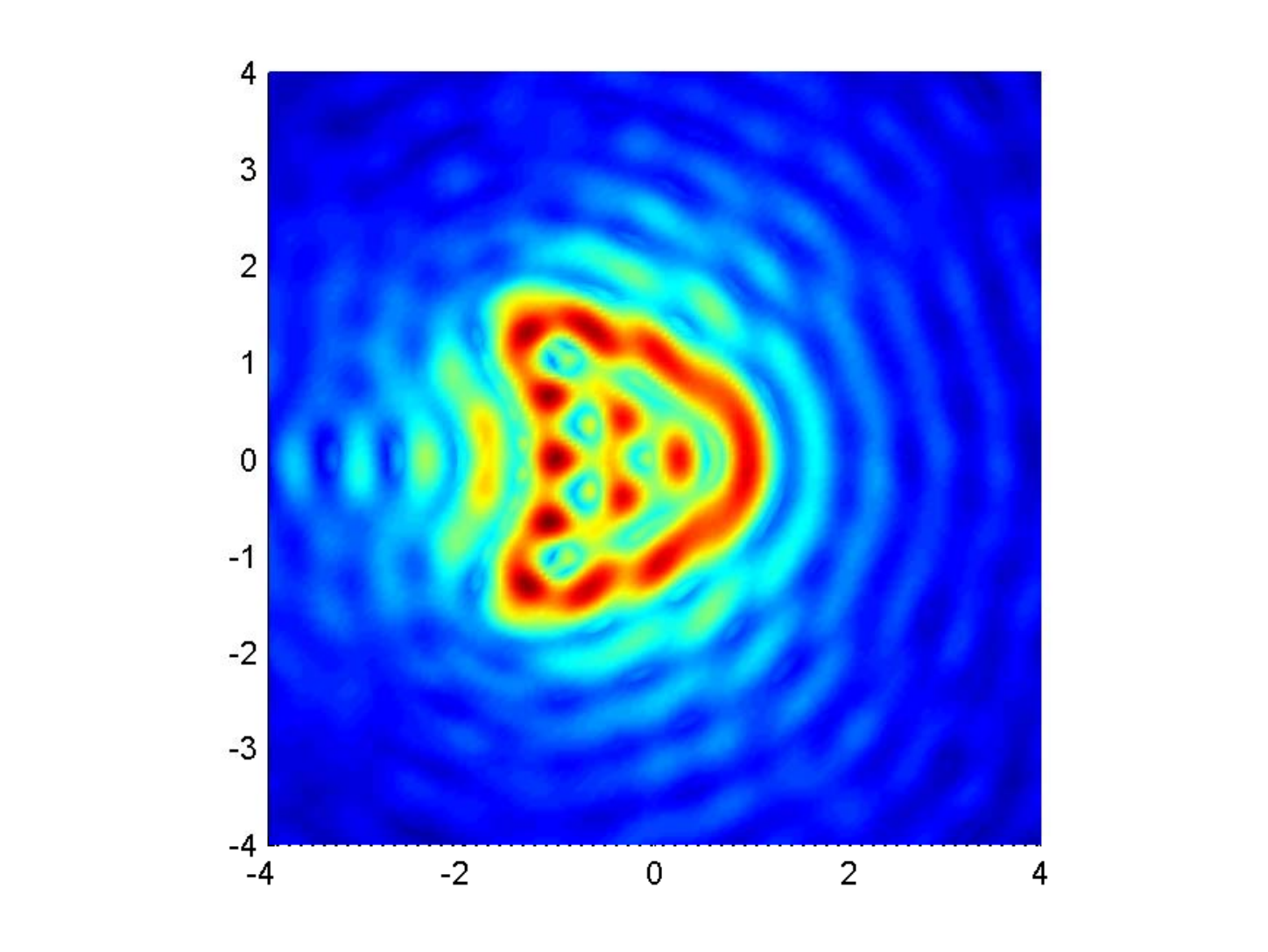}}
  \subfigure[\textbf{$k=5,\,\rho=2$}]{
    \includegraphics[width=2in]{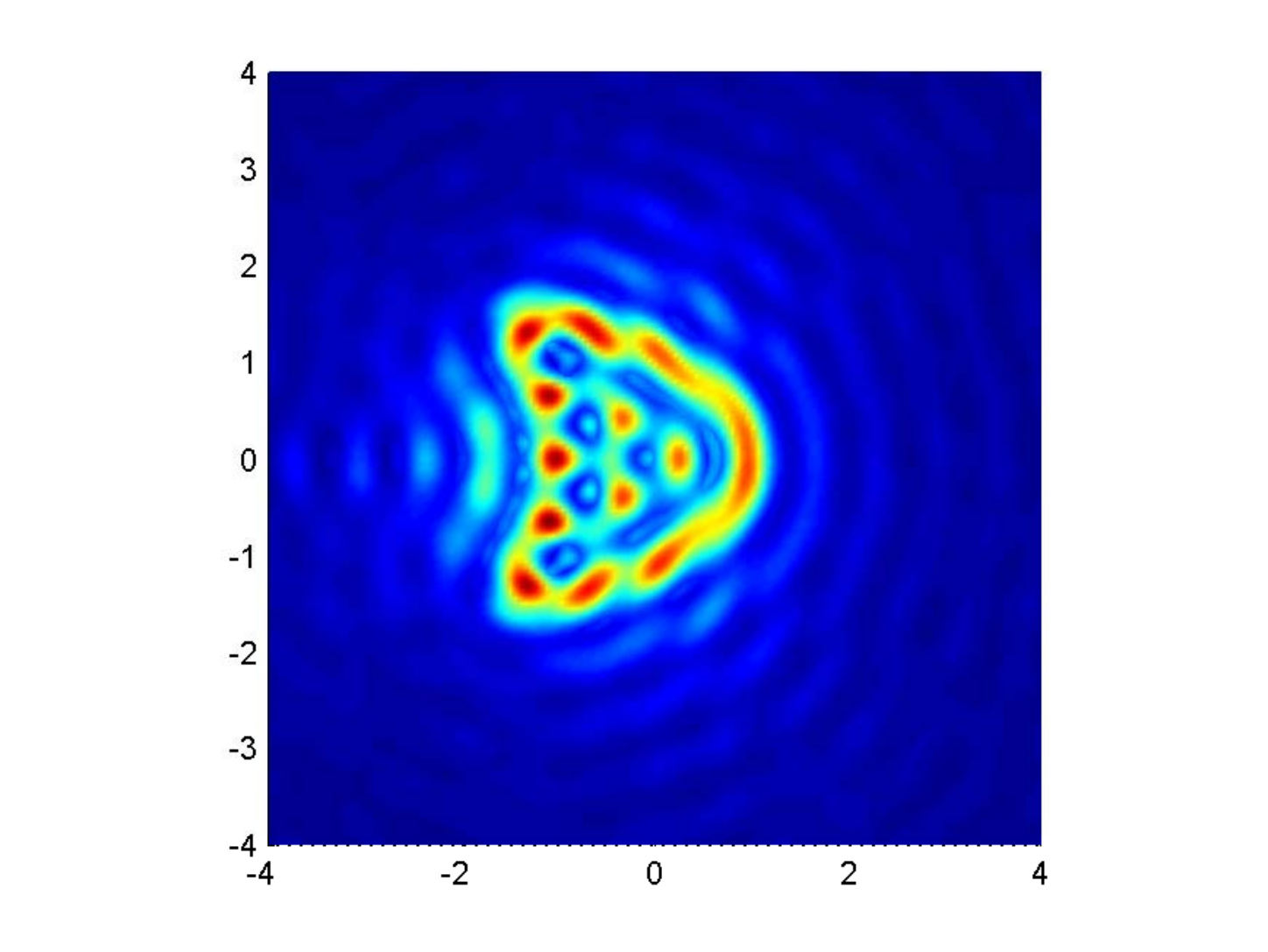}}
  \subfigure[\textbf{$k=5,\,\rho=8$}]{
    \includegraphics[width=2in]{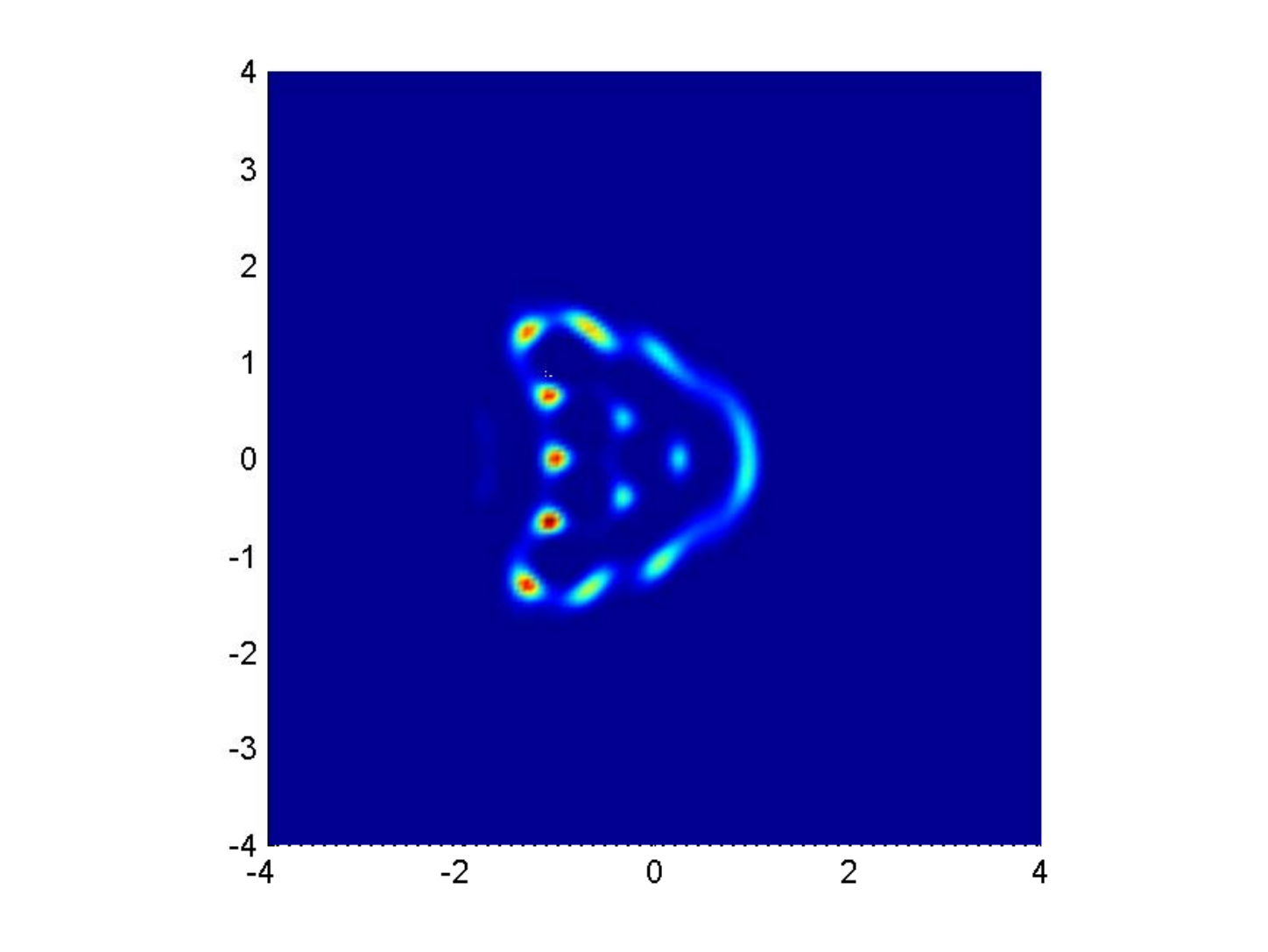}}
  \subfigure[\textbf{$k=10,\,\rho=1$}]{
    \includegraphics[width=2in]{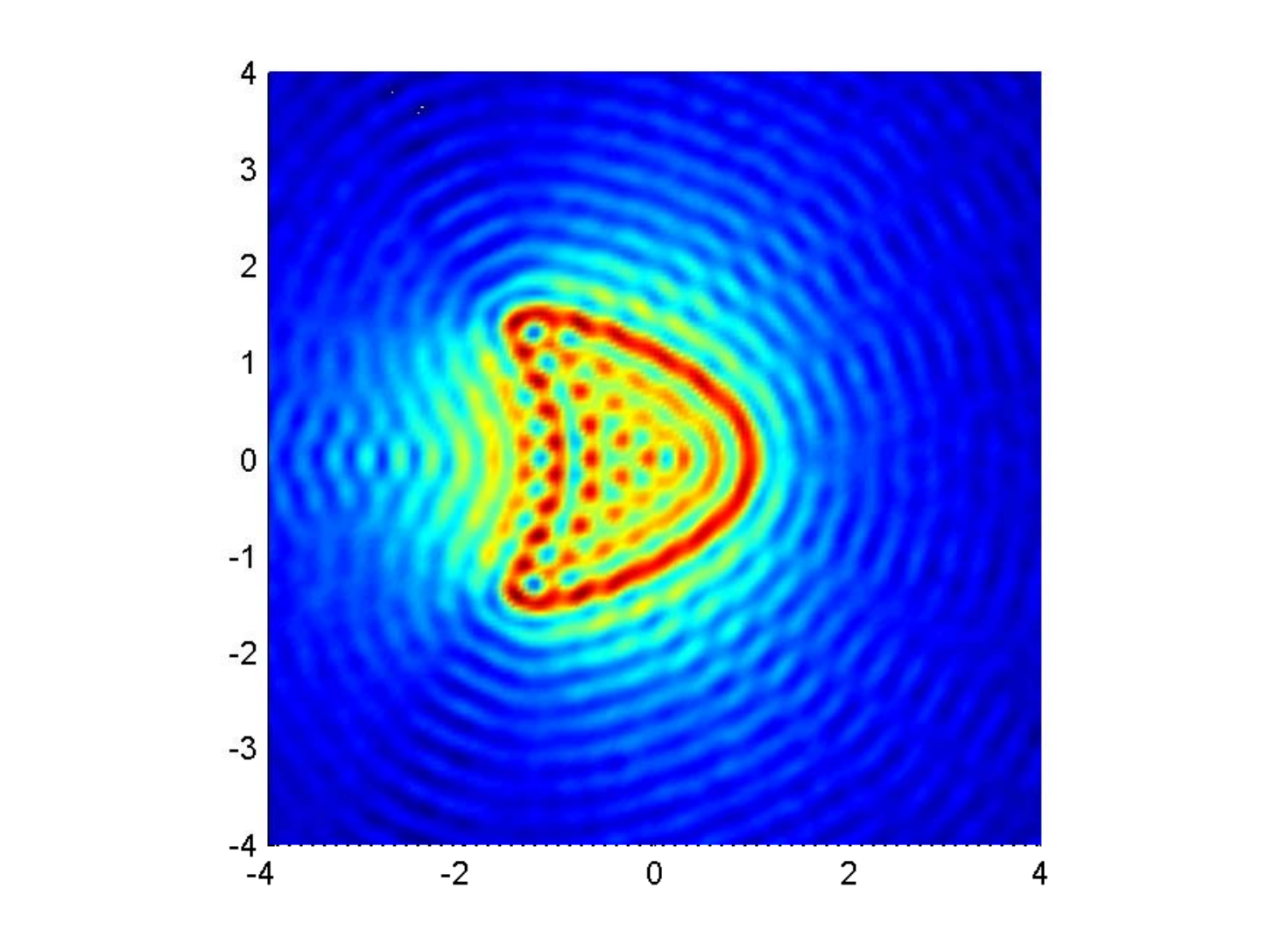}}
  \subfigure[\textbf{$k=10,\,\rho=2$}]{
    \includegraphics[width=2in]{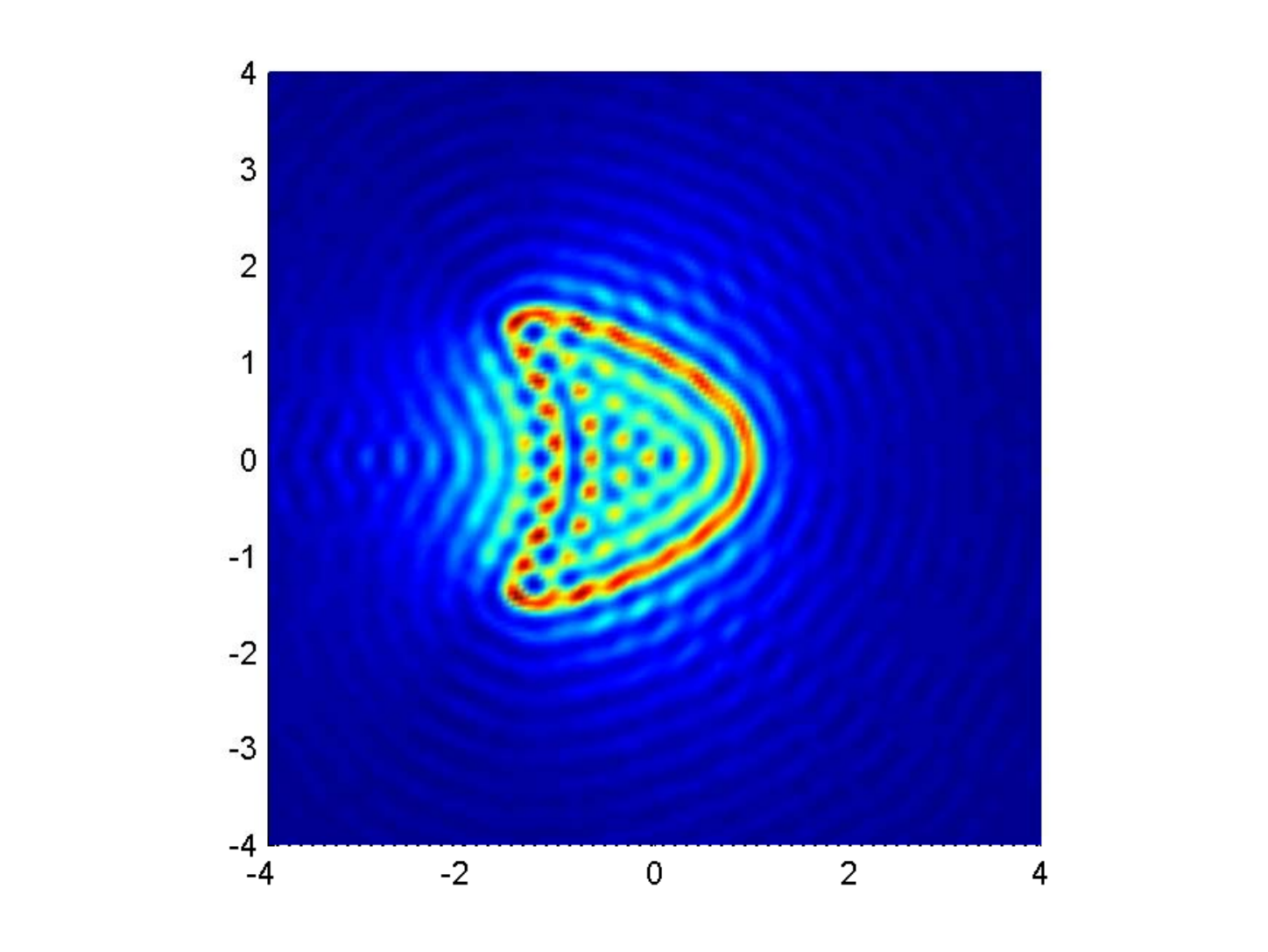}}
  \subfigure[\textbf{$k=10,\,\rho=8$}]{
    \includegraphics[width=2in]{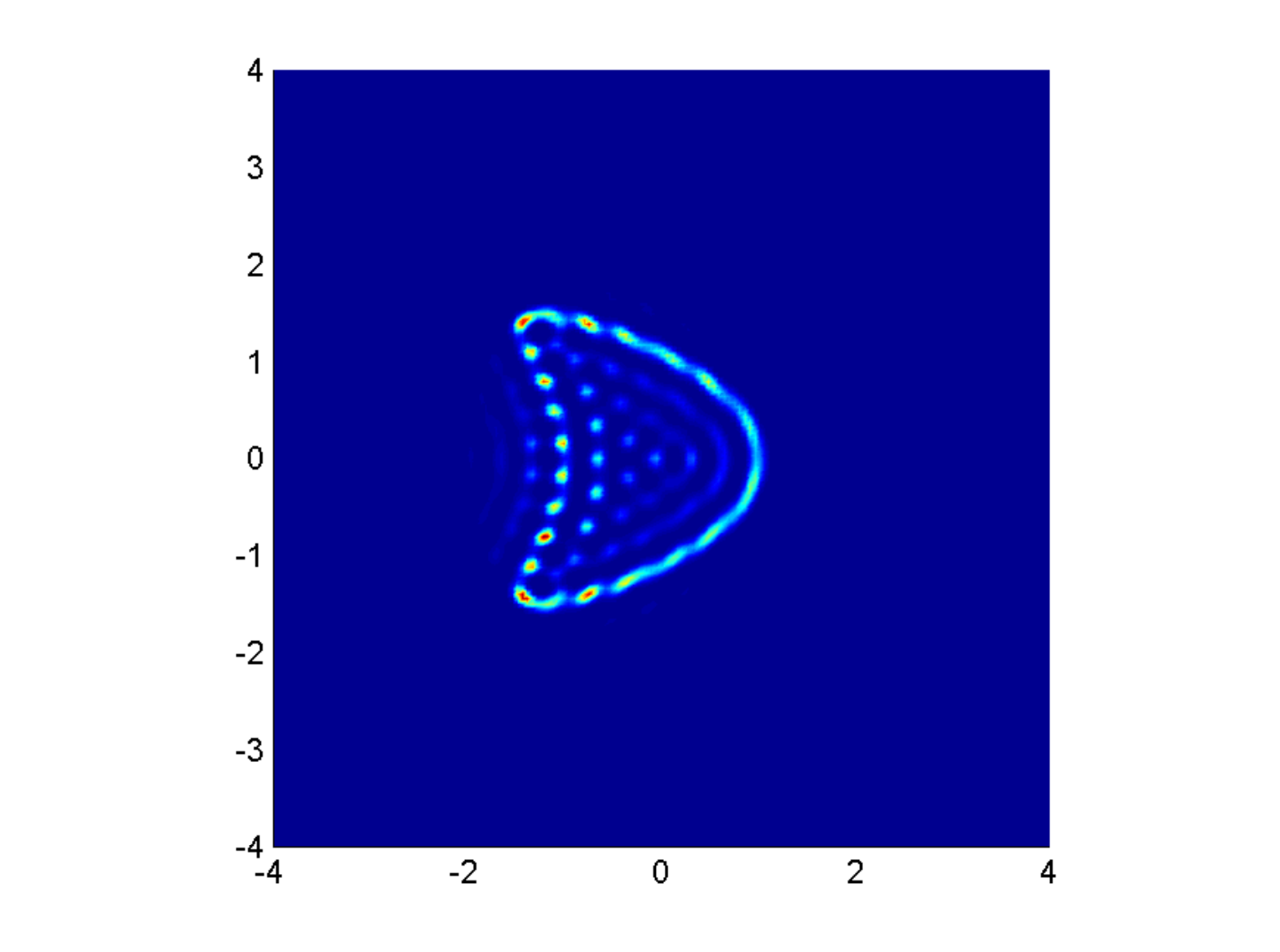}}
  \subfigure[\textbf{$k=15,\,\rho=1$}]{
    \includegraphics[width=2in]{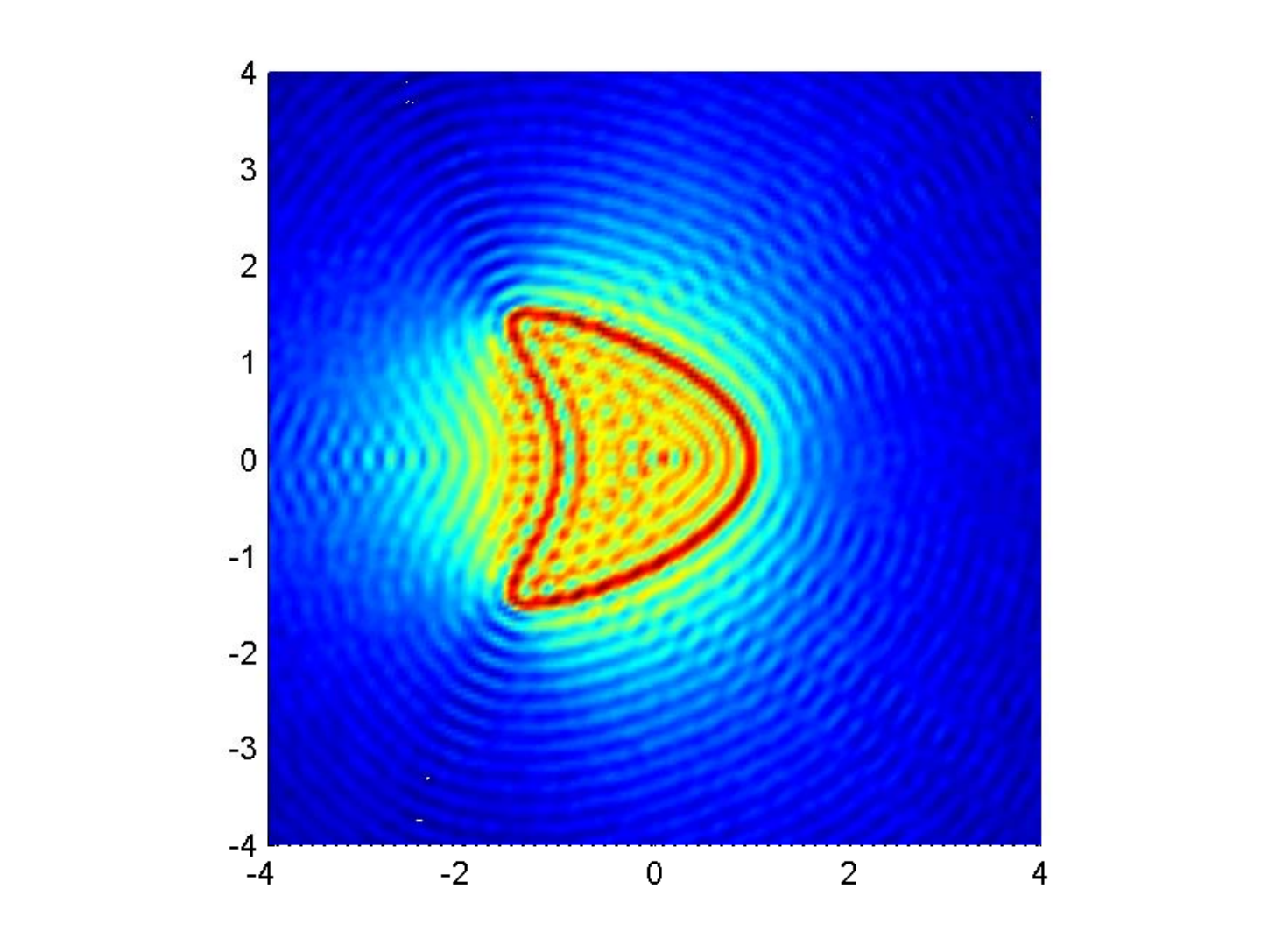}}
  \subfigure[\textbf{$k=15,\,\rho=2$}]{
    \includegraphics[width=2in]{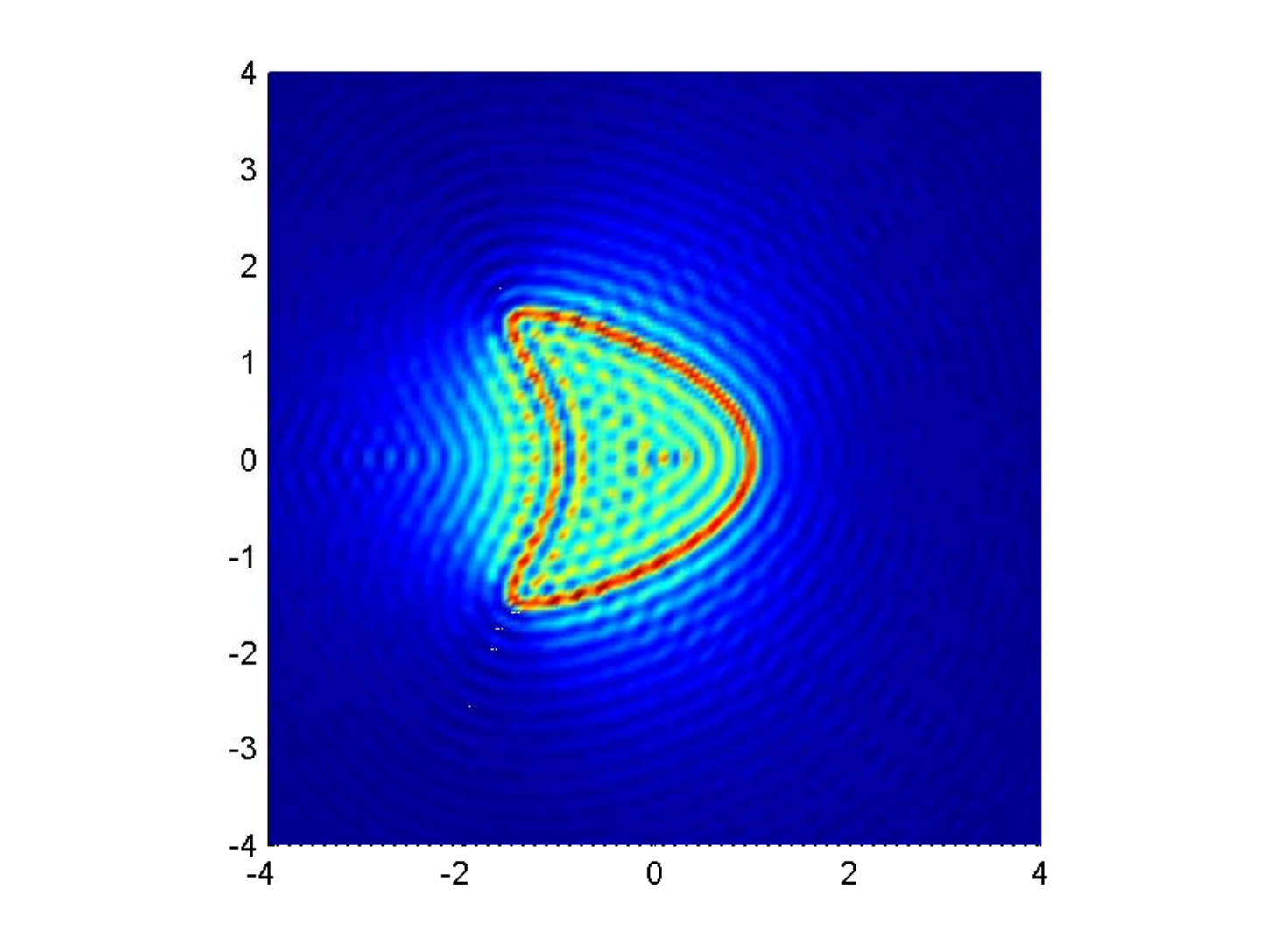}}
  \subfigure[\textbf{$k=15,\,\rho=8$}]{
    \includegraphics[width=2in]{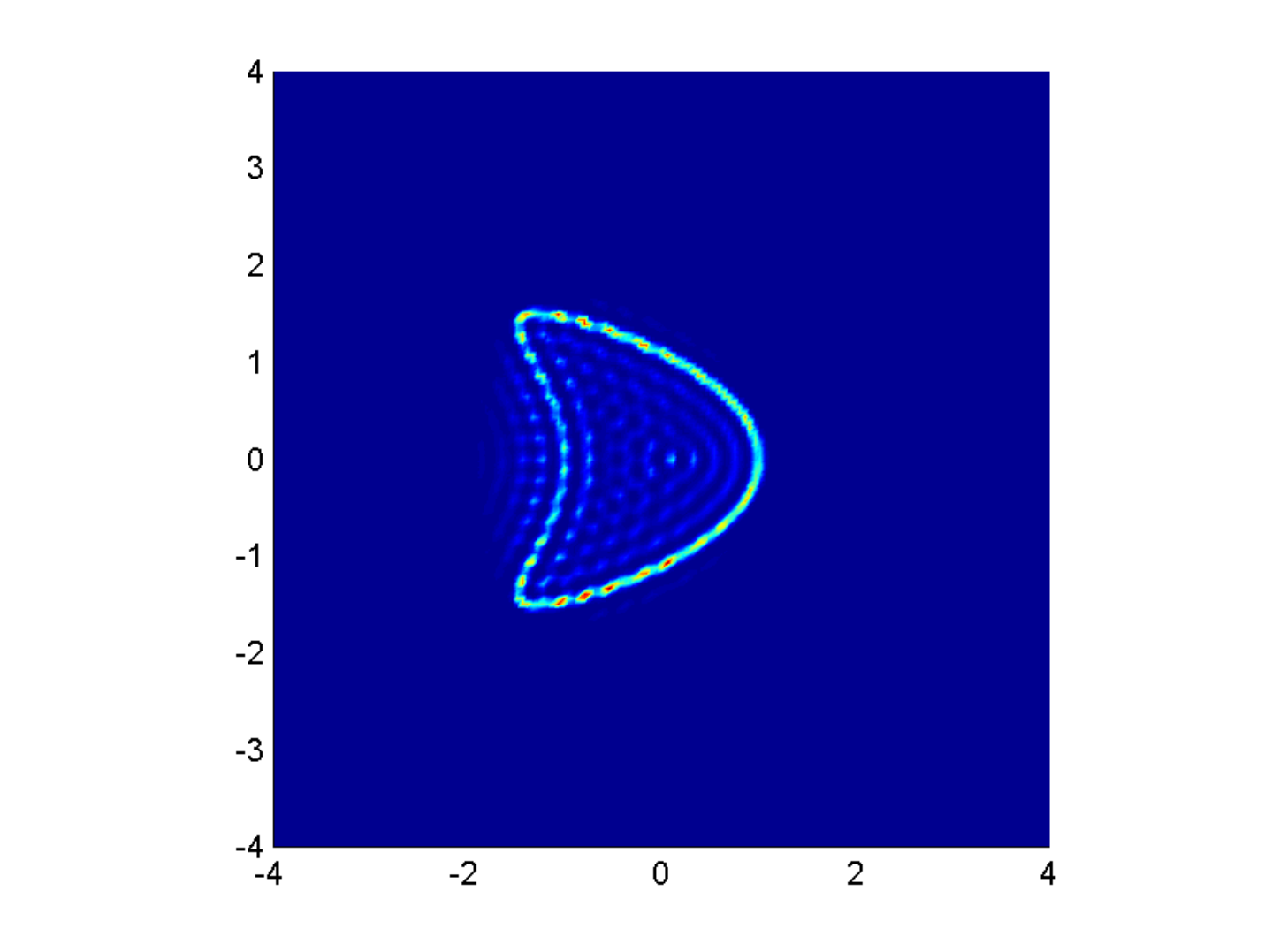}}
\caption{{\bf Example HighResolution.}\, Reconstruction by using different indicators $I^{\rho}_{new}$ with different wave numbers $k$ and different powers $\rho$. $30\%$ noise is
added to the far field data.}
\label{highresolution}
\end{figure}

In all the above examples, we observe that the reconstructions are rather satisfactory, with the consideration of the severe ill-posedness of the inverse scattering problems and the
fact that at least $30\%$ noise is added in the measurements (scattering amplitudes).

\section{Concluding remarks}
\label{sec4}
\setcounter{equation}{0}

In this paper we propose a new sampling method for shape identification in inverse acoustic scattering problems.
Both the theory foundation and numerical simulations are presented.
Only matrix multiplications are involved in the computation, thus
our method is very fast and robust against measurement noise from the numerical point of view.
The recovering scheme works independently of the physical properties of the underlying scatterers.
There might be several components with different physical properties, or with different scalar sizes, presented simultaneously.
Our method also allows us to distinguish two components of distance about one half of the wavelength, which is known to be challenging for numerical reconstruction.

However, the theory foundation is still less developed than the classical sampling method, e.g., the factorization method \cite{KirschGrinberg}. In our numerical simulations,
we have observed that the indicator always takes its maximum on the boundary of the scatterer. However, there are still no theory analysis on this fact.
Similar techniques can also be applied to inverse scattering of elastic waves or electromagnetic waves, which shall be addressed in a forthcoming work.

\section*{Acknowledgements}

The research of X. Liu was supported by the NNSF of China under grants 11571355, 61379093 and 91430102.

\end{document}